\documentclass[a4paper, reqno, 12pt]{amsart}

\usepackage{tikz, tikz-3dplot}
\usepackage{amsmath,amsthm,amssymb,graphics}
\usepackage[utf8]{inputenc}
\usepackage[english]{babel}
\usepackage[colorlinks=true, allcolors=blue]{hyperref}
\usepackage{slashbox}
\usepackage[title]{appendix}
\usepackage{multirow}
  \usepackage{setspace}

\numberwithin{equation}{section}

\theoremstyle{plain}

\newtheorem{te}{Theorem}[section]
\newtheorem{coro}[te]{Corollary}
\newtheorem{prop}[te]{Proposition}
\newtheorem{defn}[te]{Definition}
\newtheorem{lem}[te]{Lemma}

\newtheorem*{ack*}{Acknowledgment}
\theoremstyle{remark}

\newtheorem{rmk}{Remark}

\newcommand{\dsum}{\displaystyle\sum}
\newcommand{\dint}{\displaystyle\int}

\newcommand{\dprod}{\displaystyle\prod}\newcommand{\RomanNumeralCaps}[1]
    {\MakeUppercase{\romannumeral #1}}
\newcommand{\nocontentsline}[3]{}

\def\x{{\boldsymbol x}}
\def\y{{\bf y}}
\def\z{{\bf z}}
\def\w{{\bf w}}
\def\u{{\bf u}}

\def\l{{\boldsymbol l}}
\def\n{{\boldsymbol n}}
\def\b{{\boldsymbol b}}

\def\g{{\boldsymbol g}}
\def\0{{\bf 0}}
\def\a{{\boldsymbol a}}
\def\b{{\boldsymbol b}}
\def\d{{\boldsymbol d}}
\def\c{{\boldsymbol c}}
\def\m{{\boldsymbol m}}
\def\h{{\boldsymbol h}}
\def\y{{\boldsymbol y}}
\def\z{{\boldsymbol z}}
\def\f{{\boldsymbol f}}
\def\r{{\boldsymbol r}}

\def\R{{\mathbb R}}

\def\N{{\mathbb N}}

\def\S{{\mathbb S}}
\def\Z{{\mathbb Z}}
\def\P{{\mathbb P}}

\begin{document}

\author{Kiseok Yeon}
\email{kyeon@purdue.edu}
\address{Department of Mathematics, Purdue University, 150 N. University Street, West Lafayette, IN 47907-2067, USA}

\title[Kiseok Yeon]{The Hasse principle for homogeneous polynomials with random coefficients over thin sets}
\maketitle

\begin{abstract}
In this paper, we investigate the solubility of homogeneous polynomial equations. The work of Browning, Le Boudec, Sawin [$\ref{ref3}$] shows that almost all homogeneous equations of degree $d\geq 4$ in $d+1$ or more variables satisfy the Hasse principle, and in particular that a positive portion possess a non-trivial integral solution. Our main result, when combined with our sequel joint work [$\ref{ref2000}$] with H.Lee and S.Lee, shows that such a conclusion remains true even when the coefficients of homogeneous polynomials are constrained by a polynomial condition under a modest condition on the number of variables. 

To state the precise result, let $d$ and $n$ be natural numbers. Let $\nu_{d,n}: \R^n\rightarrow \R^N$ denote the Veronese embedding with $N=\binom{n+d-1}{d}$, defined by listing all the monomials of degree $d$ in $n$ variables using the lexicographical ordering. Let $\langle \a, \nu_{d,n}(\x)\rangle\in \Z[\x]$ be a homogeneous polynomial in $n$ variables of degree $d$ with integer coefficients $\a$, where $\langle\cdot,\cdot\rangle$ denotes the inner product. For a non-singular form $P\in \Z[\x]$ in $N$ variables of degree $k\geq 2,$ consider a set of integer vectors $\a\in \Z^N$, defined by $$\mathfrak{A}(A;P)=\{\a\in \Z^N|\ P(\a)=0,\ \|\a\|_{\infty}\leq A\}.$$
 We confirm that when $d\geq 4$, $n$ is sufficiently large in terms of $d$, and $k\leq d,$ the proportion of integer vectors $\a\in \Z^N$ in $\mathfrak{A}(A;P)$, whose associated equations $\langle \a, \nu_{d,n}(\x)\rangle=0$ satisfy the Hasse principle, converges to $1$ as $A\rightarrow \infty$. We make explicit a lower bound on $n$ guaranteeing this conclusion. In particular, we show that when $d\geq 14$ it suffices to take $n\geq 32d+17$. Our main tool is the Hardy-Littlewood circle method. 
    
 
\end{abstract}

\tableofcontents

\section{Introduction and statement of the results}\label{sec1}

In 1962, thanks to the celebrated work of Birch [$\ref{ref8}$], we can say that a system of $R$ forms $F_1,\ldots,F_R$ with integer coefficients of degree $d$ in $n$ variables, satisfies the \emph{smooth} Hasse principle whenever $n$ is sufficiently large in terms of $d$, $R$ and the dimension of a certain singular locus. In particular, it suffices to take \begin{equation}\label{birch}
    n>R(R+1)(d-1)2^{d-1}+\text{dim}\ W,
\end{equation}
where $$W=\{\x\in \mathbb{A}^n|\ \text{rank}(J(\x))<R\},$$
in which $J(\x)$ is the Jacobian matrix of size $R\times n$ formed from the gradient vectors $\nabla F_1(\x),\ldots,\nabla F_R(\x).$
Later on, when the variety defined by $F_1,\ldots, F_R$ is smooth and $d=2,3$, Rydin Myerson ([$\ref{ref26}$,$\ref{ref27}$])  showed that the factor $R(R+1)(d-1)2^{d-1}+\text{dim} W$ in the bound ($\ref{birch}$) can be replaced by a factor growing linearly in $R$, in particular, $d2^dR+R$. This refined bound also applies to $d\geq 4$ for a generic system of equations $F_1=\cdots=F_R=0$ (see the work of Rydin Myerson [$\ref{ref28}$]). 
Furthermore, Br\"udern and Wooley [$\ref{ref13}$] showed that whenever $n\geq13$, a system of two diagonal cubic forms satisfies the Hasse principle. This result was generalized to that for systems of diagonal cubic forms with a certain non-singularity condition on the coefficients of the system [$\ref{ref12}$]. Beyond a system of equations with the same degrees, Browning and Heath-Brown $[\ref{ref9}]$ verified the Hasse principle for a system of forms with different degrees.

For the case $R=1$, in particular, the Hasse–Minkowski theorem shows that the Hasse principle holds for a quadratic form. Furthermore, for $d=3$ and $R=1,$ Heath-Brown [$\ref{ref25}$] proved that non-singular cubic forms with $10$ variables have a nontrivial integer solution, and later on, a series of Hooley's works $([\ref{ref18}],[\ref{ref22}],[\ref{ref23}],[\ref{ref24}])$ verified the Hasse principle for non-singular cubic forms in $9$ variables and cubic forms in $9$ variables allowed a certain singularity condition. For $d=4$ and $R=1,$ due to Marmon and Vishe [$\ref{ref29}$], we can say that a non-singular quartic form satisfies the Hasse principle whenever $n$ is at least $28.$ In general, for $d\geq 3$ and $R=1,$ Browning and Prendiville [$\ref{ref10}$] showed that a non-singular form satisfies the Hasse principle whenever $n$ is at least $\left(d-\frac{1}{2}\sqrt{d}\right)2^d$.

It is known that the Brauer–Manin obstruction to the Hasse principle is empty for smooth Fano varieties of dimension at least 3 over any number field.
Moreover, for such varieties, it follows from a conjecture of Colliot-Th\'el\`ene that the Brauer–Manin obstruction to the Hasse principle should be the only one. Therefore, smooth Fano varieties are expected to satisfy the Hasse principle. For a smooth variety $V$ defined by a form in $n$ variables of degree $d$ $(R=1$ in our discussion), the number of variables $n$ is greater or equal to $d+1$ if and only if $V$ is Fano, and thus a non-singular form, with $n\geq d+1$ with $d\geq 4$ and $R=1$, is expected to satisfy the Hasse principle. We notice here that the current records described in the previous paragraph are very far from this expectation. 

In $2003$, Poonen and Voloch [$\ref{ref14}$] suggested a probabilistic point of view associated with this expectation. To be specific, let $n$ and $d$ be natural numbers with $d\geq 2. $ Throughout this paper, we let \begin{equation}\label{1.21.2}
    N:=N_{d,n}=\binom{n+d-1}{d}.
\end{equation}
Let $\nu_{d,n}: \R^n\rightarrow \R^N$ denote the Veronese embedding, defined by listing all the monomials of degree $d$ in $n$ variables using the lexicographical ordering. We denote a homogeneous polynomial in $n$ variables of degree $d$ with integer coefficients, by $\langle\a,\nu_{d,n}(\x)\rangle$ with $\a\in \Z^N$. For here and throughout, we write $f_{\a}(\x)=\langle\a,\nu_{d,n}(\x)\rangle$ for simplicity. If we consider a set of integer vectors $\a\in \Z^N$, defined by 
$\mathfrak{A}(A)=\{\a\in \Z^N|\ \|\a\|_{\infty}\leq A\},$
Poonen and Voloch [$\ref{ref14}$] conjectured that the proportion of integer vectors $\a\in \Z^N$ in $\mathfrak{A}(A)$, whose associated equations $f_{\a}(\x)=0$ satisfy the Hasse principle, converges to $1$ as $A\rightarrow \infty$, provided that $n\geq d+1$ and $d\geq3$.  Recently, Browning, Le Boudec, and Sawin [$\ref{ref3}$] confirmed this conjecture of Poonen and Voloch except for the case $d=3$ and $n=4.$ 

To bridge between this probabilistic point of view and the original expectation with a deterministic nature mentioned in the third paragraph, it is natural to investigate the distribution of integer vectors  $\a\in \Z^N$ in $\mathfrak{A}(A)$ whose associated equations $f_{\a}(\x)=0$ satisfy the Hasse principle. However, the conjecture of Poonen and Voloch [$\ref{ref14}$] and the conclusion of the work due to Browning, Le Boudec, and Sawin [$\ref{ref3}$] seem incapable of describing the distribution of those integer vectors $\a\in \Z^N$ in $\mathfrak{A}(A)$ beyond the bounds of their cardinality. Therefore, we turn our attention to the distribution of those integer vectors $\a\in \Z^N$ restricted to `thin sets'. In $2014,$ Br\"udern and Dietmann [$\ref{ref2}$] showed that the proportion of integer vectors $\a\in \Z^n$ in $ [-A,A]^n\cap\Z^n$ associated with diagonal forms satisfying the Hasse principle, converges to $1$ as $A\rightarrow\infty$, provided that $n\geq 3d+2$. 

In order to describe our main interest, for $\a\in \R^N,$ we fix a non-singular form $P(\boldsymbol{a})\in \Z[\x]$ in $N$ variables of degree $k\geq 2$. For the convenience, we define 
 \begin{equation}
     \mathfrak{A}(A;P):=\{\a\in \Z^N|\ \|\a\|_{\infty}\leq A,\ P(\a)=0\}.
 \end{equation}
 Then, we may ask a question about how many equations $f_{\a}(\x)=0$, with $\a\in \mathfrak{A}(A;P)$, satisfy the Hasse principle. Furthermore, even when the coefficients $\boldsymbol{a}\in \Z^N$ are restricted to arbitrary algebraic sets, we may ask how many equations $f_{\a}(\x)=0$, associated with these coefficients $\boldsymbol{a}\in\Z^N$, do satisfy the Hasse principle. Even in the case $R>1$, we may ask analogous questions. With more answers to these questions, we have better information about the distribution of the coefficients $\boldsymbol{a}$ whose associated equations satisfy the Hasse principle. 

 In this paper, we develop a framework via the Hardy-Littlewood circle method, in order to answer the question described in the previous paragraph. We define a set $\mathcal{A}^{\text{loc}}_{d,n}(A;P)$ of integer vectors $\a\in \Z^N$ in $\mathfrak{A}(A;P)$ having the property that the associated equation  $f_{\a}(\x)=0$ is everywhere locally soluble. When $d\geq 14,$ $k\leq d$ and $n\geq 32d+17$, we verify that the proportion of integer vectors $\a\in\mathcal{A}^{\text{loc}}_{d,n}(A;P)$ in $\mathfrak{A}(A;P)$, having the property that the number of integer solutions in $[1,X]^n$ satisfying $f_{\a}(\x)=0$ is less than $A^{-1}X^{n-d}(\log A)^{-\tau}$ for some $\tau>0,$ converges to $0$ as $A\rightarrow \infty$ (see Theorem $\ref{thm1.3}$ together with Corollary $\ref{coro1.5}$ below). Meanwhile, for $P\in \Z[\x]$ satisfying that there exists $\b\in\Z^N$ with $P(\b)=0$ such that $f_{\b}(\x)=0$ has a smooth integer point, our sequel joint work [$\ref{ref2000}$] with H.Lee and S.Lee reveals that the proportion of integer vectors $\a\in \mathcal{A}^{\text{loc}}_{d,n}(A;P)$ in $\mathfrak{A}(A;P)$ converges to a positive number as $A\rightarrow \infty.$ Therefore, at least for these $P\in \Z[\x]$, our main theorem is indeed a non-trivial one. Furthermore, we can say that with $d,k,n$ and those  $P\in \Z[\x]$ as above, for almost all $\a\in \mathcal{A}^{\text{loc}}_{d,n}(A;P)$, the equation $f_{\a}(\x)=0$ has the expected number of integer solutions inferred by the circle method, allowed up to a factor $(\log A)^{-\tau}$. Therefore, this obviously implies that for $d,k,n$ as above, the proportion of integer vectors $\a\in \Z^N$ in $\mathfrak{A}(A;P)$, whose associated equations $f_{\a}(\x)=0$ satisfy the Hasse principle, converges to $1$ as $A\rightarrow \infty.$ 
 
In order to describe our main theorems, we temporarily pause here and provide some definitions. Recall that $f_{\a}(\x)$ is a homogeneous polynomial in $n$ variables of degree $d$.
Furthermore, for $\a\in \Z^N$ and $X>0$, we define
\begin{equation}
    \mathcal{I}_{\a}(X)=\left\{\x\in [1,X]^n\cap \Z^n|\ f_{\a}(\x)=0\right\}.
\end{equation}
We note here that our argument proceeds for fixed $X>0$, and thus for simplicity, we write
\begin{equation}\label{def2.2}
    w=\log X
\end{equation}
and \begin{equation}\label{def2.3}
    W=\dprod_{p\leq w}p^{\lfloor\log w/\log p\rfloor}.
\end{equation}
Observe here that an application of the prime number theorem reveals that $\log W \leq 2w,$ which implies 
\begin{equation}
    W\leq X^{2}.
\end{equation}

\bigskip

For $Q>0$ and $\a\in \Z^N,$ we define
\begin{equation}\label{def6.1}
    \sigma(\a;Q)=Q^{-(n-1)}\#\{\g\in [1,Q]^n|\ f_{\a}(\g)\equiv 0\ \text{mod}\ Q\}.
\end{equation}
We notice that by the Chinese remainder theorem one has
\begin{equation}\label{6.26.26.2}
    \sigma(\a;Q)=\dprod_{p^r\|Q}\sigma(\a;p^r).
\end{equation}
Then, on recalling the definition $(\ref{def2.3})$ of $W$, we write
\begin{equation}
    \mathfrak{S}_{\a}^*=\sigma(\a;W)=\dprod_{p^r\|W}\sigma(\a;p^r).
\end{equation}

Recall the definition $(\ref{def2.2})$ of $w.$ Put $\zeta=w^{-5},$ and we introduce an auxiliary function 
$$\mathfrak{w}_{\zeta}(\beta)=\zeta\cdot\left(\frac{\text{sin}(\pi \zeta\beta)}{\pi \zeta\beta}\right)^2.$$
This function has the Fourier transform
$$\widehat{\mathfrak{w}}_{\zeta}(\xi)=\int_{-\infty}^{\infty}\mathfrak{w}_{\zeta}(\beta)e(-\beta \xi)d\beta=\text{max}\{0,1-|\xi|/\zeta\}.$$
For $\a\in \Z^N$ and $A,X>0,$ we define 
\begin{equation}\label{defnJ*}
    \mathfrak{J}_{\a}^*:= \mathfrak{J}_{\a}^*(A,X)=A^{-1}X^{n-d}\dint_{[0,1]^n}\zeta^{-1} \widehat{\mathfrak{w}}_{\zeta}(A^{-1}f_{\a}(\boldsymbol{\gamma}))d\boldsymbol{\gamma}.
\end{equation}


As we will show in section $\ref{sec5}$ below, the functions $\mathfrak{S}_{\a}^*$ and $\mathfrak{J}_{\a}^*$ behave in a similar manner to the truncated singular series and the truncated singular integral, which traditionally appear in the main term of an asymptotic formula for $\mathcal{I}_{\a}(X)$ predicted by the Hardy-Littlewood circle method, on average over $\a\in \Z^N$ in $\mathfrak{A}(A;P).$

\begin{te}\label{thm2.2}
Let $A$ and $X$ be positive numbers and let $n$ and $d$ be natural numbers with $d\geq 4.$ Let $n_1$ be the greatest integer with $n_1\leq \lfloor(n-1)/2\rfloor/8.$  Suppose that $n_1>2d$ and $2X^d\leq A\leq X^{n_1-d}.$ Suppose that $P\in \Z[\x]$ is a non-singular form in $N_{d,n}$ variables of degree $k\geq 2.$ Then, whenever $N_{d,n}\geq 200k(k-1)2^{k-1},$ there is a positive number $\delta<1$ such that
\begin{equation*}
    \dsum_{\substack{\|\a\|_{\infty}\leq A\\P(\a)=0}}\left|\mathcal{I}_{\a}(X)-\mathfrak{S}_{\a}^*\mathfrak{J}_{\a}^*\right|^2\ll A^{N-k-2}X^{2n-2d}(\log A)^{-\delta}.
\end{equation*}
\end{te}

\begin{rmk}
    A slight modification of the argument at the end of Section 4 allows us to deal with the case $d= 3$. However, the resulting required number of variables may be disappointing compared to the known result on the Hasse principle for cubic forms.
\end{rmk}
\bigskip

Recall the definition $(\ref{1.21.2})$ of $N:=N_{d,n}$. In advance of the statement of the following theorem, we recall the definition of $\mathcal{A}^{\text{loc}}_{d,n}(A;P)$. 
\begin{te}\label{thm2.3}
Let $A$ and $X$ be positive numbers with $X^3\leq A.$ Suppose that $n$ and $d$ are natural numbers with $n>d+1$ and $d\geq 2$. Suppose that $P\in \Z[\x]$ is a non-singular form in $N_{d,n}$ variables of degree $k\geq 2.$ Then, whenever $N_{d,n}\geq 1000n^28^k$, one has
\begin{equation*}
    \#\left\{\a\in \mathcal{A}^{\text{loc}}_{d,n}(A;P)\middle|\ \begin{aligned}
        \mathfrak{S}^*_{\a}\mathfrak{J}_{\a}^*\leq X^{n-d}A^{-1}(\log A)^{-\eta}
    \end{aligned}\right\}\ll A^{N-k}\cdot(\log A)^{-\eta/(40n)},
\end{equation*}
for any $\eta>0.$
\end{te}

\bigskip

One infers that the next theorem obviously implies that the proportion of integer vectors $\a\in \Z^N$ in $\mathfrak{A}(A;P)$, whose associated equations $f_{\a}(\x)=0$ satisfy the Hasse principle, converges to $1.$ We prove this theorem by making use of Theorem $\ref{thm2.2}$ and $\ref{thm2.3}$.

\begin{te}\label{thm1.3}
Let $A$ and $X$ be positive numbers.  Suppose that  $A,X,n,d$ and $k$ satisfy the hypotheses in Theorem $\ref{thm2.2}$ and $\ref{thm2.3}$. Then, the proportion of integer vectors $\a\in \mathcal{A}_{d,n}^{\text{loc}}(A;P)$ in $\mathfrak{A}(A;P)$, having the property that $$ \mathcal{I}_{\a}(X)<A^{-1}X^{n-d}(\log A)^{-1/5},$$ converges to $0$ as $A\rightarrow \infty.$
\end{te}
\begin{proof}
 Let $\a\in \mathcal{A}_{d,n}^{\text{loc}}(A;P)$. 
 Suppose that for $\delta>0$ obtained in Theorem $\ref{thm2.2}$ and for $A,X>0$ with $2X^{d}\leq A\leq X^{n_1-d}$, one has
    \begin{equation}\label{2.11}
        \left|\mathcal{I}_{\a}(X)-\mathfrak{S}^*_{\a}\mathfrak{J}_{\a}^*\right|\leq A^{-1}X^{n-d}\left(\log A\right)^{-\delta/4},
    \end{equation}
    and 
    \begin{equation}\label{2.12}
        \mathfrak{S}^*_{\a}\mathfrak{J}_{\a}^*> A^{-1}X^{n-d}(\log A)^{-\delta/5}.
    \end{equation}
    For sufficiently large $A>0$, it follows from $(\ref{2.11})$ and $(\ref{2.12})$ that 
    \begin{equation}\label{2.13}
    \begin{aligned}
        \mathcal{I}_{\a}(X)&\geq  \mathfrak{S}^*_{\a}\mathfrak{J}_{\a}^*-A^{-1}X^{n-d}(\log A)^{-\delta/4}\geq A^{-1}X^{n-d}(\log A)^{-1/5}.
    \end{aligned}
    \end{equation}
From the argument leading from $(\ref{2.11})$ to $(\ref{2.13})$, one finds that it suffices to estimate  the proportion of integer vectors $\a\in\mathcal{A}_{d,n}^{\text{loc}}(A;P) $ in $\mathfrak{A}(A;P)$ having the property that the inequality ($\ref{2.11}$) or $(\ref{2.12})$ fails. Define temporarily $\mathfrak{E}_1(A;P)$ to be the set of tuples $\a\in \mathcal{A}_{d,n}^{\text{loc}}(A;P)$ for which $(\ref{2.11})$ fails, and define $\mathfrak{E}_2(A;P)$ to be the set of tuples $\a\in \mathcal{A}_{d,n}^{\text{loc}}(A;P)$ for which $(\ref{2.12})$ fails.

First, we find by applying Theorem $\ref{thm2.2}$ that the number of integer vectors $\a\in \Z^N$ in $\mathfrak{A}(A;P)$, such that $(\ref{2.11})$ fails, is bounded above by $A^{N-k}(\log A)^{-\delta/2}.$ Thus, we have
    \begin{equation}\label{2.15}
        \#\mathfrak{E}_1(A;P)\ll A^{N-k}(\log A)^{-\delta/2}.
    \end{equation}

    Next, one finds by applying Theorem $\ref{thm2.3}$ that
    \begin{equation}\label{2.16}
       \#\mathfrak{E}_2(A;P)\ll A^{N-k}(\log A)^{-\delta/(200n)}.
    \end{equation}
    Therefore, the number of integer vectors $\a\in \Z^N$  in $\mathfrak{A}(A;P)$, such that when $\a\in\mathcal{A}_{d,n}^{\text{loc}}(A;P) $, the inequality ($\ref{2.11}$) or $(\ref{2.12})$ fails, is $O(A^{N-k}(\log A)^{-\eta})$ for some $\eta>0.$ 
   Meanwhile, note that since we have $N\geq 1000n^28^k$ and $P$ is a non-singular polynomial in $N$ variables having the property that $P(\a)=0$ has a nontrivial integer solution, it follows from $[\ref{ref8}]$ that for sufficiently large $A>0$ one has
    \begin{equation}\label{2.222}
        \#\{\a\in [-A,A]^N|\ P(\a)=0\} \sim cA^{N-k},
    \end{equation}
    for some $c>0$ depending on $P.$ Hence, we conclude from $(\ref{2.15})$, $(\ref{2.16})$ and $(\ref{2.222})$ that  the proportion of integer vectors $\a\in \mathcal{A}_{d,n}^{\text{loc}}(A;P)$ in $\mathfrak{A}(A;P)$, having the property that  the inequality ($\ref{2.11}$) or $(\ref{2.12})$ fails, is $O((\log A)^{-\eta})$ for some $\eta>0.$ By letting $A\rightarrow \infty,$ this completes the proof of Theorem $\ref{thm1.3}.$
\end{proof}

\begin{coro}\label{coro1.5}
Under the same restriction on $A$ and $X,$
the conclusions of Theorem $\ref{thm2.2}$, $\ref{thm2.3}$ and $\ref{thm1.3}$ hold for $d\geq 14,$ $k\leq d$ and $n\geq 32d+17$ in place of the hypotheses on $n,d$ and $k$.
\end{coro}
    \begin{proof}
    It suffices to show that the conditions $d\geq 14,$ $k\leq d$ and $n\geq 32d+17$ imply the hypotheses on $n,d$ and $k$ in Theorem $\ref{thm2.2},$ $\ref{thm2.3}$ and $\ref{thm1.3}.$
   For $d\geq 14,$ a modicum of computation reveals that we have
    $$1000\cdot 8^{d}\leq \frac{33^{d-2}}{d^2}.$$
Then, we see that whenever $d\geq 14$ and $n\geq 32d+17$, we obtain
    $$1000\cdot 8^{d}\leq \frac{1}{d^2}\cdot \left(\frac{n+d-1}{d}\right)^{d-2}.$$
  Hence, it follows that whenever $k\leq d$ one has
    \begin{equation*}
        \begin{aligned}
            1000\cdot n^2\cdot 8^k\leq 1000\cdot 8^{d} \cdot(n+d-1)^2\leq \left(\frac{n+d-1}{d}\right)^d\leq \binom{n+d-1}{d}=N_{d,n}.
        \end{aligned}
    \end{equation*}
    Furthermore, it implies that $N_{d,n}\geq 200k(k-1)2^{k-1}.$
   Additionally, with the same notation in Theorem $\ref{thm2.2}$, we readily see that $n\geq 32d+17$ implies that $n_1>2d.$ 
    \end{proof}

   \begin{rmk}
     By our sequel joint work [$\ref{ref2000}$] with H.Lee and S.Lee, for $P\in \Z[\x]$ satisfying that there exists $\b\in \Z^N$ with $P(\b)=0$ such that $f_{\b}(\x)=0$ has a smooth integer point, the proportion of integer vectors $\a\in \mathcal{A}^{\text{loc}}_{d,n}(A;P)$ in $\mathfrak{A}(A;P)$ converges to a positive number as $A\rightarrow \infty.$ Hence, at least for those  $P\in \Z[\x],$ our main theorems are indeed non-trivial ones.
   \end{rmk}


\bigskip

We notice here that the required number of variables $n\geq 32d+17$ in Corollary $\ref{coro1.5}$ is more restrictive than that required in the work of Browning, Le Boudec, and Sawin $[\ref{ref3}]$. However, we rather note that the success of establishing Theorem $\ref{thm2.2}$, $\ref{thm2.3}$ and $\ref{thm1.3}$ illuminates the flexibility of the circle method in investigating the distribution of the set of $\a\in \Z^N$, whose associated equations $f_{\a}(\x)=0$ do satisfy the Hasse principle. We note that the condition $k\geq 2$ seems not to be removable in our argument. The case $k=1$ rather seems accessible to tools used in $[\ref{ref3}]$, which mainly make use of the geometry of numbers. 

We note that via the same argument described in this paper, one can readily achieve the same conclusion even if one replaces the condition $P(\a)=0$ with $P(\a)=m$ for any $m\in \Z$ in  $\mathbb{V}^P_{d,n}(A)$, under the assumption that $P(\a)=m$ has a nontrivial integer solution. Furthermore, in general, our hope is that techniques described here may be useful in exploring rational points in a variety in $\mathbb{P}^{m_1}\times\mathbb{P}^{m_2}$ with $m_1,m_2\in\Z,$ and confirming the Hasse principle for various classes of systems of homogeneous equations.

\bigskip

\addtocontents{toc}{\protect\setcounter{tocdepth}{0}}
\section*{Stucture of the paper and notation}
 In section $\ref{sec3}$, we record two auxiliary mean value estimates, which are Theorem $\ref{thm3.1}$ and Lemma $\ref{lem3.2}$. The former one is required in the proof of Proposition $\ref{prop6.1}$ that is used in proving Theorem $\ref{thm2.3}$, and the latter one is used in many places in this paper (proofs of Theorem $\ref{thm3.1}$, Proposition $\ref{pro4.3}$, Lemma $\ref{lem5.15.1}$ and Lemma $\ref{lemma5.2}$). Furthermore, in section $\ref{sec3}$, we record several previous results in order to make this paper self-contained. Utilizing minor arcs estimates from section $\ref{sec4}$, we prove Theorem $\ref{thm2.2}$ in section $\ref{sec5}$. By making use of major arcs estimates from section $\ref{sec6}$, we prove Theorem $\ref{thm2.3}$ in section $\ref{sec7}.$

In this paper, we use bold symbols to denote vectors, which we consider as row vectors. For a given vector $\boldsymbol{v}\in \R^N$, we write the $i$-th coordinate of $\boldsymbol{v}$ by $(\boldsymbol{v})_i$ or $v_i.$ 
We use $\langle\cdot,\cdot\rangle$ for the inner product, and use $\|\cdot\|$ for the Euclidean norm. We write $0\leq \x\leq X$ to abbreviate the condition $0\leq x_1,\ldots,x_s\leq X.$ We emphasize here that we preserve summation conditions until different conditions are specified. Additionally, for a prime $p$ and natural numbers $n$ and $h$, we use $p^h\| n$ when $p^h|n$ but $p^{h+1}\nmid n$. Throughout this paper, we use $\gg$ and $\ll$ to denote Vinogradov's well-known notation, and write $e(z)$ for $e^{2\pi iz}$. We use $A\asymp B$ when both $A\ll B$ and $A\gg B$ hold. We adopt the convention that whenever $\epsilon$ appears in a statement, then the statement holds for each $\epsilon>0$, with implicit constants depending on $\epsilon.$

\section*{acknowledgement}
The author acknowledges support from NSF grant DMS-2001549 under the supervision of Trevor Wooley. The author is grateful for support from Purdue University. The author would like to thank J\"org Br\"udern, Rainer Dietmann and Tim Browning for helpful discussions. The author also would like to thank Heejong Lee for helpful comments. Especially, the author would like to thank Trevor Wooley for his guidance on this research, and for his constant encouragement and much useful advice which has improved the exposition.

\bigskip

\addtocontents{toc}{\protect\setcounter{tocdepth}{2}}

\bigskip

\section{Preliminary manoeuvre}\label{sec3}
Throughout this paper, we write
\begin{equation}\label{3.13.13.1}
P(\a)=\dsum_{1\leq j_1,\ldots,j_{k}\leq N}p_{j_1\cdots j_k}a_{j_1}a_{j_2}\cdots a_{j_{k}},
\end{equation}
with $p_{j_1\cdots j_k}\in \Z$ symmetric in the suffixes $j_1,\ldots,j_k.$
In order to describe the argument in sections $\ref{sec3}$ and $\ref{sec4}$, it is convenient to define differencing operators $\Delta_1$ by
\begin{equation*}
    \Delta_1(P(\x);\h)=P(\x+\h)-P(\x),
\end{equation*}
and so we define $\Delta_j$ for $j\geq 2$ recursively by means of the relations
\begin{equation}\label{3.3}
    \Delta_j(P(\x);\h_1,\ldots,\h_{j})=\Delta_1(\Delta_{j-1}(P(\x);\h_1,\ldots,\h_{j-1});\h_j).
\end{equation}
Furthermore, we define
\begin{equation}\label{3.23.23.2}
\begin{aligned}
\psi_j&:=\psi_j(\x^{(1)},\x^{(2)},\ldots,\x^{(k-1)};P)\\
&=k!\dsum_{1\leq j_1,\ldots,j_{k-1}\leq N}p_{j_1\cdots j_{k-1}j}(\x^{(1)})_{j_1}\cdots(\x^{(k-1)})_{j_{k-1}}.
\end{aligned}
\end{equation}

\subsection{Previous results}\label{sec3.1}
In order to make this paper self-contained, we record several previous results in this section.

\bigskip
We denote $\|z\|=\min_{n\in \Z}|z-n|$, for the following lemma.
\begin{lem}\label{lem3.1}$([\ref{ref17}$, $Lemma$ 3.3$]$,\ $[\ref{ref8}$, \textrm{Lemma} 2.3$])$
Let $L_1,\ldots, L_n$ be linear forms, defined by
$$L_i=\gamma_{i1}u_1+\cdots+\gamma_{in}u_n,\ (1\leq i\leq n),$$
satisfying the symmetry condition $\gamma_{ij}=\gamma_{ji}.$ Let $a>1$ be real, and let $N(Z)$ denote the number of sets of integers $u_1,\ldots,u_{n}$ satisfying 
\begin{equation*}
\begin{aligned}
\vert u_i\vert<aZ\ (i=1,\ldots,n),\ \|L_1\|<a^{-1}Z,\ldots,\|L_n\|<a^{-1}Z.
\end{aligned}
\end{equation*}
Then, if $0<Z_1\leq Z_2\leq 1,$ we have
$$\frac{N(Z_2)}{N(Z_1)}\ll \left(\frac{Z_2}{Z_1}\right)^n.$$
\end{lem}

\bigskip

\bigskip

\begin{lem}$([\ref{ref2},\ $Lemma 4.3$])$
Let $\alpha\in \R$, $d\in \N$ and let $A,X\geq 1$. Suppose that $q\in \N$ and $a\in \Z$ are coprime with $\vert\alpha-a/q\vert\leq q^{-2}.$ Then, one has
$$\dsum_{-A\leq b\leq A}\biggl\vert\dsum_{1\leq x\leq X}e(\alpha bx^d)\biggr\vert^{2^{d-1}}\ll AX^{2^{d-1}}\bigl(q^{-1}+X^{-1}+q(AX^d)^{-1}\bigr)(AXq)^{\epsilon}.$$
\end{lem}
 Furthermore, this lemma in conjunction with H\"older's inequality yields that
\begin{equation}\label{3.4}
    \dsum_{-A\leq b\leq A}\biggl\vert\dsum_{1\leq x\leq X}e(\alpha bx^d)\biggr|^{2}\ll A^{1+\epsilon}X^{2+\epsilon}\bigl(q^{-1}+X^{-1}+q(AX^d)^{-1}\bigr)^{2^{2-d}}
\end{equation}

\bigskip

\bigskip

Let $U_t(A,B)$ denote the number of solutions of 
$$\dsum_{j=1}^t a_j(x_j^k-y_j^k)=0$$
in integers $a_j,x_j,y_j$ satisfying 
$0<\vert a_j\vert\leq A,\ \vert x_j\vert\leq B,\ \vert y_j\vert\leq B.$
\begin{lem}\label{lem3.23.2}$([\ref{ref2025}$, $\text{Theorem 2.5}])$
Let $t\geq 2.$ Then, for real numbers $A$ and $B$ with $A\geq 2B^k\geq 1$, one has
$$U_t(A,B)\ll (AB)^t+A^{t-1}B^{2t-k+\epsilon}.$$
\end{lem}
We note that even if one replaces the condition $0<\vert a_j\vert\leq A$ by $0\leq \vert a_j\vert\leq A$ in $U_t(A,B)$, one readily obtains
\begin{equation}\label{3.5}
U_t(A,B)\ll (AB)^t+A^{t-1}B^{2t-k+\epsilon},
\end{equation}
under the same conditions as imposed in Lemma $\ref{lem3.23.2}$.

\bigskip

\bigskip

For here and throughout this paper, we define
\begin{equation*}
    \mathcal{R}_m(Q)=\{\b\in (\Z/Q\Z)^m\vert\ \text{gcd}(Q,\b)=1\},
\end{equation*}
and define $v_{p^r}(\boldsymbol{v})$ with $\boldsymbol{v}\in (\Z/p^r\Z)^{n}$ as the largest integer $s\in \{0,\ldots,r\}$ such that we have $\boldsymbol{v}\equiv \0\ \text{mod}\ p^s.$ Furthermore, in order to describe Lemma $\ref{lem3.4}$, for $e\in \{0,\ldots,r\},$ we temporarily define the set $\mathcal{R}^{(e)}_{N}(p^r)$ to be the set of $\a\in \mathcal{R}_N(p^r)$ such that there exists $\x\in \mathcal{R}_{n}(p^r)$ having the property that $f_{\a}(\x)\equiv 0\ \text{mod}\ p^r$ and $v_{p^r}(\nabla f_{\a}(\x))=e.$
\begin{lem}\label{lem3.4}$([\ref{ref3},\text{Lemma 5.7}])$
Let $d\geq 2$ and $n\geq 4.$ Let also $p$ be a prime number and $r\geq 1.$ For $e\in \{0,\ldots,r\}$ and $\a\in \mathcal{R}^{(e)}_{N}(p^r),$ we have
\begin{equation*}
      \frac{1}{p^{rn-r}}\cdot \#\{\g\in \mathcal{R}_n(p^{r})\vert\ f_{\a}(\g)\equiv 0\ \text{mod}\ p^{r}\}\geq \frac{1}{p^{(e+1)(n-1)}}.
\end{equation*}
\end{lem}

\begin{lem}\label{lem3.5}$([\ref{ref3},\text{Proof of Lemma 5.6}])$
For $\g\in \Z^n$ with $p\nmid \g$ and for $e\in \{0,\ldots, r\}$, we have 
\begin{equation*}
    \#\left\{\a\in \mathcal{R}_N(p^r)\middle\vert\ \begin{aligned}
        (f_{\a}(\g),\nabla f_{\a}(\g))\equiv \boldsymbol{0}\ \text{mod}\ p^e
    \end{aligned}\right\}\leq p^{r(N-n)+(r-e)n}.
\end{equation*}
\end{lem}

\bigskip

For any $u>0$, we let
$B_m(u)=\{\y\in \R^m\vert\ \|\y\|\leq u\}$. For $A>0$
we set 
$$H_n(A)=B_n\left(1+\frac{1}{A}\right)\setminus B_n\left(1-\frac{1}{A}\right)$$
\begin{lem}$([\ref{ref3}, \text{Proof of Lemma 5.8}])$\label{lem3.6}
For fixed $\y\in \mathcal{H}_n(A)$ with $A>0,$ we have
    $$\#\{\a\in [-A,A]^N\cap \Z^N \vert\ \|\nabla f_{\a}(\y)\|\ll 1\}\ll A^{N-n}.$$
\end{lem}

\bigskip
\begin{lem}$([\ref{ref3},\text{Proof of Lemma 5.9}])$\label{lem3.7}
    When $\lambda=O(1)$ and $\x\in \{\x\in \R^{n}\vert\ \|\x\|=1\}$, we have
    $$\text{mes}\left(\left\{\a\in B_N(N)\middle\vert\ \begin{aligned}
    \vert f_{\a}(\x)\vert\leq \lambda^2,\ \|\nabla f_{\a}(\x)\|\leq 2\lambda
    \end{aligned}\right\}\right)\ll \lambda^{n+1}.$$
\end{lem}

\bigskip

\subsection{Auxiliary mean value estimates}

In this section, we provide a mean value estimate which plays a crucial role, in section 5.2, in verifying that $\mathfrak{S}_{\a}$ is rarely small over $\a$ with $\|\a\|_{\infty}\leq A$ and $P(\a)=0$. Since this mean value may be of independent interest, we record it in this separate section.

\begin{te}\label{thm3.1}
Let $P\in \Z[\x]$ be a non-singular form in $N$ variables of degree $k$. Suppose that $P(\x)=0$ has a nontrivial integer solution. Suppose that $W$ is an integer with $W\in [1,A^{2/3}]$ for sufficiently large $A$. Define $$\mathbf{N}(A)=\#\left\{ \x,\y\in [-A,A]^N\cap \Z^N\middle|\ P(\x)=P(\y)=0,\ \x \equiv\y\ \text{mod}\ W\right\}.$$ Then, whenever $N\geq 18k(k-1)4^{k+3}$, we have
\begin{equation}\label{5.1}
\mathbf{N}(A)\asymp A^{2N-2k}W^{1-N}.    
\end{equation}
\end{te}

\bigskip

Since $P(\x)$ is a homogeneous polynomial in $N$ variables of degree $k$, probabilistic arguments deliver the expectation that the number of integer solutions $\x$ and $\y$ satisfying $P(\x)=P(\y)=0$ is $O(A^{2N-2k})$, and that the additional congruence condition on $\x$ and $\y$, which is $\x\equiv \y$ (mod $W$), provides an extra factor $W^{1-N}.$ Therefore, we see that the right-hand side in ($\ref{5.1}$) has the expected magnitude up to a constant. 
Without the congruence condition on $\x$ and $\y$, the expected magnitude immediately follows by the work of Birch [$\ref{ref8}$] whenever $N> (k-1)2^k$ (see also [$\ref{ref10}$]). However, for $\mathbf{N}(A)$ stated in $(\ref{5.1})$, Birch's argument seems not capable of immediately deriving the expected magnitude. 

In this subsection, we introduce an argument that may be helpful in general when one deduces an expected magnitude for the number of integer solutions of a system of equations together with congruence conditions. 
For the upper bound of $\mathbf{N}(A)$, we first represent $\mathbf{N}(A)$ by mean values of exponential sums via orthogonality. Next, on observing that the condition $P(\x)=P(\y)=0$ obviously induces $P(\x)\equiv P(\y)\equiv0$ (mod $W$) and by taking this induced condition into account in mean values of exponential sums, we are capable of drawing both information about $P(\x)=P(\y)=0$ and $\x\equiv\y$ (mod $W$) via suitable applications of the Weyl differencing argument. Eventually, we obtain the expected upper bound for $\mathbf{N}(A),$ by applying classical estimates for major and minor arcs together with such information.
For the lower bound of $\mathbf{N}(A)$, we first obtain an upper bound for the quantity
 \begin{equation}\label{quantity}
\#\left\{\b\in [1,W]^N\cap \Z^N\middle|\ P(\b)\equiv0\ \text{mod}\ W\right\},     
 \end{equation}
 with $W\in \N.$ Next, by applying the Cauchy-Schwarz inequality together with the upper bound for $(\ref{quantity})$, we obtain the expected lower bound for $\mathbf{N}(A)$.

By making use of the same idea of this proof, we are also capable of dealing with the case $W=A^{\mu}$ with $2/3< \mu<1.$ However, the required number of variables $N$ becomes more restrictive, as $\mu\rightarrow 1.$ As we will see in the proof of Proposition $\ref{prop6.1}$, the bound for $\mathbf{N}(A)$ with $1\leq W\leq A^{2/3}$ is enough for our purpose, and thus we record this theorem only with $\mu\leq 2/3.$

\bigskip

 For the proof of Theorem $\ref{thm3.1}$, we require two auxiliary lemmas. The first lemma of these gives the upper bound for 
 $$\#\left\{\b\in [1,W]^N\cap \Z^N\middle|\ P(\b)\equiv0\ \text{mod}\ W\right\},$$
 with $W\in \N.$
 \begin{lem}\label{lem3.93.9}
    Let $P\in \Z[\x]$ be a non-singular form in $N$ variables of degree $k.$ Suppose that $W$ is a natural number, and that $\mathfrak{p}$ is the smallest prime divisor of $W.$ Then, whenever $N> (k-1)2^{k+1},$ one has
$$\#\left\{\b\in [1,W]^N\cap \Z^N\middle|\ P(\b)\equiv0\ \text{mod}\ W\right\}=W^{N-1}+O\left(W^{N-1}\cdot\mathfrak{p}^{-N/(2^k(k-1))}\right).$$
 \end{lem}
\begin{proof}
    For simplicity, we write
    \begin{equation}\label{37}
     N_1(W)=\#\left\{\b\in [1,W]^N\cap \Z^N\middle|\ P(\b)\equiv0\ \text{mod}\ W\right\}.   
    \end{equation}
    By applying orthogonality, we have
\begin{equation*}
    N_1(W)=W^{-1}\dsum_{1\leq \b\leq W}\dsum_{1\leq r\leq W} e\left(\frac{P(\b)r}{W}\right).
\end{equation*}
By splitting the sum over $r$ in terms of values of $r_0=W/(W,r)$, we see that
\begin{equation}\label{3.73.73.7}
    N_1(W)=W^{N-1}+W^{-1}\dsum_{\substack{r_0|W\\ r_0\neq 1}}\dsum_{\substack{1\leq r\leq W\\(W,r)=W/r_0}}\dsum_{1\leq \b\leq W}e\left(\frac{P(\b)r}{W}\right).
\end{equation}

For fixed $r_0$ and $r$ with $r_0|W$ and $r_0=W/(W,r)$, we write $ \widetilde{r}=r/(W,r),$ and so $(\widetilde{r},r_0)=1$. Then, since $P$ is a non-singular form, it follows by the Weyl type estimate for exponential sums over minor arcs [$\ref{ref8}$, Lemma 5.4] that
\begin{equation}\label{3.103.10}
    \begin{aligned}
        \dsum_{\substack{r_0|W\\ r_0\neq 1}}\dsum_{\substack{1\leq r\leq W\\(W,r)=W/r_0}}\dsum_{1\leq \b\leq W}e\left(\frac{P(\b)r}{W}\right)&\leq \dsum_{\substack{r_0|W\\ r_0\neq 1}}\dsum_{\substack{1\leq \widetilde{r}\leq r_0\\(\widetilde{r},r_0)=1}}\left|\dsum_{1\leq \b\leq W}e\left(\frac{P(\b)\widetilde{r}}{r_0}\right)\right|\\
        &\leq \dsum_{\substack{r_0|W\\ r_0\neq 1}}\dsum_{\substack{1\leq \widetilde{r}\leq r_0\\(\widetilde{r},r_0)=1}}(W/r_0)^N\left|\dsum_{1\leq \b\leq r_0}e\left(\frac{P(\b)\widetilde{r}}{r_0}\right)\right|\\
        &\ll\dsum_{\substack{r_0|W\\ r_0\neq 1}}\dsum_{\substack{1\leq \widetilde{r}\leq r_0\\(\widetilde{r},r_0)=1}}W^N\cdot r_0^{-N/(2^{k-1}(k-1))+\epsilon}.
    \end{aligned}
\end{equation}
Recall the definition of $\mathfrak{p}$ in the statement of Lemma $\ref{lem3.93.9}$. Then, on noting that for all $\sigma>1$ one has
$$\dsum_{\substack{r_0|W\\r_0\neq 1}}r_0^{-\sigma}\leq \dsum_{\substack{\mathfrak{p}\leq r_0}}r_0^{-\sigma}\ll \mathfrak{p}^{1-\sigma},$$
it follows from ($\ref{3.103.10}$) and the hypothesis $N> 2^{k+1}(k-1)$ that 
\begin{equation*}
\begin{aligned}
     \dsum_{\substack{r_0|W\\ r_0\neq 1}}\dsum_{(W,r)=W/r_0}\dsum_{1\leq \b\leq W}e\left(\frac{P(\b)r}{W}\right)&\ll \dsum_{\substack{r_0|W\\ r_0\neq 1}}W^N\cdot r_0^{1-N/(2^{k-1}(k-1))+\epsilon}\\   
     &\ll W^N\cdot \mathfrak{p}^{2-N/(2^{k-1}(k-1))+\epsilon}\\
     &\leq W^N \cdot \mathfrak{p}^{-N/(2^k(k-1))}.
\end{aligned}
\end{equation*}

On substituting this bound into $(\ref{3.73.73.7})$, we conclude that
\begin{equation*}
    N_1(W)=W^{N-1}+O\left(W^{N-1}\cdot \mathfrak{p}^{-N/(2^{k}(k-1))}\right).
\end{equation*}
\end{proof}

\bigskip
 
 Throughout this paper, we only use certain cases ($l=1,\ l=k-1$) of the following lemma, however, we record this in full generality for convenience of the statement and future works. 

\begin{lem}\label{lem3.2}
Let $\beta$ be a real and $l$ be a natural number with $1\leq l\leq k-1$. Suppose that $A_1,\ldots,A_l$ are sufficiently large positive numbers where $A_1\leq \cdots\leq A_{l+1}$ with $ 1\leq\frac{\log A_{l+1}}{\log A_1}\leq 3$. Suppose that $B$ is a non-negative integer. Define
$$\mathcal{S}(\beta)=\dsum_{y}\dsum_{\x^{(1)}}\cdots\dsum_{\x^{(l)}}\biggl|\dsum_{\x^{(l+1)}\in \mathcal{B}}e(\beta \Delta_l(P(\x^{(l+1)});\x^{(1)},\ldots,\x^{(l)}))\biggr|^2,$$
where the summation in $y$ is over $[-B,B]\cap \Z$ and sums in $\x^{(i)}$ are over
$ [-A_i,A_i]^N\cap \Z^N\ (i=1,\ldots,l)$,
and the sum in $\x^{(l+1)}$ is over a rectangular box $\mathcal{B}\subseteq [-A_{l+1},A_{l+1}]^N\cap \Z^N$$($or the empty set$)$ depending on $y,\x^{(1)},\ldots,\x^{(l)}$. Then, for any $0<\sigma<1$, whenever $$N\geq 3(\sigma+2)\sigma^{-2}k(k-1)2^{k-l},$$ we have 
\begin{equation}\label{3.33.3}
    \dint_0^1\mathcal{S}(\beta)^{\sigma}d\beta\ll (B+1)^{\sigma}(A_1\cdots A_lA_{l+1}^2)^{\sigma N}(A_1\cdots A_lA_{l+1}^{k-l})^{-1},
\end{equation}
where the implicit constant depends on $k.$
\end{lem}


\bigskip

 We notice that the variable $y$ in the outer sum of $\mathcal{S}(\beta)$ is not associated with the argument of the exponential sum $\mathcal{S}(\beta)$, that is $\Delta_l(P(\x^{(l+1)});\x^{(1)},\ldots,\x^{(l)}),$ and is only involved in the range of $\x^{(l+1)}$ of the innermost sum. We note that Lemma $\ref{lem3.2}$ will be required not only for the proof of Theorem $\ref{thm3.1}$ but also in other places in this paper (proofs of Proposition $\ref{pro4.3}$, Lemma $\ref{lem5.15.1}$ and Lemma $\ref{lemma5.2}$), and that corresponding exponential sums $\mathcal{S}(\beta)$ in these places are obtained by the Weyl differencing argument. 
 
 The proof of Lemma $\ref{lem3.2}$ makes use of the standard treatment of major arcs and minor arcs used in [$\ref{ref8}$]. We emphasize here that the improvement of the required number of variables $N$, described in Theorem $\ref{thm3.1}$ and Lemma $\ref{lem3.2}$, may slightly improve the range of the degree $k$ of $P(\a)$ in Corollary ${\ref{coro1.5}}$. However, as we implicitly reflected in the proof of Corollary $\ref{coro1.5}$, since we are rather interested in the required number of variables $n$ of $f_{\a}$ in terms of the degree $d$ of $f_{\a}$ so that the hypotheses on $n,d$, and $k$ in Theorem $\ref{thm2.2}$, $\ref{thm2.3}$ and $\ref{thm1.3}$ hold for large $d$, we do not put our effort into optimizing the required number of variables $N$ throughout this paper.
\begin{rmk}
One of the ways to improve the bound for the required number of variables $N$ is to make use of pruning arguments to deal with minor arcs estimates, used in [$\ref{ref8}$]. It may provide sharper bounds for the required number of variables $N$ than that stated in Theorem $\ref{thm3.1}$ and Lemma $\ref{lem3.2}$.
\end{rmk}

\bigskip

For the proof of Lemma $\ref{lem3.2}$, we recall the definition ($\ref{3.23.23.2}$) of $\psi_j$ and it is convenient to define
\begin{equation}\label{7.1}
\begin{aligned}
&\Gamma_P(\beta;A_1,\ldots,A_{k-1},A_k)\\
&= \dsum_{\substack{\x^{(1)},\ldots,\x^{(k-1)}\\\|\x^{(1)}\|_{\infty}\leq A_1,\ldots,\|\x^{(k-1)}\|_{\infty}\leq A_{k-1}}}\dprod_{j=1}^N\min(A_k,\|\beta \psi_j(\x^{(1)},\x^{(2)},\ldots,\x^{(k-1)};P)\|^{-1}).  
\end{aligned}
\end{equation}

\begin{proof}[Proof of Lemma $\ref{lem3.2}$]
We denote $\|z\|=\min_{n\in \Z}|z-n|$, for this proof.
By applying the Weyl differencing argument to $\mathcal{S}(\beta)$, whenever $l\leq k-2$ we observe that
\begin{equation}\label{3.6}
\begin{aligned}
\mathcal{S}(\beta)\ll (B+1)(A_1\cdots A_{l+1}^2)^N(A_1\cdots A_{l+1}^{k-l})^{-2^{2-k+l}N}\Gamma_P(\beta;A_1,A_1,\ldots,A_{k})^{2^{2-k+l}},  
\end{aligned}
\end{equation}
with $A_{l+1}=\cdots=A_k.$ When $l=k-1,$ by the trivial estimate 
$$\biggl|\dsum_{\x^{(k)}\in \mathcal{B}}e(\beta \Delta_l(P(\x^{(k)};\x^{(1)},\ldots,\x^{(k-1)})))\biggr|\leq A_k^N,$$
one has
\begin{equation}\label{3.83.8}
    \mathcal{S}(\beta)\ll (B+1)A_k^N\Gamma_P(\beta;A_1,A_1,\ldots,A_{k}).
\end{equation}
Recall the definition ($\ref{3.23.23.2}$) of
$$\psi_j:=\psi_j(\x^{(1)},\x^{(2)},\ldots,\x^{(k-1)};P).$$
Then, by using the argument in the proof of [$\ref{ref17}$, Lemma 3.2], we deduce that
\begin{equation}\label{3.14}
\Gamma_P(\beta;A_1,\ldots,A_{k-1},A_k)\ll A_k^{N}(\log A_k)^NN(A_1,\ldots,A_{k-1},A_k^{-1};\beta),
\end{equation}
where 
\begin{equation}\label{3.9}
\begin{aligned}
&N(A_1,\ldots,A_{k-1},A_k^{-1};\beta)\\
&=\#\{\|\x^{(i)}\|_{\infty}\leq A_i\ (i=1,\ldots,k-1)|\ \|\beta \psi_j\|<A_k^{-1}\ (j=1,\ldots,N)\}.
\end{aligned}  
\end{equation}

 Meanwhile, we see that $\psi_j$ is a linear form in $\x^{(i)}$ for fixed $\x^{(m)}\ (1\leq m\leq k-1,m\neq i)$. Hence, for any positive number $X$ with $X<A_1,$ we set $n=N,$ $L_i=\psi_i$ and
\begin{equation*}
\begin{aligned}
&Z_1=X^{\frac{1}{2}(i+1)}(A_1\cdots A_iA_k)^{-\frac{1}{2}},\\ &Z_2=(A_iX^{i-1})^{\frac{1}{2}}(A_1\cdots A_{i-1}A_k)^{-\frac{1}{2}},\\
&a=(A_1\cdots A_iA_k)^{\frac{1}{2}}X^{-\frac{1}{2}(i-1)}  
\end{aligned}  
\end{equation*} 
We note here that the condition $X< A_1$ and the hypothesis $A_1\leq \cdots\leq A_{l+1}$ in the statement of Lemma $\ref{lem3.2}$ ensure that $a>1$ and $0<Z_1<Z_2\leq 1.$ It is possible to replace the hypothesis  $A_1\leq\cdots\leq  A_{l}$ by a weaker condition $\text{max}_{1\leq i\leq l}A_i\leq A_{l+1}$, although we have to instead choose $X$ less than $\text{min}_{1\leq i\leq l}A_i$. However, we adapt the hypothesis $A_1\leq \cdots\leq A_l\leq A_{l+1}$, purely for notational convenience throughout this proof. With these quantities, it follows by applying Lemma $\ref{lem3.1}$
that for $i$ with $1\leq i\leq k-2$ we have
\begin{equation*}
\begin{aligned}
&N(X,\ldots,X,A_i,\ldots,A_{k-1},X^{i-1}(A_1\cdots A_{i-1}A_k)^{-1};\beta)\\
&\ll (A_i/X)^N N(X,\ldots,X,A_{i+1},\ldots,A_{k-1},X^i(A_1\cdots A_iA_k)^{-1};\beta),
\end{aligned}  
\end{equation*}
and for $i=k-1$ we have
\begin{equation*}
\begin{aligned}
&N(X,\ldots,X,A_{k-1},X^{k-2}(A_1\cdots A_{k-2}A_k)^{-1};\beta)\\
&\ll (A_{k-1}/X)^N N(X,\ldots,X,X^{k-1}(A_1\cdots A_{k-1}A_k)^{-1};\beta).
\end{aligned}  
\end{equation*}
Hence, we deduce by applying this with $i=1,\ldots,k-1$ recursively that \begin{equation}\label{3.63.63.6}
\begin{aligned}
&N(A_1,\ldots,A_{k-1},A_k^{-1};\beta)\\
&\ll (A_1/X)^N\cdots (A_{k-1}/X)^N N(X,\ldots,X,X^{k-1}(A_1\cdots A_{k})^{-1};\beta).
\end{aligned}
\end{equation}

Furthermore, on noting that $P$ is a non-singular form and that $A_{l+1}=\cdots=A_k$, application of the argument used in [$\ref{ref20}$, Lemma 3.3, Lemma 3.4] together with ($\ref{3.14}$) and ($\ref{3.63.63.6}$) readily delivers that either
\begin{equation}\label{3.7}
(i)\ \Gamma_P(\beta;A_1,\ldots,A_{k})\ll (A_1\cdots A_l A_{l+1}^{k-l})^N(\log A_{l+1})^{N+1}X^{-N}    
\end{equation}
or there exist $a\in \Z$ and $q\in \N$ such that
\begin{equation}\label{3.8}
(ii)\ (q,a)=1,\ q\leq X^{k-1}\ \text{and}\ |q\beta-a|\leq X^{k-1}(A_1\cdots A_{l}A_{l+1}^{k-l})^{-1}.    
\end{equation}

In order to analyze the mean value  $\dint_0^1\mathcal{S}(\beta)^{\sigma}d\beta$ in $(\ref{3.33.3})$, we must define 
$$\widetilde{\mathfrak{M}}(Q)=\displaystyle\bigcup_{\substack{0\leq a\leq q\leq Q\\(q,a)=1}}\widetilde{\mathfrak{M}}_{q,a}(Q),$$
where $$\widetilde{\mathfrak{M}}_{q,a}(Q)=\{\beta\in [0,1)|\ |\beta-a/q|\leq q^{-1}Q(A_1\cdots A_{l}A_{l+1}^{k-l})^{-1}\},$$
and define $\widetilde{\mathfrak{m}}(Q)=[0,1)\setminus \widetilde{\mathfrak{M}}(Q).$ Then, we observe that
\begin{equation}\label{3.10}
    \dint_0^1\mathcal{S}(\beta)^{\sigma}d\beta\ll \mathcal{J}_1+\mathcal{J}_2,
\end{equation}
where 
\begin{equation}\label{3.73.7}
    \begin{aligned}
    \mathcal{J}_1=\dint_{\widetilde{\mathfrak{M}}(A_{1}^{\eta})}\mathcal{S}(\beta)^{\sigma}d\beta\ \ \text{and}\ \ \mathcal{J}_2=\dint_{\widetilde{\mathfrak{m}}(A_{1}^{\eta})}\mathcal{S}(\beta)^{\sigma}d\beta
    \end{aligned}
\end{equation}
in which $\eta=\frac{\sigma}{2(\sigma+2)}.$

We first analyze $\mathcal{J}_1$ by making use of the standard treatment of major arcs used in [$\ref{ref8}$, section 5]. Let $\beta\in \widetilde{\mathfrak{M}}_{q,a}(A_{1}^{\eta}),$ and write $\delta=\beta-a/q$ with $|\delta|\leq q^{-1}A^{\eta}_{1}(A_1\cdots A_{l}A_{l+1}^{k-l})^{-1}.$ On recalling the definition of $\mathcal{S}(\beta)$, one finds that 
$$\mathcal{S}(\beta)=\dsum_{y}\dsum_{\x^{(1)},\ldots,\x^{(l)}}\dsum_{\x^{(l+1)},\x^{(l+2)}\in \mathcal{B}}e(\beta \phi_1(\x^{(l+1)},\x^{(l+2)})),$$
where
\begin{equation*}
\begin{aligned}
 \phi_1(\x^{(l+1)},\x^{(l+2)})&:=\phi_1(\x^{(l+1)},\x^{(l+2)};\x^{(1)},\ldots,\x^{(l)})\\
 &=\Delta_l(P(\x^{(l+1)});\x^{(1)},\ldots,\x^{(l)})-\Delta_l(P(\x^{(l+2)});\x^{(1)},\ldots,\x^{(l)}).
\end{aligned}   
\end{equation*}

Let us write $\x^{(i)}=q\y^{(i)}+\z^{(i)}\ (i=1,\ldots,l+2)$, where $1\leq \z^{(i)}\leq q$ and $\y^{(i)}$ runs over boxes so that $\|\x^{(i)}\|_{\infty}\leq A_{i}$ with $i=1,\ldots,l$, and $1\leq \z^{(l+1)},\z^{(l+2)}\leq q$ and $\y^{(l+1)},\y^{(l+2)}$ run over boxes so that $\x^{(l+1)},\x^{(l+2)}\in \mathcal{B}$. 
\begin{equation}\label{3.93.9}
\begin{aligned}
    &\mathcal{S}(\beta)= \dsum_{y}\dsum_{\z^{(1)},\ldots,\z^{(l+2)}}e\left(\frac{a}{q} \phi_1(\z^{(l+1)},\z^{(l+2)};\z^{(1)},\ldots,\z^{(l)})\right) \mathcal{T}(\delta,\z^{(1)},\ldots,\z^{(l+2)}),
\end{aligned}
\end{equation}
where 
\begin{equation*}
    \mathcal{T}(\delta,\z^{(1)},\ldots,\z^{(l+2)})=\dsum_{\y^{(1)},\ldots,\y^{(l+2)}}e\left(\delta\phi_2(\y^{(l+1)},\y^{(l+2)};\y^{(1)},\ldots,\y^{(l)})\right)
\end{equation*}
in which 
\begin{equation*}
\begin{aligned}
 \phi_2(\y^{(l+1)},\y^{(l+2)})&:=\phi_2(\y^{(l+1)},\y^{(l+2)};\y^{(1)},\ldots,\y^{(l)})\\
 &=\phi_1(q\y^{(l+1)}+\z^{(l+1)},q\y^{(l+2)}+\z^{(l+2)};q\y^{(1)}+\z^{(1)},\ldots,q\y^{(l)}+\z^{(l)}).   
\end{aligned}
\end{equation*}
By the definition of $\mathcal{B},$ we see that the ranges of $\x^{(l+1)}=q\y^{(l+1)}+\z^{(l+1)}$, $\x^{(l+2)}=q\y^{(l+2)}+\z^{(l+2)}$ are equal and depend at most on $y,\z^{(1)},\ldots,\z^{(l)}$ and $\y^{(1)},\ldots,\y^{(l)},$ 


Then, by the same treatment used in [$\ref{ref8}$, Lemma 5.1], we find that there exist boxes $\mathfrak{D}_i\ (i=1,\ldots,l+2)$ with side length at most $2$ such that
\begin{equation}\label{3.113.11}
\begin{aligned}
\mathcal{T}(\delta,\z^{(1)},\ldots,\z^{(l+2)})-\mathcal{U}(\delta)\ll E_1+E_2,
\end{aligned}
\end{equation}
where \begin{equation*}
 \begin{aligned}
&E_1=q(A_1/q)^N(A_2/q)^N\cdots(A_{l+1}/q)^{2N}A_2\cdots A_l A_{l+1}^{k-l}|\delta|\\
&E_2\leq(A_1/q)^{N-1}(A_2/q)^N\cdots(A_{l+1}/q)^{2N}
\end{aligned}
\end{equation*}
and
\begin{equation*}
\begin{aligned}
\mathcal{U}(\delta)=\dint_{q\boldsymbol{\gamma}^{(1)}+\z^{(1)}\in A_{1}\mathfrak{D}_1}\cdots\dint_{q\boldsymbol{\gamma}^{(l+2)}+\z^{(l+2)}\in A_{l+2}\mathfrak{D}_{l+2}}e(\delta \phi_2(\boldsymbol{\gamma}^{(l+1)},\boldsymbol{\gamma}^{(l+2)}))d\boldsymbol{\gamma}^{(l+2)}\cdots d\boldsymbol{\gamma}^{(1)}
\end{aligned}
\end{equation*}
 in which  $A_{l+1}=A_{l+2}$ and $\mathfrak{D}_{l+1}=\mathfrak{D}_{l+2}$ depend at most on $y,\z^{(1)},\ldots,\z^{(l)},\boldsymbol{\gamma}^{(1)},\ldots,\boldsymbol{\gamma}^{(l)},$ and are uniformly bounded.
Notice that $E_2$ accounts for the initial and final interval with length at most $O(1)$ for each coordinate of $\boldsymbol{\gamma}^{(j)}(j=1,\ldots,l+2)$ in $\mathcal{U}(\delta).$ 

Furthermore, by change of variables $q\boldsymbol{\gamma}^{(i)}+\boldsymbol{z}^{(i)}=A_{i}\boldsymbol{\eta}^{(i)}$ with $1\leq i\leq l+2$ and $A_{l+1}=A_{l+2}$, we find that
\begin{equation}\label{3.123.12}
\begin{aligned}
   \mathcal{U}(\delta)=q^{-(l+2)N}(A_1\cdots A_lA_{l+1}^2)^NI(A_1\cdots A_lA_{l+1}^{k-l}\delta),
\end{aligned}
\end{equation}
where 
$$I(\delta)=\dint_{\mathfrak{D}_1}\cdots\dint_{\mathfrak{D}_{l+2}}e(\delta \phi_1(\boldsymbol{\eta}^{(l+1)},\boldsymbol{\eta}^{(l+2)};\boldsymbol{\eta}^{(1)}\ldots,\boldsymbol{\eta}^{(l)}))d\boldsymbol{\eta}^{(l+2)}\cdots d\boldsymbol{\eta}^{(1)}.$$
Notice here that $I(\delta)$ depends at most on  $y,\z^{(1)},\ldots,\z^{(l)}.$

Meanwhile, on recalling the definition of $\phi_1$, observe that
\begin{equation*}
\begin{aligned}
&\dsum_{\z^{(1)},\ldots,\z^{(l+2)}}e\left(\frac{a}{q}\phi_1(\z^{(l+1)},\z^{(l+2)};\z^{(1)},\ldots,\z^{(l)})\right),     \\
&=\dsum_{\z^{(1)},\ldots,\z^{(l)}}\biggl|\dsum_{1\leq \z\leq q}e\left(\frac{a}{q}\Delta_l(P(\z);\z^{(1)},\ldots,\z^{(l)})\right)\biggr|^2.
\end{aligned}
\end{equation*}
For simplicity, we write the inner sum over $\z$ by
$$S_{q,a}:=S_{q,a}(\z^{(1)},\ldots,\z^{(l)})=\dsum_{1\leq \z\leq q}e\left(\frac{a}{q}\Delta_l(P(\z);\z^{(1)},\ldots,\z^{(l)})\right).$$
Hence, by substituting $(\ref{3.123.12})$ into ($\ref{3.113.11}$) and that into ($\ref{3.93.9}$), we obtain
\begin{equation}\label{3.13}
    \mathcal{S}(\beta)=\dsum_{y}\dsum_{\z^{(1)},\ldots,\z^{(l)}}|S_{q,a}|^2\cdot q^{-(l+2)N}(A_1\cdots A_lA_{l+1}^2)^NI(A_1\cdots A_lA_{l+1}^{k-l}\delta) +O(E(\beta)),
\end{equation}
where
$E(\beta)=(B+1)q^{(l+2)N}(E_1+E_2).$
This is possible because $I(\delta)$ depends at most on $y,\z^{(1)},\ldots,\z^{(l)}.$
For simplicity, we write 
\begin{equation*}
    \mathcal{I}(A_1\cdots A_lA_{l+1}^{k-l}\delta)=\sup_{\z^{(1)},\ldots,\z^{(l)}}\biggl|\dsum_{-B\leq y\leq B}I(A_1\cdots A_lA_{l+1}^{k-l}\delta)\biggr|
\end{equation*}
and 
\begin{equation*}
    \mathcal{L}(q,a)=\dsum_{\z^{(1)},\ldots,\z^{(l)}}|S_{q,a}|^2\cdot q^{-(l+2)N}(A_1\cdots A_lA_{l+1}^2)^N.
\end{equation*}
Then, it follows from ($\ref{3.13}$) that 
\begin{equation}\label{8.14}
    \mathcal{S}(\beta)\ll \mathcal{L}(q,a)\cdot  \mathcal{I}(A_1\cdots A_lA_{l+1}^{k-l}\delta)+|E(\beta)|.
\end{equation}

Meanwhile, since $\beta\in \widetilde{\mathfrak{M}}_{q,a}(A_{1}^{\eta})$, we find from the definition of $E_1$ and $E_2$ that
\begin{equation}\label{3.143.14}
\begin{aligned}
E(\beta)\ll(B+1)(A_1\cdots A_lA_{l+1}^2)^NA_{1}^{-1+\eta}.
\end{aligned}
\end{equation}
Note that $$\text{mes}(\widetilde{\mathfrak{M}}(A_{1}^{\eta}))\leq (A_1\cdots A_lA_{l+1}^{k-l})^{-1}A_{1}^{2\eta}.$$ 
Thus, on substituting $(\ref{3.143.14})$ into $(\ref{8.14})$ and that into $\mathcal{J}_1$ in ($\ref{3.73.7}$), we deduce from the elementary inequality $(a+b)^{\sigma}\leq a^{\sigma}+b^{\sigma}$ with $0<\sigma<1$ and $a,b>0$ that
\begin{equation}\label{3.18}
\begin{aligned}
    \mathcal{J}_1\ll \dsum_{1\leq q\leq A_{1}^{\eta}}\dsum_{\substack{1\leq a\leq q\\ (q,a)=1}}(\mathcal{L}(q,a))^{\sigma}\mathfrak{J}(A_{1}^{\eta})+ E,
\end{aligned}
\end{equation}
where
$$\mathfrak{J}(A_{1}^{\eta})=\dint_{|\delta|\leq q^{-1}A^{\eta}_{1}\left(A_1\cdots A_{l}A_{l+1}^{k-l}\right)^{-1}}\left(\mathcal{I}(A_1\cdots A_lA_{l+1}^{k-l}\delta))\right)^{\sigma}d\delta$$
and 
\begin{equation*}
\begin{aligned}
E&=O\left(\text{mes}(\widetilde{\mathfrak{M}}(A_{1}^{\eta}))(B+1)^{\sigma}(A_1\cdots A_lA_{l+1}^2)^{\sigma N}A_{1}^{-\sigma+\sigma\eta}\right)\\
&=O\left((B+1)^{\sigma}(A_1\cdots A_lA_{l+1}^2)^{\sigma N}(A_1\cdots A_lA_{l+1}^{k-l})^{-1}A_{1}^{-\sigma+(\sigma+2)\eta}\right).    
\end{aligned}
\end{equation*}
By the same treatment used in [$\ref{ref8}$, Lemma 5.2, Lemma 5.4] together with ($\ref{3.113.11}$) using $q=1,a=0$ and with the Weyl type estimate over minor arcs for the exponential sum $S(\beta)$ derived by ($\ref{3.6}$), ($\ref{3.7}$) and ($\ref{3.8}$), one infers that
$$\mathcal{I}(\gamma)\ll (B+1)\text{min}(1, |\gamma|^{-2^{2-k+l}N/(k-1)+\epsilon})$$
and 
$$\mathcal{L}(q,a)\ll q^{-2^{2-k+l}N/(k-1)+\epsilon}\cdot(A_1\cdots A_lA_{l+1}^2)^N.$$
Hence, whenever $N\geq 3(k-1)2^{k-l+2}\sigma^{-1}$ and $1\leq l\leq k-2$, one infers that
\begin{equation}\label{3.2020}
    \begin{aligned}
    &(i)\ \mathfrak{J}(A^{\eta}_{1})\ll (B+1)^{\sigma}(A_1\cdots A_l A_{l+1}^{k-l})^{-1}\ \text{uniformly in} \ q\geq 1.\\
   &(ii)\ \dsum_{1\leq q\leq A_{1}^{\eta}}\dsum_{\substack{1\leq a\leq q\\ (q,a)=1}}(\mathcal{L}(q,a))^{\sigma}\ll (A_1\cdots A_lA_{l+1}^2)^{N\sigma}.
    \end{aligned}
\end{equation}
By using ($\ref{3.83.8}$) in place of $(\ref{3.6})$, when $l=k-1$ and $N\geq 3(k-1)\sigma^{-1}$, we obtain the same estimates with ($\ref{3.2020}$).
Hence, since $\eta=\frac{\sigma}{2(\sigma+2)}$, on substituting $(\ref{3.2020})$ into ($\ref{3.18}$) we conclude that 
\begin{equation}\label{3.19}
\begin{aligned}
\mathcal{J}_1\ll(B+1)^{\sigma}(A_1\cdots A_{l+1}^2)^{N\sigma}(A_1\cdots A_l A_{l+1}^{k-l})^{-1}.
\end{aligned}
\end{equation}

Next, we turn to estimate $\mathcal{J}_2$ in $(\ref{3.73.7})$. By ($\ref{3.6}$), ($\ref{3.7}$) and $(\ref{3.8})$, whenever $\beta \in \widetilde{\mathfrak{m}}(A_{1}^{\eta})$, one easily infers that when $1\leq l\leq k-2$
$$\mathcal{S}(\beta)\ll (B+1)(A_1\cdots A_lA_{l+1}^{2})^N(\log A_{l+1})^{N+1}A_{1}^{-2^{2-k+l}(k-1)^{-1}N\eta}.$$
By using ($\ref{3.83.8}$) in place of $(\ref{3.6}),$ when $l=k-1$, we obtain the bound
$$\mathcal{S}(\beta)\ll (B+1)(A_1\cdots A_lA_{l+1}^{2})^N(\log A_{l+1})^{N+1}A_{1}^{-(k-1)^{-1}N\eta}.$$
Therefore, on recalling  the hypothesis $1\leq\frac{\log A_{l+1}}{\log A_1}\leq 3$ in the statement of Lemma $\ref{lem3.2}$, in all cases, it follows that  whenever $N\geq 6k(k-1)2^{k-l-2}\eta^{-1}\sigma^{-1}$ one has
\begin{equation}\label{3.20}
\begin{aligned}
    \mathcal{J}_2\ll \sup_{\beta\in \widetilde{\mathfrak{m}}(A_{1}^{\eta})}\mathcal{S}(\beta)^{\sigma}
    &= o\left((B+1)^{\sigma}(A_1\cdots A_lA_{l+1}^{2})^{\sigma N}(A_1\cdots A_l A_{l+1}^{k-l})^{-1}\right).
\end{aligned}
\end{equation}
The inequality ($\ref{3.10}$) with ($\ref{3.19}$) and ($\ref{3.20}$) delivers that 
$$ \dint_0^1\mathcal{S}(\beta)^{\sigma}d\beta\ll(B+1)^{\sigma}(A_1\cdots A_lA_{l+1}^{2})^{\sigma N}(A_1\cdots A_l A_{l+1}^{k-l})^{-1}, $$
and thus this completes the proof of Lemma $\ref{lem3.2}.$
\end{proof}

\bigskip

\begin{proof}[Proof of Theorem $\ref{thm3.1}$]
We shall first derive the lower bound for $\mathbf{N}(A).$ By making the trivial observation that whenever $P(\x)=0$, one has $P(\x)\equiv 0\ \text{mod}\ W,$ it follows by applying the Cauchy-Schwarz inequality that 
\begin{equation}\label{3.313.31}
    \begin{aligned}
        &\#\left\{\x\in [-A,A]^N\cap \Z^N\middle|\ P(\x)=0\right\}\\
        &=\dsum_{\substack{1\leq \b\leq W\\P(\b)\equiv 0\ \text{mod}\ W}} \#\left\{\x\in [-A,A]^N\cap \Z^N\middle|\ P(\x)=0,\ \x\equiv \b\ \text{mod}\ W\right\}\\
        &\leq \biggl(\dsum_{\substack{1\leq\b\leq W\\P(\b)\equiv 0\ \text{mod}\ W}}1\biggr)^{1/2}\cdot \mathcal{S}^{1/2},
    \end{aligned}
\end{equation}
where
$$\mathcal{S}=\dsum_{\substack{1\leq \b\leq W\\P(\b)\equiv 0\ \text{mod}\ W}}\#\left\{\x\in [-A,A]^N\cap \Z^N\middle|\ P(\x)=0,\ \x\equiv \b\ \text{mod}\ W\right\}^2.$$
On noting that 
$\mathcal{S}\leq \mathbf{N}(A)$ and by applying Lemma $\ref{lem3.93.9}$, it follows from $(\ref{3.313.31})$ that
\begin{equation}\label{3.323.32}
  \#\left\{\x\in [-A,A]^N\cap \Z^N\middle|\ P(\x)=0\right\}\ll (W^{N-1})^{1/2}\cdot \mathbf{N}(A)^{1/2}.
\end{equation}
Meanwhile, since $P(\x)=0$ has a nontrivial integer solution and $P$ is a non-singular form, we find by [$\ref{ref8}$] that 
$$\#\left\{\x\in [-A,A]^N\cap \Z^N\middle|\ P(\x)=0\right\}\asymp A^{N-k},$$
and thus it follows by $(\ref{3.323.32})$ that
\begin{equation}\label{3.34}
 A^{2N-2k}W^{1-N}\ll \mathbf{N}(A).   
\end{equation}

Next, we shall derive the upper bound for $\mathbf{N}(A).$
By making the trivial observation that whenever $P(\x)=P(\y)=0$, one has $P(\x)\equiv P(\y)\equiv 0\ \text{mod}\ W$, it follows by orthogonality and by applying the triangle inequality that
\begin{equation}\label{5.25.2}
\begin{aligned}
  \mathbf{N}(A)&= W^{-N}\dint_{[0,1)^2}\dsum_{\substack{1\leq \l\leq W\\\l\in \Z^{N}}}\dsum_{\substack{-A\leq\x,\y\leq A\\ P(\x)\equiv P(\y)\equiv 0\ \text{mod}\ W \\ \x,\y\in \Z^N}}e(\alpha P(\x)+\beta P(\y))e\left(\frac{\langle\x-\y,\l\rangle}{W}\right)d\alpha \ d\beta    \\
  &\leq  W^{-N}\dint_{[0,1)^2}\dsum_{\substack{1\leq \l\leq W\\\l\in \Z^{N}}}\biggl|\dsum_{\substack{-A\leq\x,\y\leq A\\ P(\x)\equiv P(\y)\equiv 0\ \text{mod}\ W\\ \x,\y\in \Z^N}}e(\alpha P(\x)+\beta P(\y))e\left(\frac{\langle\x-\y,\l\rangle}{W}\right)\biggr|d\alpha \ d\beta.
\end{aligned}
\end{equation}
We find by applying the Cauchy-Schwarz inequality that the last expression in  $(\ref{5.25.2})$ is bounded above by 
\begin{equation}\label{3.22}
W^{-N}\dint_0^1\Xi(\alpha)^{1/2}d\alpha\dint_0^1\Xi(\beta)^{1/2}d\beta=W^{-N}\biggl(\dint_0^1\Xi(\alpha)^{1/2}d\alpha\biggr)^2,
\end{equation}
where
$$\Xi(\xi)=\dsum_{1\leq\l\leq W}\biggl|\dsum_{\substack{-A\leq \x\leq A\\P(\x)\equiv0\ \text{mod}\ W}}e(\xi P(\x))e\left(\frac{\langle \x,\l\rangle}{W}\right)\biggr|^2.$$
By orthogonality again, the sum $\Xi(\xi)$ is seen to be
\begin{equation*}
    \Xi(\xi)=\dsum_{1\leq\l\leq W}\biggl|W^{-1}\dsum_{1\leq r\leq W}\dsum_{-A\leq \x\leq A}e(\xi P(\x))e\left(\frac{\langle \x,\l\rangle}{W}\right)e\biggl(\frac{P(\x)r}{W}\biggr)\biggr|^2.
\end{equation*}

We now analyze the sum $\Xi(\xi).$ Define 
$$S(r/W,\xi)=\dsum_{-2A/W\leq \h\leq 2A/W}\dsum_{\x\in I_{\h}}e(\xi( P(\x+W\h)-P(\x)))e\biggl(\frac{P(\x)r}{W}\biggr),$$
where $I_{\h}=\{-A\leq \x\leq A|\ -A\leq\x+W\h\leq A\}$. Then,
by squaring out and inverting the order of summation, we see by applying orthogonality with respect to variables $\l$ and applying the change of variables with respect to $\x$ that
\begin{equation}\label{3.23}
    \Xi(\xi)=W^{N-1}\dsum_{1\leq r\leq W}S(r/W,\xi).
\end{equation}

 In order to apply the Weyl differencing argument to the sum $S(r/W,\xi),$ we recall the definition $(\ref{3.3})$ of the differencing operator $\Delta_j$, and define
\begin{equation*}
\begin{aligned}
U_1(\x^{(1)})&:=U_1(\x^{(1)};\h,\x^{(k-1)},\ldots,\x^{(2)})\\
&=\Delta_{k-1}(P(\x^{(1)});W\h,\x^{(k-1)},\ldots,\x^{(2)})    \\
&=W\Delta_{k-1}(P(\x^{(1)});\h,\x^{(k-1)},\ldots,\x^{(2)}) 
\end{aligned}
\end{equation*}
and
\begin{equation*}
\begin{aligned}
U_2(\x^{(1)})&:=U_2(\x^{(1)};\x^{(k-1)},\ldots,\x^{(2)})=\Delta_{k-2}(P(\x^{(1)});\x^{(k-1)},\ldots,\x^{(2)}).
\end{aligned}
\end{equation*}
Notice here that $U_1(\x^{(1)})$ is a polynomial in $\x^{(1)}$ of degree $1$, and $U_2(\x^{(1)})$ is a polynomial in $\x^{(1)}$ of degree $2.$ Furthermore, define
\begin{equation*}
    \mathcal{B}^{(0)}:=\mathcal{B}^{(0)}(\h)=\left\{\x\in [-A,A]^N\middle|\ \x+W\h\in [-A,A]^N\right\}
\end{equation*}
and define $\mathcal{B}^{(j)},$ when $j\geq 1$, by recursively setting
\begin{equation*}
\begin{aligned}
    \mathcal{B}^{(1)}&:=\mathcal{B}^{(1)}(\h,\x^{(k-1)})\\
    &=\mathcal{B}^{(0)}\cap \left\{\x\in [-A,A]^N\middle|\ \x+\x^{(k-1)}\in \mathcal{B}^{(0)}\right\}    
\end{aligned}
\end{equation*}
and 
\begin{equation*}
\begin{aligned}
    \mathcal{B}^{(j)}&:=\mathcal{B}^{(j)}(\h,\x^{(k-1)},\ldots,\x^{(k-j)})\\
    &=\mathcal{B}^{(j-1)}\cap \left\{\x\in [-A,A]^N\middle|\ \x+\x^{(k-j)}\in \mathcal{B}^{(j-1)}\right\}.
\end{aligned}
\end{equation*}
Additionally, we define
\begin{equation*}
    T(r/W,\xi)=\dsum_{\h}\dsum_{\x^{(2)},\ldots,\x^{(k-1)}}\biggl|\dsum_{\x^{(1)}}e(\xi U_1(\x^{(1)}))e((r/W)U_2(\x^{(1)}))\biggr|,
    \end{equation*}
where $-2A/W\leq \h\leq 2A/W$, $-2A\leq \x^{(2)},\ldots,\x^{(k-1)}\leq 2A$, and $\x^{(1)}$ runs over the rectangular box $\mathcal{B}^{(k-2)}.$

Then, on observing that $P(\x+W\h)-P(\x)=\Delta_1(P(\x);W\h)$ and by applying the standard Weyl differencing argument to the sum $S(r/W,\xi)$, we deduce that
\begin{equation}\label{5.4}
    S(r/W,\xi)\ll (A/W)^{N-2^{2-k}N}A^{N-(k-1)2^{2-k}N}T(r/W,\xi)^{2^{2-k}}.
\end{equation}


 In order to apply the Weyl differencing argument once again to $T(r/W,\xi)$, we define 
\begin{equation*}
\begin{aligned}
    U_3(\h,\x^{(1)})&=\Delta_{k-1}(P(\x^{(1)});\h,\x^{(k-1)},\ldots,\x^{(2)})\\
    U_4(\x^{(0)},\x^{(1)})&=\Delta_{k-1}(P(\x^{(0)});\x^{(k-1)},\ldots,\x^{(1)}),
\end{aligned}
\end{equation*}
and
\begin{equation*}
    V(r/W,\xi)=\dsum_{\h}\dsum_{\x^{(1)},\x^{(2)},\ldots,\x^{(k-1)}}\dsum_{\x^{(0)}}e(W\xi U_3(\h,\x^{(1)}))e((r/W)U_4(\x^{(0)},\x^{(1)})),
\end{equation*}
where $-2A\leq \x^{(1)}\leq 2A$ and $\x^{(0)}$ runs over the rectangular box $\mathcal{B}^{(k-1)}.$ 
Notice from the definition of $\mathcal{B}^{(k-1)}$ that the range of $\x^{(0)}$ depends on $\h,\x^{(1)},\ldots,\x^{(k-1)}$, and that $U_3(\h,\x^{(1)})=U_1(\x^{(1)})/W.$
Then, by applying the Cauchy-Schwarz inequality and Weyl differencing, one deduces that
\begin{equation}\label{3.24}
\begin{aligned}
    T(r/W,\xi)\leq W^{-N/2}A^{(k-1)N/2}V(r/W,\xi)^{1/2}.
\end{aligned}
\end{equation}

By applying the triangle inequality and the Cauchy-Schwarz inequality again, we find that
\begin{equation}\label{5.5}
\begin{aligned}
    V(r/W,\xi)&\leq \dsum_{\x^{(1)},\ldots,\x^{(k-1)}}|Y(r/W,\xi)|   \leq A^{(k-1)N/2}\biggl(\dsum_{\x^{(1)},\ldots,\x^{(k-1)}}|Y(r/W,\xi)|^2\biggr)^{1/2},
\end{aligned}
\end{equation}
where 
$$Y(r/W,\xi)=\dsum_{\h,\x^{(0)}}e(W\xi U_3(\h,\x^{(1)}))e((r/W)U_4(\x^{(0)},\x^{(1)})).$$
Meanwhile, 
we see from the triangle inequality that
\begin{equation}\label{5.6}
\begin{aligned}
   Y(r/W,\xi)\leq \dsum_{-2A/W\leq\h\leq 2A/W}|Z_1(r/W)|,
\end{aligned}
\end{equation}
where
\begin{equation*}
    Z_1(r/W)=\dsum_{\x^{(0)}\in \mathcal{B}^{(k-1)}}e((r/W)U_4(\x^{(0)},\x^{(1)})).
\end{equation*}
 Note that the range of $\h$ is contained in $[-2A/W,2A/W]$ and the range of $\x^{(0)}$ is contained in $[-A,A]^N.$ Hence, if we change the order of sums over $\h,\x^{(0)}$ in $Y(r/W,\xi)$ and if we further define the range of $\h$ to be the empty set when $\x^{(0)}\notin [-A,A]^N,$ one sees from the triangle inequality again that
\begin{equation}\label{5.7}
\begin{aligned}
    Y(r/W,\xi)\leq \dsum_{-2A\leq \x^{(0)}\leq 2A}|Z_2(\xi)|,
\end{aligned}
\end{equation}
where
\begin{equation*}
    Z_2(\xi)=\dsum_{\h\in B(\x^{(0)})}e(W\xi U_3(\h,\x^{(1)})),
\end{equation*}
in which $B(\x^{(0)}):=B(\x^{(0)};\x^{(1)},\ldots,\x^{(k-1)})$ is an $N$-dimensional set depending on $\x^{(0)}$, $\x^{(1)}$,$\ldots$, $\x^{(k-1)}.$ We note from the definition of $\mathcal{B}^{(k-1)}$ that $B(\x^{(0)})$ is a rectangular box (or the empty set) contained in $[-2A/W,2A/W]^N$.

Hence, on substituting $(\ref{5.6})$ and $(\ref{5.7})$ into $(\ref{5.5})$, we discern that
\begin{equation*}
\begin{aligned}
    &V(r/W,\xi)\\
    &\leq A^{(k-1)N/2}\biggl(\dsum_{\x^{(1)},\ldots,\x^{(k-1)}}\biggl(\dsum_{-2A/W\leq\h\leq 2A/W}|Z_1(r/W)|\biggr)\biggl(\dsum_{-2A\leq \x^{(0)}\leq 2A}|Z_2(\xi)|\biggr)\biggr)^{1/2}.   
\end{aligned}
\end{equation*}
By applying the Cauchy-Schwarz inequality and Weyl differencing argument again, we deduce that
\begin{equation}\label{3.29}
\begin{aligned}
    V(r/W,\xi)&\leq A^{(k-1)N/2}G_1(r/W)^{1/4}G_2(\xi)^{1/4},
\end{aligned}
\end{equation}
where 
$$G_1(r/W)=\dsum_{\x^{(1)},\ldots,\x^{(k-1)}}\biggl(\dsum_{-2A/W\leq\h\leq 2A/W}|Z_1(r/W)|\biggr)^2$$
and
$$G_2(\xi)=\dsum_{\x^{(1)},\ldots,\x^{(k-1)}}\biggl(\dsum_{-2A\leq \x^{(0)}\leq 2A}|Z_2(\xi)|\biggr)^2.$$

On substituting $(\ref{3.29})$ into $(\ref{3.24})$ and that into $(\ref{5.4})$, one has
\begin{equation*}
    S(r/W,\xi)\ll (A/W)^{N}A^NA^{-(k+3)2^{-k}N}W^{2^{1-k}N}G_1(r/W)^{2^{-1-k}}G_2(\xi)^{2^{-1-k}}.
\end{equation*}
Then, on substituting this bound for $S(r/W,\xi)$ into $(\ref{3.23})$, we deduce that
\begin{equation*}
    \Xi(\xi)\ll W^{-1}A^{2N}A^{-(k+3)2^{-k}N}W^{2^{1-k}N}G_2(\xi)^{2^{-1-k}}\dsum_{1\leq r\leq W}G_1(r/W)^{2^{-1-k}}.
\end{equation*}
Thus, it follows from ($\ref{5.25.2}$) and ($\ref{3.22}$) that 
\begin{equation}\label{3.35}
    \mathbf{N}(A)\ll C(A,W)\biggl(\dint_0^1G_2(\alpha)^{2^{-2-k}}d\alpha\biggr)^{2},
\end{equation}
where
$$C(A,W)=W^{-N-1+2^{1-k}N}A^{2N-(k+3)2^{-k}N}\dsum_{1\leq r\leq W}G_1(r/W)^{2^{-1-k}}.$$

As the endgame of the proof of Theorem $\ref{thm3.1}$, we first analyse the mean value
\begin{equation}\label{3.36}
\dint_0^1G_2(\alpha)^{2^{-2-k}}d\alpha,    
\end{equation}
in $(\ref{3.35})$. In order to apply Lemma $\ref{lem3.2}$, we temporarily pause to bound $G_2(\alpha)$ by a different type of exponential sum. On recalling the definition ($\ref{3.13.13.1}$) of $P(\a)$, we define 
$$\Psi(\x^{(1)},\ldots, \x^{(k)})=k!\dsum_{1\leq j_1,\ldots,j_k\leq N}p_{j_1\cdots j_k}(\x^{(1)})_{j_1}\cdots (\x^{(k)})_{j_k}.$$
Then, on recalling the definition of $G_2(\alpha)$ and noting that $P$ is a homogeneous polynomial of degree $k,$ we see that
$$G_2(\alpha)=\dsum_{\x^{(1)},\ldots,\x^{(k-1)}}\biggl(\dsum_{-2A\leq \x^{(0)}\leq 2A}\biggl|\dsum_{\h\in B(\x^{(0)})}e(W\alpha \Psi^{(1)}(\h,\x^{(k-1)}))\biggr|\biggr)^2,$$
where $\Psi^{(1)}(\h,\x^{(k-1)})=\Psi(\h,\x^{(1)},\ldots,\x^{(k-1)}).$
Then, we see from the Cauchy-Schwarz inequality that 
\begin{equation*}
     G_2(\alpha)\ll A^N\dsum_{\x^{(1)},\ldots,\x^{(k-1)}}\dsum_{\x^{(0)}}\dsum_{\h^{(1)},\h^{(2)}\in B(\x^{(0)})}e(W\alpha \Psi^{(1)}(\h^{(1)}-\h^{(2)},\x^{(k-1)})).
\end{equation*}
 By changing the order of summations, we find that
\begin{equation*}
     G_2(\alpha)\ll A^N\dsum_{-2A/W\leq \h^{(1)},\h^{(2)}\leq 2A/W}\dsum_{\x^{(0)},\ldots,\x^{(k-2)}}\dsum_{\x^{(k-1)}\in B}e(W\alpha \Psi^{(1)}(\h^{(1)}-\h^{(2)},\x^{(k-1)})),
\end{equation*}
where $B$ is an $N$-dimensional set depending on $\h^{(1)},\h^{(2)},\x^{(0)}\ldots,\x^{(k-2)}$. We note from the definition of $B(\x^{(0)})$ that $B$ is a rectangular box of dimension $N$.
By change of variables $\y=\h^{(1)}-\h^{(2)}$ and by applying the triangle inequality, we deduce that
\begin{equation*}
\begin{aligned}
    G_2(\alpha)\ll A^N\dsum_{-4A/W\leq \y\leq 4A/W}\dsum_{\h^{(2)}}\dsum_{\x^{(0)},\ldots,\x^{(k-2)}}\biggl|\dsum_{\x^{(k-1)}\in B_1}e(W\alpha \Psi^{(1)}(\y,\x^{(k-1)}))\biggr|,
\end{aligned}
\end{equation*}
where $B_1:=B_1(\x^{(0)},\h^{(2)},\y,\x^{(1)},\ldots,\x^{(k-2)})\subseteq [-2A,2A]^N$ is a rectangular box (or the empty set). 

 For fixed $\y,\h^{(2)},\x^{(0)},\ldots,\x^{(k-2)},$ we observe that 
 \begin{equation*}
     \biggl|\dsum_{\x^{(k-1)}\in B_1}e(W\alpha \Psi^{(1)}(\y,\x^{(k-1)}))\biggr|=\biggl|\dsum_{\x^{(k-1)}\in B_1}e(W\alpha \Delta_{k-1}(P(\x^{(k-1)})))\biggr|,
 \end{equation*}
where $\Delta_{k-1}(P(\x^{(k-1)})):=\Delta_{k-1}(P(\x^{(k-1)});\y,\x^{(1)},\ldots,\x^{(k-2)}).$
Hence, by applying the Cauchy-Schwarz inequality, we find that
\begin{equation}\label{3.37}
    G_2(\alpha)\ll A^N((A/W)^{2N}A^{(k-1)N})^{1/2}G_3(\alpha)^{1/2},
\end{equation}
where
$$G_3(\alpha)=\dsum_{\x^{(0)},\h^{(2)}}\dsum_{\y}\dsum_{\x^{(1)},\ldots,\x^{(k-2)}}\biggl|\dsum_{\x^{(k-1)}\in B_1}e(W\alpha \Delta_{k-1}(P(\x^{(k-1)})))\biggr|^2.$$
Meanwhile, we note that variables $\x^{(0)}$ and $\h^{(2)}$ in $G_3(\alpha)$ are not associated with the argument of the exponential sum in the innermost sum, and are only involved in $B_1$. Consider an arbitrary injective function $\mathcal{F}$ mapping from $$(\x^{(0)},\h^{(2)})\in \left([-2A,2A]^N\times [-2A/W,2A/W]^N\right)\cap \Z^{2N}$$ to $$b\in [-(20A)^{2N}W^{-N},(20A)^{2N}W^{-N}]\cap \Z.$$ Furthermore, we define a set $ D_1:=D_1(b,\y,\x^{(1)},\ldots,\x^{(k-2)})\subseteq \Z^N$ in the following way. If $b$ is not in the image of $\mathcal{F}$, we define $D_1$ to be the empty set, and if $b$ is in the image of $\mathcal{F}$, we define
\begin{align*}
    D_1&:=D_1(b,\y,\x^{(1)},\ldots,\x^{(k-2)})\\
    &=D_1(\mathcal{F}(\x^{(0)},\h^{(2)}),\y,\x^{(1)},\ldots,\x^{(k-2)})\\
    &=B_1(\x^{(0)},\h^{(2)},\y,\x^{(1)},\ldots,\x^{(k-2)})\cap \Z^N.
\end{align*}
Then, one may write 
\begin{equation}\label{3.202020}
    G_3(\alpha)=\dsum_{ |b|\leq (20A)^{2N}W^{-N}}\dsum_{\y}\dsum_{\x^{(1)},\ldots,\x^{(k-2)}}\biggl|\dsum_{\x^{(k-1)}\in D_1}e(W\alpha \Delta_{k-1}(P(\x^{(k-1)})))\biggr|^2.
\end{equation}

Hence, on substituting $(\ref{3.202020})$ into $(\ref{3.37})$ and that into $(\ref{3.36})$, it follows by applying periodicity that
\begin{equation}\label{3.38}
\begin{aligned}
    \dint_0^1G_2(\alpha)^{2^{-2-k}}d\alpha&\ll A^{(k+3)2^{-3-k}N}W^{-2^{-2-k}N}\dint_0^1 G_3(\alpha)^{2^{-3-k}}d\alpha\\
    &=A^{(k+3)2^{-3-k}N}W^{-2^{-2-k}N}W^{-1}\dint_0^WG_3(W^{-1}\alpha)^{2^{-3-k}}d\alpha\\
    &=A^{(k+3)2^{-3-k}N}W^{-2^{-2-k}N}\dint_0^1G_3(W^{-1}\alpha)^{2^{-3-k}}d\alpha.
\end{aligned}
\end{equation}
Since $D_1$ is a rectangular box depending on $b,\y,\x^{(1)},\ldots,\x^{(k-2)}$, by applying Lemma $\ref{lem3.2}$ to $$\int_0^1G_3(W^{-1}\alpha)^{2^{-3-k}}d\alpha,$$ with $$l=k-1, B=(20A)^{2N}W^{-N}, A_1=4A/W, A_2=A_3=\cdots=A_k=2A, \sigma=2^{-3-k},$$ it follows that whenever $N\geq 18k(k-1)4^{k+3}$ one has
\begin{equation}\label{3.39}
    \dint_0^1G_3(W^{-1}\alpha)^{2^{-3-k}}d\alpha\ll (A^{k+3}/W^2)^{2^{-3-k}N}WA^{-k}.
\end{equation}
By substituting $(\ref{3.39})$ into $(\ref{3.38})$, one has
\begin{equation}\label{3.40}
    \dint_0^1G_2(\alpha)^{2^{-2-k}}d\alpha\ll A^{(k+3)2^{-2-k}N}W^{-2^{-1-k}N}WA^{-k}.
\end{equation}

Next, we turn to estimate $C(A,W)$ in $(\ref{3.35}).$ Recall the definition of $C(A,W),$ that is
\begin{equation}\label{3.41}
C(A,W)=W^{-N-1+2^{1-k}N}A^{2N-(k+3)2^{-k}N}\dsum_{1\leq r\leq W}G_1(r/W)^{2^{-1-k}}.    
\end{equation}
By splitting the summation over $r$ into the values of $d=(r,W),$ we see that
\begin{equation}\label{3.42}
\dsum_{1\leq r\leq W}G_1(r/W)^{2^{-1-k}}=\dsum_{d|W}\dsum_{(r,W)=W/d}G_1(r/W)^{2^{-1-k}}.
\end{equation}
Meanwhile, on recalling the definition of $G_1(r/W)$, it follows by applying the Cauchy-Schwarz inequality that
$$G_1(r/W)\ll (A/W)^N\dsum_{\h}\dsum_{\x^{(1)},\ldots,\x^{(k-1)}}\dsum_{\z^{(1)},\z^{(2)}\in \mathcal{B}^{(k-1)}}e((r/W)U_4(\z^{(1)}-\z^{(2)},\x^{(1)})).$$
Furthermore, if we use the definition ($\ref{7.1}$) of $\Gamma_P$ , we find by change of variables $\z=\z^{(1)}-\z^{(2)}$ that
\begin{equation*}
    G_1(r/W)\ll (A/W)^{2N}A^{N}\Gamma_P(r/W,A,\ldots,A).
\end{equation*}
Additionally, let us use the definition ($\ref{3.9}$) of $N(A_1,\ldots,A_{k-1},A_k^{-1};\beta).$ Then, whenever $(r,W)=W/d$, one infers from the argument in [$\ref{ref17}$, Lemma 3.2] that 
\begin{equation*}
    \Gamma_P(r/W,A,\ldots,A)\ll A^N(\log d)^N N(A,\ldots,A,A^{-1}; r/W).
\end{equation*}
Notice here that naive application of [$\ref{ref17}$, Lemma 3.2] gives $\log A$ in place of $\log d$ on the right-hand side, however, one readily sees from the fact that $r/W$ is a rational number with $(r,W)=W/d$ that the factor $\log A$ could be replaced by $\log d$.
Hence, since $P$ is a non-singular form, the arguments in [$\ref{ref20}$, Lemma 3.3, Lemma 3.4] using ($\ref{3.14}$) and $(\ref{3.63.63.6})$ with $A_1=\cdots=A_k=A$ and $\beta=r/W$ delivers that 
\begin{equation}\label{3.43}
G_1(r/W)\ll (A/W)^{2N}A^{(k+1)N}(\log d)^{N+1}d^{-N/(k-1)}.    
\end{equation}
On substituting $(\ref{3.43})$ into $(\ref{3.42})$, whenever $N\geq k2^{k+2},$ we deduce that
\begin{equation}\label{3.45}
\begin{aligned}
    \dsum_{1\leq r\leq W}G_1(r/W)^{2^{-1-k}}&\ll \dsum_{d|W}\dsum_{(r,W)=W/d}(A/W)^{2^{-k}N}A^{(k+1)2^{-1-k}N}(\log d)^{2^{-1-k}(N+1)}d^{-2^{-1-k}N/(k-1)}\\
   & \ll (A/W)^{2^{-k}N}A^{(k+1)2^{-1-k}N}\dsum_{1\leq d}d^{-2}\ll  (A/W)^{2^{-k}N}A^{(k+1)2^{-1-k}N}
\end{aligned}
\end{equation}
Then, substituting ($\ref{3.45}$) into ($\ref{3.41}$), one has
\begin{equation}\label{3.46}
    C(A,W)\ll W^{-N-1+2^{-k}N}A^{2N-(k+3)2^{-1-k}N}.
\end{equation}
Therefore, by substituting ($\ref{3.40}$) and ($\ref{3.46}$) into ($\ref{3.35}$), we conclude that
\begin{equation}\label{3.313.313.31}
\mathbf{N}(A)\ll A^{2N-2k}W^{1-N}.    
\end{equation}

Combining $(\ref{3.313.313.31})$ and $(\ref{3.34})$, we complete the proof of Theorem $\ref{thm3.1}.$
\end{proof}

\bigskip

\section{Minor arcs}\label{sec4}
Our goal in this section is to prove Proposition $\ref{prop4.1}$ below. In order to describe this proposition, we provide some definitions.
For any measurable set $\mathfrak{B}\subseteq[0,1)$, $\a\in \Z^N$ and $X>1$, define
\begin{equation}\label{4.1}
    \mathcal{I}_{\a}(X,\mathfrak{B})=\dint_{\mathfrak{B}}\dsum_{1\leq \x\leq X}e(\alpha f_{\a}(\x))d\alpha.
\end{equation}
For $B>0,$ define the major arcs 
\begin{equation}\label{4.24.2}
    \mathfrak{M}(B)=\bigcup_{\substack{0\leq a\leq q\leq B\\(q,a)=1}}\mathfrak{M}(q,a),
\end{equation}
where $$\mathfrak{M}(q,a)=\{\alpha\in [0,1)|\ |\alpha-a/q|\leq BA^{-1}X^{-d}\},$$ 
and define the minor arcs $\mathfrak{m}(B)=[0,1)\setminus \mathfrak{M}(B).$ Recall the definition ($\ref{def2.2}$) of $w$. Here and throughout, we abbreviate $\mathfrak{M}(w)$ and $\mathfrak{m}(w)$ simply to $\mathfrak{M}$ and $\mathfrak{m}.$

\begin{prop}\label{prop4.1}
Let $n$ and $d$ be natural numbers with $d\geq 2.$ Let $n_1$ be the greatest integer with $n_1\leq \lfloor(n-1)/2\rfloor/8.$  Suppose that $n_1>2d$ and $2X^d\leq A\leq X^{n_1-d}.$ Suppose that $P\in \Z[\x]$ is a non-singular form in $N_{d,n}$ variables of degree $k\geq 2.$ Then, whenever $N\geq 200k(k-1)2^{k-1},$ one has
\begin{equation}\label{in4.1}
\dsum_{\substack{\|\a\|_{\infty}\leq A\\ P(\a)=0}}\left|\mathcal{I}_{\a}(X,\mathfrak{m})\right|^2\ll A^{N-k-2}X^{2n-2d}(\log A)^{-1}.
\end{equation}
\end{prop}
We shall prove Proposition $\ref{prop4.1}$ at the end of this section. We temporarily pause here and provide the key idea in dealing with the constraint $P(\a)=0$ in the summation over $\a$ in $(\ref{in4.1})$. By orthogonality, we can represent 
 $$\dsum_{\substack{\|\a\|_{\infty}\leq A\\ P(\a)=0}}\left|\mathcal{I}_{\a}(X,\mathfrak{m})\right|^2$$
 by mean values of an associated exponential sum which is the shape of 
 \begin{equation}\label{4.44.4}
     \dsum_{\x}\dsum_{\|\a\|_{\infty}\leq A}e(\beta P(\a))e(F(\a,\x)),
 \end{equation}
 where $F(\a,\x)$ is a polynomial in $\x,$ and a linear function in $\a$. On noting that $F(\a,\x)$ is linear in $\a$, by applying the Cauchy-Schwarz inequality and Weyl differencing argument, we obtain an upper bound for $(\ref{4.44.4})$ in terms of
 \begin{equation}\label{4.54.5}
     \dsum_{\x}\dsum_{\a,\h}e(\beta(P(\a+\h)-P(\a)))e(F(\h,\x)).
 \end{equation}
 By changing the order of summations and applying the triangle inequality, the expression $(\ref{4.54.5})$ is bounded above by 
 \begin{equation}\label{4.64.6}
     \dsum_{\h}\biggl| \dsum_{\a}e(\beta(P(\a+\h)-P(\a)))\biggr|\biggl|\dsum_{\x}e(F(\h,\x))\biggr|.
 \end{equation}
Then, by applying the Cauchy-Schwarz inequality again, the expression $(\ref{4.64.6})$ is bounded above by
\begin{equation}\label{4.74.7}
     \left(\dsum_{\h}\biggl| \dsum_{\a}e(\beta(P(\a+\h)-P(\a)))\biggr|^2\right)^{1/2}\left(\dsum_{\h}\biggl|\dsum_{\x}e(F(\h,\x))\biggr|^2\right)^{1/2}.
 \end{equation}
Now, one sees that the first part in the expression $(\ref{4.74.7})$ includes all the information about $P(\a)=0$, and this allows us to deal with the constraint $P(\a)=0$ separately. We emphasize here that even if we replace the constraint $P(\a)=0$ by a more general constraint $P_1(\a)=P_2(\a)=\cdots=P_m(\a)=0$ with $P_i(\a)\in \Z[\a]$ for $1\leq i\leq m,$ we could follow the same procedure leading from $(\ref{4.44.4})$ to $(\ref{4.74.7})$ and deal with the constraint separately. Thus, we call this procedure `the separation procedure'.

\bigskip

At this point, recall the Veronese embedding and that we write $f_{\a}(\x)=\langle\a, \nu_{d,n}(\x)\rangle$ for a homogeneous polynomial in $n$ variables of degree $d$ with coefficients $\a$. We introduce some definitions useful in sections $\ref{sec4}$ and $\ref{sec5}$.

\begin{defn}\label{defn4.1}
Define $\boldsymbol{b}:=\boldsymbol{b}(f_{\boldsymbol{a}})=(b_1,\ldots,b_n)\in \Z^n$ and $\boldsymbol{c}:=\boldsymbol{c}(f_{\boldsymbol{a}})=(c_1,\ldots, c_{N-n})\in \Z^{N-n}$ to be vectors associated with the coefficients $\a$ of $f_{\a}(\x)$ such that
$b_k$ is the coefficient of $x_k^d$ with $k=1,\ldots,n$, and the coefficients $c_k$ are the coefficients of the remaining monomials in lexicographical order with $k=1,\ldots, N-n.$
\end{defn}

\begin{defn}\label{defn4.2}
Let $n$ and $d$ be natural numbers with $d\geq 2.$ Consider the monomials of degree $d$ in $n$ variables $x_1,\ldots, x_n$. In particular, the number of these monomials is $N=\binom{n+d-1}{d}.$ Then, define $v_d(\x)\in \R^n$ and $w_d(\x)\in \R^{N-n}$ to be vectors associated with those monomials such that $(v_d(\x))_i$ is $x_i^d$ with $i=1,\ldots,n$ and the polynomials $(w_d(\x))_j$ are the remaining monomials in lexicographical order with $j=1,\ldots, N-n$, respectively.
\end{defn}

For example, for $f_{\a}(\x)=a_1x_1^3+a_2x_1^2x_2+a_3x_1x_2^2+a_4x_2^3,$ we have $\b=(a_1,a_4)$ and $\c=(a_2,a_3).$ 
Furthermore, we find that 
$$v_3(x_1,x_2)=(x_1^3,x_2^3)\quad\ \text{and}\quad\ w_3(x_1,x_2)=(x_1^2x_2,x_1x_2^2).$$
Then, we notice that 
$$f_{\a}(\x)=\langle\b,v_d(\x)\rangle+\langle\c,w_d(\x)\rangle.$$

\bigskip

\subsection{The separation procedure}\label{subsec4.1}
Our goal in this subsection is to derive Proposition $\ref{pro4.3}$ below. This is the culmination of our work in this subsection. In order to describe Proposition $\ref{pro4.3}$, we recall the definition $(\ref{4.1})$ of $\mathcal{I}_{\a}(X,\mathfrak{B}),$ and define
\begin{equation}\label{4.4}
T(\alpha):=\dsum_{-A\leq b\leq A}\bigl|\dsum_{1\leq x\leq X}e(\alpha bx^d)\bigr|^2,    
\end{equation}
with $\alpha\in\R.$ 
\begin{prop}\label{pro4.3}
Let $n$ and $d$ be natural numbers. Suppose that $P\in \Z[\x]$ is a non-singular form in $N_{d,n}$ variables of degree $k\geq 2$. For any measurable set $\mathfrak{B}\subseteq[0,1)$, whenever $N\geq 200k(k-1)2^{k-1},$ we have
\begin{equation}
    \dsum_{\substack{\|\a\|_{\infty}\leq A\\ P(\a)=0}}\bigl|\mathcal{I}_{\a}(X,\mathfrak{B})\bigr|^2\ll A^{N-n/8-k}X^{7n/4}\dint_{\mathfrak{B}}T(\alpha_1)^{\lfloor n/2\rfloor/8}d\alpha_1\dint_{\mathfrak{B}}T(\alpha_2)^{\lceil n/2\rceil/8}d\alpha_2.
\end{equation}
\end{prop}
We shall prove Proposition $\ref{pro4.3}$ at the end of this subsection, by combining Lemma $\ref{lem4.4}$ and Lemma $\ref{lem4.5}$. To be specific, in Lemma $\ref{lem4.4},$ we bound
\begin{equation*}
    \dsum_{\substack{\|\a\|_{\infty}\leq A\\ P(\a)=0}}\left|\mathcal{I}_{\a}(X,\mathfrak{B})\right|^2
\end{equation*}
by mean values of two exponential sums, via applications of the Cauchy-Schwarz inequality in a way we mentioned in the explanation following Proposition $\ref{prop4.1}.$ 
From Lemma $\ref{lem4.5}$, we have an upper bound for one of these exponential sums, in terms of mean values of the exponential sum $T(\alpha)$. Furthermore, by Lemma $\ref{lem3.2},$ we have an upper bound for the other exponential sum. We will combine these estimates to prove Proposition $\ref{pro4.3}.$

\bigskip

In order to describe Lemma $\ref{lem4.4}$, it is convenient to introduce some definitions. For $\boldsymbol{h}\in \Z^N$, we define
$$I_{\boldsymbol{h}}=\{\a\in \Z^N|\ \|\a\|_{\infty}\leq A, \|\a+\boldsymbol{h}\|_{\infty}\leq A\}.$$
Furthermore, we define
\begin{equation*}
F_1(\alpha_1,\alpha_2,\boldsymbol{h})=\dsum_{\substack{1\leq \x,\y\leq X\\\x,\y\in \Z^n}}e(\alpha_1\langle\boldsymbol{h},\nu_{d,n}(\x)\rangle-\alpha_2\langle\boldsymbol{h},\nu_{d,n}(\y)\rangle)    
\end{equation*} and $$F_2(\beta,\boldsymbol{h})=\dsum_{\substack{\a\in I_{\boldsymbol{h}}}}e(\beta(P(\a+\boldsymbol{h})-P(\a))).$$
Additionally, we define
$$G_1(\alpha_1,\alpha_2)=\dsum_{\substack{\|\boldsymbol{h}\|_{\infty}\leq 2A\\ \boldsymbol{h}\in \Z^N}}|F_1(\alpha_1,\alpha_2,\boldsymbol{h})|^2\quad\ \text{and}\quad\ G_2(\beta)=\dsum_{\substack{\|\boldsymbol{h}\|_{\infty}\leq 2A\\ \boldsymbol{h}\in \Z^N}}|F_2(\beta,\boldsymbol{h})|^2.$$

\begin{lem}\label{lem4.4}
We have
\begin{equation}\label{ineq4.2}
     \dsum_{\substack{\|\a\|_{\infty}\leq A\\ P(\a)=0}}\left|\mathcal{I}_{\a}(X,\mathfrak{B})\right|^2\ll X^n\dint_{\mathfrak{B}^2}G_1(\alpha_1,\alpha_2)^{1/4}d\alpha_1d\alpha_2\dint_0^1G_2(\beta)^{1/4}d\beta.
\end{equation}

\end{lem}
\begin{proof}
Recall the definition of $F_1(\alpha_1,\alpha_2,\a)$. Then, by orthogonality, we have
\begin{equation}\label{4.2}
\begin{aligned}
     \dsum_{\substack{\|\a\|_{\infty}\leq A\\ P(\a)=0}}\left|\mathcal{I}_{\a}(X,\mathfrak{B})\right|^2= \dsum_{\substack{\|\a\|_{\infty}\leq A}}\dint_0^1\dint_{\mathfrak{B}^2}F_1(\alpha_1,\alpha_2,\a)e(\beta P(\a))d\alpha_1d\alpha_2d\beta
\end{aligned}
\end{equation}
By changing the order of summations together with the triangle inequality, one finds that
\begin{equation}\label{4.77}
   \dsum_{\substack{\|\a\|_{\infty}\leq A\\ P(\a)=0}}\left|\mathcal{I}_{\a}(X,\mathfrak{B})\right|^2 \leq \dint_0^1\dint_{\mathfrak{B}^2}F(\alpha_1,\alpha_2,\beta)d\alpha_1d\alpha_2d\beta,  
\end{equation}
where 
$$F(\alpha_1,\alpha_2,\beta)=\dsum_{\substack{1\leq \x,\y\leq X\\\x,\y\in \Z^n}}\biggl|\dsum_{\substack{\|\a\|_{\infty}\leq A}}e(\beta P(\a)+\alpha_1\langle\a,\nu_{d,n}(\x)\rangle-\alpha_2\langle\a,\nu_{d,n}(\y)\rangle)\biggr|.$$
We emphasize again that we preserve summation conditions until different conditions are specified. 

By applying the Cauchy-Schwarz inequality and a conventional Weyl differencing argument, one sees that
\begin{equation*}
    F(\alpha_1,\alpha_2,\beta)\leq (X^{2n})^{1/2}(\Xi(\alpha_1,\alpha_2,\beta))^{1/2},
\end{equation*}
where 
\begin{equation*}
\begin{aligned}
&\Xi(\alpha_1,\alpha_2,\beta)\\
&=\dsum_{\x,\y}\dsum_{\substack{\|\boldsymbol{h}\|_{\infty}\leq 2A\\\boldsymbol{h}\in \Z^N}}\dsum_{\a\in I_{\boldsymbol{h}}}e(\beta(P(\a+\boldsymbol{h})-P(\a))+\alpha_1\langle\boldsymbol{h},\nu_{d,n}(\x)\rangle-\alpha_2\langle\boldsymbol{h},\nu_{d,n}(\y)\rangle\rangle).       
\end{aligned}
\end{equation*}
By applying the triangle inequality, we have
$$F(\alpha_1,\alpha_2,\beta)\leq X^{n}\bigl(\dsum_{\boldsymbol{h}}|F_1(\alpha_1,\alpha_2,\boldsymbol{h})||F_2(\beta,\boldsymbol{h})|\bigr)^{1/2},$$
where 
$$F_1(\alpha_1,\alpha_2,\boldsymbol{h})=\dsum_{\x,\y}e(\alpha_1\langle\boldsymbol{h},\nu_{d,n}(\x)\rangle-\alpha_2\langle\boldsymbol{h},\nu_{d,n}(\y)\rangle)$$
and $$F_2(\beta,\boldsymbol{h})=\dsum_{\a\in I_{\boldsymbol{h}}}e(\beta(P(\a+\boldsymbol{h})-P(\a)))$$

By applying the Cauchy-Schwarz inequality again and recalling the definition of $G_1(\alpha_1,\alpha_2)$ and $G_2(\beta)$, we see that 
\begin{equation}\label{4.3}
    F(\alpha_1,\alpha_2,\beta)\leq X^{n}G_1(\alpha_1,\alpha_2)^{1/4}G_2(\beta)^{1/4}.
\end{equation}
Therefore, on substituting $(\ref{4.3})$ into $(\ref{4.77}),$ we complete the proof of Lemma $\ref{lem4.4}$.
\end{proof}

\bigskip

Next, Lemma $\ref{lem4.5}$ provides the upper bound for $G_1(\alpha_1,\alpha_2)$.
\begin{lem}\label{lem4.5}
We have
\begin{equation*}
    G_1(\alpha_1,\alpha_2)\ll A^{N-n/2}X^{3n}T(\alpha_1)^{\lfloor n/2\rfloor/2}T(\alpha_2)^{\lceil n/2\rceil/2}.
\end{equation*}
\end{lem}
\begin{proof}
Recall the definition $\b:=\b(f_{\a})$ of Definition $\ref{defn4.1}$. By squaring out and applying the triangle inequality, we find that 
\begin{equation*}
\begin{aligned}
    &G_1(\alpha_1,\alpha_2)\ll A^{N-n}\dsum_{\substack{1\leq \x^{(1)},\x^{(2)}\leq X\\1\leq\y^{(1)},\y^{(2)}\leq X}}\biggl|\dsum_{\substack{\|\b\|_{\infty}\leq 2A\\\b\in \Z^{n}}}e(\Psi(\alpha_1,\alpha_2,\b,\x^{(1)},\x^{(2)},\y^{(1)},\y^{(2)}))\biggr|,    
\end{aligned}
\end{equation*}
where 
$$\Psi(\alpha_1,\alpha_2,\b,\x^{(1)},\x^{(2)},\y^{(1)},\y^{(2)})=\alpha_1\langle\b,v_d(\x^{(1)})-v_d(\x^{(2)})\rangle-\alpha_2\langle\b,v_d(\y^{(1)})-v_d(\y^{(2)})\rangle.$$
By applying the Cauchy-Schwarz inequality and the triangle inequality, one sees that 
\begin{equation}\label{ineq4.6}
    G_1(\alpha_1,\alpha_2)\ll A^{N-n}(X^{4n})^{1/2}\bigl(A^nH(\alpha_1,\alpha_2)\bigr)^{1/2},
\end{equation}
where $$H(\alpha_1,\alpha_2)=\dsum_{\b}\biggl|\dsum_{\substack{\x^{(1)},\x^{(2)}\\\y^{(1)},\y^{(2)}}}e(\Psi(\alpha_1,\alpha_2,\b,\x^{(1)},\x^{(2)},\y^{(1)},\y^{(2)}))\biggr|.$$

We note that 
\begin{equation*}
\begin{aligned}
H(\alpha_1,\alpha_2)=\biggl(\dsum_{-A\leq b\leq A}\biggl|\dsum_{1\leq x\leq X}e(\alpha_1bx^d)\biggr|^2\biggl|\dsum_{1\leq y\leq X}e(\alpha_2 by^d)\biggr|^2\biggr)^{n}.
\end{aligned}
\end{equation*}
Then, on recalling the definition $(\ref{4.4})$ of $T(\alpha)$ and applying the trivial bounds that
$$\dsum_{1\leq x\leq X}e(\alpha_1bx^d)\leq X\ \textrm{and}\ \dsum_{1\leq y\leq X}e(\alpha_2 by^d)\leq X,$$ we see that
\begin{equation}\label{4.8}
    H(\alpha_1,\alpha_2)\leq X^{2n}T(\alpha_1)^{\lfloor n/2\rfloor}T(\alpha_2)^{\lceil n/2\rceil}.
\end{equation}
On substituting ($\ref{4.8}$) into $(\ref{ineq4.6})$, we complete the proof of Lemma $\ref{lem4.5}.$
\end{proof}

\bigskip

\bigskip

\begin{proof}[Proof of Proposition $\ref{pro4.3}$]
Recall the definition of $G_2(\beta)$ and $I_{\h}$ leading to Lemma $\ref{lem4.4}$. Since $I_{\h}$ is a rectangular box depending on $\h$, it follows by Lemma $\ref{lem3.2}$ with $l=1$, $B=0,$ $A_1=A_2=2A$ and $\sigma=1/4$ that whenever $N\geq 200k(k-1)2^{k-1},$ one has $$\int_0^1G_2(\beta)^{1/4}d\beta\ll A^{3N/4-k}.$$
Then, on substituting this estimate and the bound for $G_1(\alpha_1,\alpha_2)$ obtained in Lemma $\ref{lem4.5}$ into ($\ref{ineq4.2}$), we readily conclude that 
$$   \dsum_{\substack{\|\a\|_{\infty}\leq A\\ P(\a)=0}}\bigl|\mathcal{I}_{\a}(X,\mathfrak{B})\bigr|^2\ll A^{N-n/8-k}X^{7n/4}\dint_{\mathfrak{B}}T(\alpha_1)^{\lfloor n/2\rfloor/8}d\alpha_1\dint_{\mathfrak{B}}T(\alpha_2)^{\lceil n/2\rceil/8}d\alpha_2.$$ Thus, we complete the proof of Proposition $\ref{pro4.3}$. 
\end{proof}

\bigskip

\subsection{Bounds for large moduli}\label{subsec4.2}
In this subsection, for given $\delta>0$, we derive the bound for
\begin{equation*}
     \dsum_{\substack{\|\a\|_{\infty}\leq A\\ P(\a)=0}}\left|\mathcal{I}_{\a}(X,\mathfrak{m}(X^{\delta}))\right|^2.
\end{equation*}

To derive a bound for this quantity, we make use of Proposition $\ref{pro4.3}$ and Lemma $\ref{lem3.23.2}.$

\bigskip

\begin{lem}\label{lem4.6}
Let $\delta$ be a positive number with $\delta<1/6$. Let $n$ and $d$ be natural numbers with $d\geq 2.$ Let $n_1$ be the greatest integer with $n_1\leq\lfloor (n-1)/2\rfloor/8.$ Suppose that $n_1>2d$ and $2X^d\leq A\leq X^{n_1-d}.$ Suppose that $P\in \Z[\x]$ is a non-singular form in $N_{d,n}$ variables of degree $k\geq 2.$ Then, whenever $N\geq 200k(k-1)2^{k-1},$ there exists $\delta'>0$ such that
\begin{equation*}
     \dsum_{\substack{\|\a\|_{\infty}\leq A\\ P(\a)=0}}\left|\mathcal{I}_{\a}(X,\mathfrak{m}(X^{\delta}))\right|^2\ll A^{N-k-2}X^{2n-2d}A^{-\delta'}.
\end{equation*}
In particular, the exponent $\delta'$ depends on $\delta,n$ and $d.$
\end{lem}

\begin{proof}
By Proposition $\ref{pro4.3}$, we obtain the bound
\begin{equation}\label{4.9}
\begin{aligned}
     & \dsum_{\substack{\|\a\|_{\infty}\leq A\\ P(\a)=0}}\bigl|\mathcal{I}_{\a}(X,\mathfrak{m}(X^{\delta}))\bigr|^2\\
     &\ll A^{N-n/8-k}X^{7n/4}\dint_{\mathfrak{m}(X^{\delta})}T(\alpha_1)^{\lfloor n/2\rfloor/8}d\alpha_1\dint_{\mathfrak{m}(X^{\delta})}T(\alpha_2)^{\lceil n/2\rceil/8}d\alpha_2
\end{aligned}
\end{equation}

For simplicity, we write 
\begin{equation*}
\begin{aligned}
    m_1=\lfloor (n-1)/2\rfloor/8\ \text{and}\    m_2=\lceil (n-1)/2\rceil/8,
\end{aligned}
\end{equation*}
and write $n_i$ for the greatest integer less than $m_i$ for $i=1$ and $2$, respectively. Then, on noting that $
 n/8=m_1+m_2+1/8$,
one finds from ($\ref{4.9}$) that
\begin{equation}\label{4.12}
     \dsum_{\substack{\|\a\|_{\infty}\leq A\\ P(\a)=0}}\bigl|\mathcal{I}_{\a}(X,\mathfrak{m}(X^{\delta}))\bigr|^2\ll  A^{N-n/8-k}X^{7n/8}\cdot\mathcal{L}_1\cdot\mathcal{L}_2\cdot\sup_{\alpha\in \mathfrak{m}(X^{\delta})}{T}(\alpha)^{\sigma},
\end{equation}
where 
\begin{equation*}
\begin{aligned}
  &\sigma=m_1+m_2-n_1-n_2+1/8  \\
  &\mathcal{L}_1=\dint_0^1{T}(\alpha_1)^{n_1}d\alpha_1\\
  &\mathcal{L}_2=\dint_0^1{T}(\alpha_2)^{n_2}d\alpha_2
\end{aligned}
\end{equation*}

Recall the hypothesis $2X^d\leq A\leq X^{n_1-d}$ in the statement of Lemma $\ref{lem4.6}$. Then, we find from orthogonality and $(\ref{3.5})$ that
\begin{equation}\label{4.13}
   \mathcal{L}_1\ll A^{n_1-1}X^{2n_1-d+\epsilon}.
\end{equation}
Similarly, one finds that
\begin{equation}\label{4.14}
     \mathcal{L}_2\ll A^{n_2-1}X^{2n_2-d+\epsilon}.
\end{equation}
Furthermore, it follows by ($\ref{3.4}$) that 
\begin{equation}\label{4.15}
    \sup_{\alpha\in \mathfrak{m}(X^{\delta})}{T}(\alpha)^{\sigma}\ll (AX^2)^{\sigma+\epsilon}X^{-2^{2-k}\sigma\delta }.
\end{equation}
On substituting $(\ref{4.13})$, $(\ref{4.14})$ and $(\ref{4.15})$ into $(\ref{4.12})$, we conclude that
\begin{equation*}
      \dsum_{\substack{\|\a\|_{\infty}\leq A\\ P(\a)=0}}\bigl|\mathcal{I}_{\a}(X,\mathfrak{m}(X^{\delta}))\bigr|^2\ll A^{N-k-2}X^{2n-2d}X^{-\delta''},
\end{equation*}
where $\delta''$ is a positive number with $\delta''<2^{2-k}\sigma\delta.$ Hence, on recalling that $A\leq X^{n_1-d}$, we complete the proof of Lemma $\ref{lem4.6}.$
\end{proof}

\bigskip

\subsection{Bounds for small moduli}\label{subsec4.3}
In this section, we derive the bound for
\begin{equation*}
   \dsum_{\substack{\|\a\|_{\infty}\leq A\\ P(\a)=0}}\left|\mathcal{I}_{\a}(X,\mathfrak{M}(X^{\delta})\setminus \mathfrak{M}(\log X))\right|^2.
\end{equation*}

To derive a bound for this quantity, we use the standard treatments of exponential sums in major arcs and pruning arguments.

\bigskip

\begin{lem}\label{lem4.7}
 With the same notation in Lemma $\ref{lem4.6}$, suppose that $n_1>2d$, $2X^d\leq A\leq X^{n_1-d}$ and $0<\delta<1/6$. Suppose that $n_1>2d$ and $2X^d\leq A\leq X^{n_1-d}.$ Suppose that $P\in \Z[\x]$ is a non-singular form in $N_{d,n}$ variables of degree $k\geq 2.$ Then, whenever $N\geq 200k(k-1)2^{k-1},$ one has
\begin{equation*}
     \dsum_{\substack{\|\a\|_{\infty}\leq A\\ P(\a)=0}}\left|\mathcal{I}_{\a}(X,\mathfrak{M}(X^{\delta})\setminus \mathfrak{M}(\log X))\right|^2\ll A^{N-k-2}X^{2n-2d}(\log A)^{-1}.
\end{equation*}
\end{lem}

\begin{proof}
On observing that 
$$\mathfrak{M}(X^{\delta})\setminus \mathfrak{M}(\log X)=\bigcup_{j=0}^{J}\mathfrak{M}(2^{j+1}\log X)\setminus\mathfrak{M}(2^j\log X),$$
with $J=O(\log X),$ we deduce by the Cauchy-Schwarz inequality that
\begin{equation}\label{4.16}
    \begin{aligned}
     & \dsum_{\substack{\|\a\|_{\infty}\leq A\\ P(\a)=0}}\left|\mathcal{I}_{\a}(X,\mathfrak{M}(X^{\delta})\setminus \mathfrak{M}(\log X))\right|^2\\
     &\ll J\dsum_{j=0}^J \dsum_{\substack{\|\a\|_{\infty}\leq A\\ P(\a)=0}}\left|\mathcal{I}_{\a}(X,\mathfrak{M}(2^{j+1}\log X)\setminus\mathfrak{M}(2^j\log X))\right|^2\\
    \end{aligned}
\end{equation}

Now, we analyze the mean value
$$ \dsum_{\substack{\|\a\|_{\infty}\leq A\\ P(\a)=0}}\left|\mathcal{I}_{\a}(X,\mathfrak{M}(2Q)\setminus\mathfrak{M}(Q))\right|^2,$$
with $\log X\leq Q\leq X^{\delta}.$ For simplicity, we temporarily write
\begin{equation}
    \mathfrak{C}=\mathfrak{M}(2Q)\setminus\mathfrak{M}(Q).
\end{equation}
Then, by Proposition $\ref{pro4.3}$, we deduce that
\begin{equation}\label{4.18}
     \dsum_{\substack{\|\a\|_{\infty}\leq A\\ P(\a)=0}}\left|\mathcal{I}_{\a}(X,\mathfrak{C})\right|^2\ll A^{N-n/8-k}X^{7n/4}\dint_{\mathfrak{C}}T(\alpha_1)^{\lfloor n/2\rfloor/8}d\alpha_1\dint_{\mathfrak{C}}T(\alpha_2)^{\lceil n/2\rceil/8}d\alpha_2.
\end{equation}

Meanwhile, whenever $\alpha\in\mathfrak{C},$ there exist $q\in \N$ and $a\in \Z$ with $(q,a)=1$ such that $Q\leq q\leq 2Q$ and
$$|\alpha-a/q|\leq 2Q(AX^d)^{-1}.$$
Also, for $b\in \Z\setminus \{0\}$ with $|b|\leq A$, if we write $l=(q,b)$, $\widetilde{q}=q/l$ and $\widetilde{b}=b/l$, we have
\begin{equation*}
    \left|b\alpha-\frac{a\widetilde{b}}{\widetilde{q}}\right|\leq \frac{2Qb}{AX^d}\leq \frac{2Q}{X^d}.
\end{equation*}
Hence, since $\widetilde{q}\leq 2Q$, it follows by [$\ref{ref15}$, Lemma 2.7] that
\begin{equation}\label{4.19}
    \dsum_{1\leq x\leq X}e(b\alpha x^d)-(\widetilde{q})^{-1}S(\widetilde{q},a\widetilde{b})v(\beta)=O((2Q)^2),
\end{equation}
where $S(q,a)=\dsum_{n=1}^qe(an^d/q)$ and $v(\beta)=\dint_0^Xe(\beta \gamma^d)d\gamma$ with $\beta=b\alpha-\frac{a\widetilde{b}}{\widetilde{q}}.$ Thus, when $\alpha\in \mathfrak{C}$, we deduce from ($\ref{4.19}$) that 
\begin{equation}\label{4.20}
\begin{aligned}
    T(\alpha)&\ll X^2+\dsum_{l|q}\dsum_{\substack{|b|\leq A\\ (q,b)=l}}\biggl|\dsum_{1\leq x\leq X}e(b\alpha x^d)\biggr|^2\\
    &\ll X^2 +\dsum_{l|q}\dsum_{\substack{|b|\leq A\\ (q,b)=l}}
    \left(|((\widetilde{q})^{-1}S(\widetilde{q},a\widetilde{b})v(\beta)|^2+(2Q)^4\right).
\end{aligned}
\end{equation}

By [$\ref{ref15}$, Theorem 4.2], we have a bound $S(q/l,a(b/l))\ll (q/l)^{1-1/d}$
and a trivial bound $v(\beta)\leq X$. Hence, on substituting these estimates into $(\ref{4.20})$, we find that
\begin{equation}\label{4.21}
    T(\alpha)\ll X^2+\dsum_{l|q}AX^2q^{-2/d}l^{2/d-1}+AQ^4.
\end{equation}
The second term $\sum_{l|q}AX^2q^{-2/d}l^{2/d-1}$ is bounded above by $AX^2q^{-2/d}\sum_{l|q}1,$ and by the standard divisor estimate, this bound is $O(AX^2q^{-2/d+\epsilon})$. Recall that $\alpha\in \mathfrak{C}$, and thus we have the bound $Q\leq q\leq 2Q$ with $\log X\leq Q\leq X^{\delta}.$ Furthermore, we recall the hypothesis $X^d\leq A$ in the statement in Lemma $\ref{lem4.7}.$ Hence, it follows from $(\ref{4.21})$ that 
\begin{equation}\label{4.22}
    T(\alpha)\ll AX^2Q^{-2/d+\epsilon}.
\end{equation}
Therefore, on substituting ($\ref{4.22}$) into ($\ref{4.18}$), we obtain the bound
\begin{equation*}
     \dsum_{\substack{\|\a\|_{\infty}\leq A\\ P(\a)=0}}\left|\mathcal{I}_{\a}(X,\mathfrak{C})\right|^2\ll A^{N-n/8-k}X^{7n/4}\textrm{mes}(\mathfrak{C})^2(AX^2Q^{-2/d+\epsilon})^{n/8}.
\end{equation*}
On noting that $\textrm{mes}(\mathfrak{C})\ll Q^3(AX^d)^{-1}$, we conclude that
\begin{equation}\label{4.23}
      \dsum_{\substack{\|\a\|_{\infty}\leq A\\ P(\a)=0}}\left|\mathcal{I}_{\a}(X,\mathfrak{C})\right|^2\ll A^{N-k-2}X^{2n-2d}Q^{6-n/(4d)+\epsilon}.
\end{equation}
Recall that $n_1$ is the greatest integer with $n_1\leq\lfloor (n-1)/2\rfloor/8.$ We find from ($\ref{4.23}$) together with the hypothesis $n_1>2d$ in the statement of Lemma $\ref{lem4.7}$ that
$$  \dsum_{\substack{\|\a\|_{\infty}\leq A\\ P(\a)=0}}\left|\mathcal{I}_{\a}(X,\mathfrak{C})\right|^2\ll A^{N-k-2}X^{2n-2d}Q^{-2}.$$
Hence, on recalling that $J=O(\log X)$, it follows from ($\ref{4.16}$) that
\begin{equation*}
    \dsum_{\substack{\|\a\|_{\infty}\leq A\\ P(\a)=0}}\left|\mathcal{I}_{\a}(X,\mathfrak{M}(X^{\delta})\setminus \mathfrak{M}(\log X))\right|^2\ll A^{N-k-2}X^{2n-2d}(\log X)^{-1}.
\end{equation*}
On noting that $A\leq X^{n_1-d}$ in the statement of Lemma $\ref{lem4.7}$, we complete the proof of Lemma $\ref{lem4.7}.$

\end{proof}

\bigskip

\begin{proof}[Proof of Proposition $\ref{prop4.1}$]
Note that 
\begin{equation*}
    \mathfrak{m}(\log X)=(\mathfrak{M}(X^{\delta})\setminus\mathfrak{M}(\log X))\cup \mathfrak{m}(X^{\delta}).
\end{equation*}
Then, by the elementary inequality $(a+b)^2\leq 2a^2+2b^2,$ we find that
\begin{equation*}
\begin{aligned}
    & \dsum_{\substack{\|\a\|_{\infty}\leq A\\ P(\a)=0}}\left|\mathcal{I}_{\a}(X,\mathfrak{m}(\log X))\right|^2\\
    &\ll    \dsum_{\substack{\|\a\|_{\infty}\leq A\\ P(\a)=0}}\left|\mathcal{I}_{\a}(X,\mathfrak{M}(X^{\delta})\setminus \mathfrak{M}(\log X))\right|^2+   \dsum_{\substack{\|\a\|_{\infty}\leq A\\ P(\a)=0}}\left|\mathcal{I}_{\a}(X,\mathfrak{m}(X^{\delta}))\right|^2.
\end{aligned}
\end{equation*}
Thus, by Lemma $\ref{lem4.6}$ and Lemma $\ref{lem4.7}$, we conclude that
\begin{equation*}
    \dsum_{\substack{\|\a\|_{\infty}\leq A\\ P(\a)=0}}\left|\mathcal{I}_{\a}(X,\mathfrak{m}(\log X))\right|^2\ll  A^{N-k-2}X^{2n-2d}(\log A)^{-1}.
\end{equation*}
Therefore, we complete the proof of Proposition $\ref{prop4.1}.$
\end{proof}

\bigskip

\section{Proof of Theorem 1.1}\label{sec5}
In this section, we provide the proof of Theorem $\ref{thm2.2}$. To do this, we require two auxiliary lemmas. These lemmas reveal that $\mathfrak{S}_{\a}^*$ and $\mathfrak{J}_{\a}^*$, defined in the introduction, behave in a similar manner to the truncated singular series and the truncated singular integral on average over $\a$ with $\|\a\|_{\infty}\leq A$ and $P(\a)=0$. It is worth noting that in the proofs of these lemmas, we make use of the same idea as that in Section $\ref{sec4}$ (the separation procedure). In advance of the statement of the first lemma of these, it is convenient to define the exponential sum $S_{\a}(q)$ with $\a\in \Z^N$, $q\in \N$ by
\begin{equation*}
    S_{\a}(q):=S_{\a}(q;n)=q^{-n}\dsum_{\substack{1\leq b\leq q\\ (q,b)=1}}\dsum_{\substack{1\leq \r\leq q\\\r\in \Z^n}}e\left(\frac{b}{q}f_{\a}(\r)\right).
\end{equation*}
\begin{lem}\label{lem5.15.1}
Let $n$ and $d$ be natural numbers. Suppose that $A,B,C$ are sufficiently large positive numbers with $B<C.$ Suppose that $P\in \Z[\x]$ is a non-singular form in $N_{d,n}$ variables of degree $k\geq 2.$ Then, for any set $\mathcal{C}\subseteq [B,C]\cap \Z$, whenever $N\geq 200k(k-1)2^{k-1}$  we have 
\begin{equation*}
    \dsum_{\substack{\|\a\|_{\infty}\leq A\\ P(\a)=0}}\biggl|\dsum_{q\in \mathcal{C}}S_{\a}(q)\biggr|^2\ll A^{N-k}\biggl(\dsum_{q\in \mathcal{C}}q^{1+\epsilon}(q^{-1}+q^{-4/d})^{n/16}\biggr)^2.
\end{equation*}
\end{lem}
\begin{proof}
We deduce by applying orthogonality and changing the order of summations that
\begin{equation}\label{6.16.1}
\begin{aligned}
      \dsum_{\substack{\|\a\|_{\infty}\leq A\\ P(\a)=0}}\biggl|\dsum_{q\in \mathcal{C}}S_{\a}(q)\biggr|^2
      &=\dsum_{\substack{\|\a\|_{\infty}\leq A}}\dint_0^1\dsum_{q_1,q_2\in \mathcal{C}}S_{\a}(q_1)\overline{S_{\a}(q_2)}e(\beta P(\a))d\beta\\
      &=q_1^{-n}q_2^{-n}\dsum_{q_1,q_2\in \mathcal{C}}\dsum_{\substack{1\leq b_1\leq q_1\\ 1\leq b_2\leq q_2\\ (q_1,b_1)=(q_2,b_2)=1}}\dint_0^1\dsum_{\substack{1\leq \r_1\leq q_1\\1\leq \r_2\leq q_2\\\r_1,\r_2\in \Z^n}} T_{\a}\left(\beta,\frac{b_1}{q_1},\frac{b_2}{q_2},\r_1,\r_2\right)  d\beta,
\end{aligned}
\end{equation}
where 
$$T_{\a}\left(\beta,\frac{b_1}{q_1},\frac{b_2}{q_2},\r_1,\r_2\right)=\dsum_{\substack{\|\a\|_{\infty}\leq A}}e(\beta P(\a))e\left(\frac{b_1}{q_1}f_{\a}(\r_1)-\frac{b_2}{q_2}f_{\a}(\r_2)\right).$$

We now investigate the integrand in ($\ref{6.16.1}$), that is
$$\dsum_{\substack{1\leq \r_1\leq q_1\\1\leq \r_2\leq q_2}} T_{\a}\left(\beta,\frac{b_1}{q_1},\frac{b_2}{q_2},\r_1,\r_2\right). $$
In order to obtain the upper bound for this quantity, we
define 
$$F_1\left(\frac{b_1}{q_1},\frac{b_2}{q_2},\h\right)=\dsum_{\substack{1\leq \r_1\leq q_1\\1\leq\r_2\leq q_2}}e\left(\frac{b_1}{q_1}f_{\h}(\r_1)-\frac{b_2}{q_2}f_{\h}(\r_2)\right)$$
and
$$F_2\bigl(\beta,\h\bigr)=\dsum_{\a\in I_{\h}}e(\beta(P(\a+\h)-P(\a)))$$
in which 
$$I_{\h}=\{\a\in \Z^N|\ \|\a\|_{\infty}\leq A,\ \|\a+\h\|_{\infty}\leq A\}.$$ Then,
by applying the Cauchy-Schwarz inequality and a conventional Weyl differencing together with the triangle inequality, one deduces that
\begin{equation*}
\begin{aligned}
    \dsum_{\substack{ \r_1,\r_2}} T_{\a}\left(\beta,\frac{b_1}{q_1},\frac{b_2}{q_2},\r_1,\r_2\right)\ll q_1^{n/2}q_2^{n/2}\biggl(\dsum_{\substack{\|\h\|_{\infty}\leq 2A\\\h\in\Z^N}}\biggl|F_1\biggl(\frac{b_1}{q_1},\frac{b_2}{q_2},\h\biggr)\biggr|\biggl|F_2\bigl(\beta,\h\bigr)\biggr|\biggr)^{1/2}.
\end{aligned}
\end{equation*}
If we further define
$$G_1\biggl(\frac{b_1}{q_1},\frac{b_2}{q_2}\biggr)=\dsum_{\|\h\|_{\infty}\leq 2A}\biggl|F_1\biggl(\frac{b_1}{q_1},\frac{b_2}{q_2},\h\biggr)\biggr|^2\ \text{and}\ G_2(\beta)=\dsum_{\|\h\|_{\infty}\leq 2A}\biggl|F_2\bigl(\beta,\h\bigr)\biggr|^2,$$
 we find by applying the Cauchy-Schwarz inequality again that
\begin{equation}\label{6.26.2}
     \dsum_{\substack{\r_1,\r_2}} T_{\a}\left(\beta,\frac{b_1}{q_1},\frac{b_2}{q_2},\r_1,\r_2\right)\ll q_1^{n/2}q_2^{n/2}\left(G_1\biggl(\frac{b_1}{q_1},\frac{b_2}{q_2}\biggr)\right)^{1/4}\left(G_2(\beta)\right)^{1/4}.
\end{equation}

We first analyse $G_1\bigl(\frac{b_1}{q_1},\frac{b_2}{q_2}\bigr).$ Recall Definition $\ref{defn4.2}$. Furthermore, for $\d,\r_1^{(1)},\r_1^{(2)},\r_2^{(1)},\r_2^{(2)}\in \Z^n$, we define 
\begin{equation}
\begin{aligned}
   & \Psi(\d,\r_1^{(1)},\r_1^{(2)},\r_2^{(1)},\r_2^{(2)})\\
    &:=\Psi\left(\d,\r_1^{(1)},\r_1^{(2)},\r_2^{(1)},\r_2^{(2)};\frac{b_1}{q_1},\frac{b_2}{q_2},\right)\\
    &= \frac{b_1}{q_1}(\langle\d,v_{d}(\r_1^{(1)})\rangle-\langle\d,v_{d}(\r_1^{(2)})\rangle)-\frac{b_2}{q_2}(\langle\d,v_{d}(\r_2^{(1)})\rangle-\langle\d,v_{d}(\r_2^{(2)})\rangle).
\end{aligned}
\end{equation} 
Then, by squaring out and applying the triangle inequality, we deduce that
\begin{equation}\label{6.36.3}
\begin{aligned}
    G_1\biggl(\frac{b_1}{q_1},\frac{b_2}{q_2}\biggr)\ll A^{N-n}\dsum_{\substack{1\leq \r_1^{(1)},\r_1^{(2)}\leq q_1\\1\leq\r_2^{(1)},\r_2^{(2)}\leq q_2}}\biggl|\dsum_{\substack{\|\d\|_{\infty}\leq 2A\\ \d\in \Z^n}}e\left(\Psi(\d,\r_1^{(1)},\r_1^{(2)},\r_2^{(1)},\r_2^{(2)})\right)\biggr|.    
\end{aligned}
\end{equation}
By applying the Cauchy-Schwarz inequality together with the triangle inequality and by changing the order of summations, one sees from ($\ref{6.36.3}$) that
\begin{equation*}
\begin{aligned}
    G_1\biggl(\frac{b_1}{q_1},\frac{b_2}{q_2}\biggr)\ll A^{N-n}\cdot  q_1^{n}q_2^n  \cdot \biggl(\dsum_{\substack{\|\d_1\|_{\infty}\leq 2A\\\|\d_2\|_{\infty}\leq 2A}}\biggl|\dsum_{ \substack{1\leq\r_1^{(1)},\r_1^{(2)}\leq q_1\\1\leq \r_2^{(1)},\r_2^{(2)}\leq q_2}}e\left(\Psi(\d_1-\d_2,\r_1^{(1)},\r_1^{(2)},\r_2^{(1)},\r_2^{(2)})\right)\biggr|\biggr)^{1/2}.
\end{aligned}
\end{equation*}
By change of variables $\d=\d_1-\d_2$, we have
\begin{equation*}
\begin{aligned}
    G_1\biggl(\frac{b_1}{q_1},\frac{b_2}{q_2}\biggr)\ll A^{N-n}\cdot  q_1^{n}q_2^n  \cdot \biggl(\dsum_{\substack{\|\d\|_{\infty}\leq 4A}}A^n\biggl|\dsum_{ \r_1^{(1)},\r_1^{(2)},\r_2^{(1)},\r_2^{(2)}}e\left(\Psi(\d,\r_1^{(1)},\r_1^{(2)},\r_2^{(1)},\r_2^{(2)})\right)\biggr|\biggr)^{1/2}.
\end{aligned}
\end{equation*}
On writing
$$H\left(\alpha,q\right)=\dsum_{1\leq r\leq q}e\left(\alpha r^d\right),$$
we find that
\begin{equation*}
\begin{aligned}
   G_1\biggl(\frac{b_1}{q_1},\frac{b_2}{q_2}\biggr)\ll A^{N-n/2}\cdot q_1^{n}q_2^n \cdot \biggl(\dsum_{|f|\leq 4A}\biggl|H\left(\frac{b_1f}{q_1},q_1\right)\biggr|^2\biggl|H\left(\frac{b_2f}{q_2},q_2\right)\biggr|^2\biggr)^{n/2}.
\end{aligned}
\end{equation*}
By applying the Cauchy-Schwarz inequality again, one obtains 
\begin{equation}\label{6.5}
     G_1\biggl(\frac{b_1}{q_1},\frac{b_2}{q_2}\biggr)\ll A^{N-n/2}\cdot q_1^{n}q_2^{n}\cdot I_1^{n/4}\cdot I_2^{n/4},
\end{equation}
where 
\begin{equation*}
    I_1=\dsum_{|f|\leq 4A}\left|H\left(\frac{b_1f}{q_1},q_1\right)\right|^4\ \text{and}\  I_2=\dsum_{|f|\leq 4A}\left|H\left(\frac{b_2f}{q_2},q_2\right)\right|^4.
\end{equation*}
By splitting sum over $f$ in terms of values of $(f,q_1),$ we see that
\begin{equation}\label{6.6}
I_1=\dsum_{e|q_1}\dsum_{\substack{(f,q_1)=q_1/e\\|f|\leq 4A}}\left|H\left(\frac{b_1f}{q_1},q_1\right)\right|^4.    
\end{equation}
Meanwhile, when $(f,q_1)=q_1/e$, on writing that $f=(q_1/e)\widetilde{f}$ with $(\widetilde{f},e)=1,$ it follows from [$\ref{ref15}$, Theorem 4.2] that
\begin{equation}\label{6.7}
\begin{aligned}
H\left(\frac{b_1f}{q_1},q_1\right)=\dsum_{1\leq r\leq q_1}e\left(\frac{b_1f}{q_1} r^d\right)=\frac{q_1}{e}\cdot \dsum_{1\leq r\leq e}e\left(\frac{b_1\widetilde{f}}{e}r^d\right)\ll q_1\cdot e^{-1/d}.        
\end{aligned}
\end{equation}
Hence, on substituting $(\ref{6.7})$ into $(\ref{6.6})$, one has
\begin{equation}\label{6.8}
    I_1\ll \dsum_{e|q_1}(Ae/q_1)(q_1e^{-1/d})^4\leq Aq_1^{3}\dsum_{e|q_1}e^{1-4/d}  \leq  Aq_1^{4+\epsilon}(q_1^{-1}+q_1^{-4/d}).
\end{equation}
Similarly, we have
\begin{equation}\label{6.9}
    I_2\ll Aq_2^{4+\epsilon}(q_2^{-1}+q_2^{-4/d}).
\end{equation}
Then, on substituting $(\ref{6.8})$ and $(\ref{6.9})$ into $(\ref{6.5})$, we obtain 
\begin{equation}\label{6.10}
     G_1\biggl(\frac{b_1}{q_1},\frac{b_2}{q_2}\biggr)\ll A^N\cdot q_1^{2n+\epsilon}q_2^{2n+\epsilon}\cdot (q_1^{-1}+q_1^{-4/d})^{n/4}(q_2^{-1}+q_2^{-4/d})^{n/4}.
\end{equation}

Next, recall the definition of $G_2(\beta).$ Since the range of summation $I_{\h}$ in the definition of $G_2(\beta)$ is a rectangular box depending on $\h,$ it follows by Lemma $\ref{lem3.2}$ with $l=1,\ B=0,\ A_1=A_2=2A$ and $\sigma=1/4$ that
\begin{equation}\label{6.11}
    \dint_0^1G_2(\beta)^{1/4}d\beta\ll A^{3N/4-k}.
\end{equation}
On substituting $(\ref{6.10})$ into ($\ref{6.26.2}$) and that into $(\ref{6.16.1})$, we have
\begin{equation}\label{6.12}
\begin{aligned}
      &\dsum_{\substack{\|\a\|_{\infty}\leq A\\ P(\a)=0}}\biggl|\dsum_{q\in \mathcal{C}}S_{\a}(q)\biggr|^2\\
      &\ll \dsum_{q_1,q_2\in \mathcal{C}}\dsum_{\substack{1\leq b_1\leq q_1\\ 1\leq b_2\leq q_2\\ (q_1,b_1)=(q_2,b_2)=1}} A^{N/4}\cdot q_1^{\epsilon}q_2^{\epsilon}\cdot(q_1^{-1}+q_1^{-4/d})^{n/16}(q_2^{-1}+q_2^{-4/d})^{n/16}\cdot \dint_0^1G_2(\beta)^{1/4}d\beta.
\end{aligned}
\end{equation}
Therefore, on substituting $(\ref{6.11})$ into $(\ref{6.12})$, we conclude that
\begin{equation*}
    \dsum_{\substack{\|\a\|_{\infty}\leq A\\ P(\a)=0}}\biggl|\dsum_{q\in \mathcal{C}}S_{\a}(q)\biggr|^2\ll A^{N-k}\biggl(\dsum_{q\in \mathcal{C}}q^{1+\epsilon}(q^{-1}+q^{-4/d})^{n/16}\biggr)^2.
\end{equation*}
This completes the proof of Lemma $\ref{lem5.15.1}.$
\end{proof}

\bigskip

In advance of the statement of the second auxiliary lemma, we define 
\begin{equation}\label{defnJ}
   \mathfrak{J}_{\a}(w)=X^{n-d}A^{-1} \dint_{|\beta|\leq w}\dint_{[0,1]^n}e(\beta A^{-1}f_{\a}(\boldsymbol{\gamma}))d\boldsymbol{\gamma} d\beta,
\end{equation}
with $\a\in \Z^N.$
Furthermore, we recall the definition ($\ref{defnJ*}$) of $\mathfrak{J}_{\a}^*.$ Note again that in the proof of Lemma $\ref{lemma5.2}$ below, we shall make use of the same idea as that in Section $\ref{sec4}$ (the separation procedure).

\begin{lem}\label{lemma5.2}
Let $n$ and $d$ be natural numbers with $n\geq 8(d+1).$ Suppose that $P\in \Z[\x]$ is a non-singular form in $N_{d,n}$ variables of degree $k\geq 2.$ Then, whenever $N\geq 200k(k-1)2^{k-1}$, we have
    \begin{equation*}
        \dsum_{\substack{\|\a\|_{\infty}\leq A\\ P(\a)=0}}\left|\mathfrak{J}_{\a}^*-\mathfrak{J}_{\a}(w)\right|^2\ll A^{N-k-2}X^{2n-2d}(w^{2-n/(2(d+1))}+\zeta).
    \end{equation*}
\end{lem}
\begin{proof}
Put $K(\beta)=\zeta^{-1}\mathfrak{w}_{\zeta}(\beta),$ and we introduce functions $A_1(\a)$ and $A_2(\a)$ defined by 
$$A_1(\a)=X^{n-d}A^{-1}\dint_{[0,1]^n}\dint_{|\beta|\leq\zeta^{-3}}K(\beta)e(\beta A^{-1}f_{\a}(\boldsymbol{\gamma}))d\beta d\boldsymbol{\gamma}$$
and 
$$A_2(\a)=X^{n-d}A^{-1}\int_{[0,1]^n}\int_{|\beta|\leq\zeta^{-3}}e(\beta A^{-1}f_{\a}(\boldsymbol{\gamma}))d\beta d\boldsymbol{\gamma}.$$
Then, by applying the elementary inequality $(a+b+c)^2\ll (a^2+b^2+c^2)$, we deduce that 
\begin{equation}\label{5.145.145.14}
\begin{aligned}
    &\dsum_{\substack{\|\a\|_{\infty}\leq A\\P(\a)=0}}\left|\mathfrak{J}_{\a}^*-\mathfrak{J}_{\a}(w)\right|^2   \\
    &\ll \dsum_{\substack{\|\a\|_{\infty}\leq A\\P(\a)=0}}\left|A_2(\a)-\mathfrak{J}_{\a}(w)\right|^2+\dsum_{\substack{\|\a\|_{\infty}\leq A\\P(\a)=0}}\left|A_2(\a)-A_1(\a)\right|^2+\dsum_{\substack{\|\a\|_{\infty}\leq A\\P(\a)=0}}\left|A_1(\a)-\mathfrak{J}_{\a}^*\right|^2.
\end{aligned}
\end{equation}
We first analyze the first term. By orthogonality, we have
    \begin{equation}\label{5.135.13}
         \dsum_{\substack{\|\a\|_{\infty}\leq A\\ P(\a)=0}}\left|A_2(\a)-\mathfrak{J}_{\a}(w)\right|^2=\dsum_{\|\a\|_{\infty}\leq A}\dint_0^1\left|A_2(\a)-\mathfrak{J}_{\a}(w)\right|^2e(\alpha P(\a))d\alpha.
    \end{equation}
    On noting that 
    \begin{equation}
    \begin{aligned}
        &\left|A_2(\a)-\mathfrak{J}_{\a}(w)\right|^2\\ 
        &=X^{2(n-d)}A^{-2}\dint_{\substack{w<|\beta_1|\leq \zeta^{-3}\\w<|\beta_2|\leq \zeta^{-3}}}\dint_{[0,1]^{2n}}e(\beta_1 A^{-1}f_{\a}(\boldsymbol{\gamma}_1)-\beta_2 A^{-1}f_{\a}(\boldsymbol{\gamma}_2))d\boldsymbol{\gamma}_1 d\boldsymbol{\gamma}_2 d\beta_1 d\beta_2,
    \end{aligned}
    \end{equation}
    we find that
\begin{equation}\label{5.155.15}
    \begin{aligned}
           &\dsum_{\substack{\|\a\|_{\infty}\leq A\\ P(\a)=0}}\left|A_2(\a)-\mathfrak{J}_{\a}(w)\right|^2\\
           &\ll (X^{n-d}A^{-1})^2\dint_0^1\dint_{\substack{w<|\beta_1|\leq \zeta^{-3}\\w<|\beta_2|\leq \zeta^{-3}}}\dint_{[0,1]^{2n}}T(\alpha,\beta_1,\beta_2,\boldsymbol{\gamma}_1,\boldsymbol{\gamma}_2)  d\boldsymbol{\gamma}_1d\boldsymbol{\gamma}_2d\beta_1 d\beta_2 d\alpha,
    \end{aligned}
\end{equation}
where
$$T(\alpha,\beta_1,\beta_2,\boldsymbol{\gamma}_1,\boldsymbol{\gamma}_2)=\dsum_{-A\leq \a\leq A}e(\alpha P(\a))e(\beta_1 A^{-1}f_{\a}(\boldsymbol{\gamma}_1)-\beta_2A^{-1}f_{\a}(\boldsymbol{\gamma}_2)).$$

In order to apply the Cauchy-Schwarz inequality, we define
\begin{equation*}
    T_1(\beta_1,\beta_2,\h)=\dint_{[0,1]^{2n}}e(\beta_1 A^{-1}f_{\h}(\boldsymbol{\gamma}_1)-\beta_2 A^{-1}f_{\h}(\boldsymbol{\gamma}_2))d\boldsymbol{\gamma}_1 d\boldsymbol{\gamma}_2
\end{equation*}
and 
\begin{equation*}
    T_2(\alpha,\h)=\dsum_{\a\in I_{\h}}e(\alpha(P(\a+\h))-P(\a)))
\end{equation*}
in which 
$$I_{\h}=\{\a\in \Z^N|\ \|\a+\h\|_{\infty}\leq A,\ \|\a\|_{\infty}\leq A\}.$$
Then, by applying the Cauchy-Schwarz inequality and Weyl differencing argument together with the triangle inequality, we deduce that
\begin{equation*}
\dint_{[0,1]^{2n}}T(\alpha,\beta_1,\beta_2,\boldsymbol{\gamma}_1,\boldsymbol{\gamma}_2)d\boldsymbol{\gamma}_1d\boldsymbol{\gamma}_2\leq \biggl(\dsum_{\|\h\|_{\infty}\leq 2A}\left|T_1(\beta_1,\beta_2,\h)\right|\left|T_2(\alpha,\h)\right|\biggr)^{1/2}.
\end{equation*}
If we further define 
$$U_1(\beta_1,\beta_2)=\dsum_{\|\h\|_{\infty}\leq 2A}\left|T_1(\beta_1,\beta_2,\h)\right|^2\ \text{and}\ U_2(\alpha)=\dsum_{\|\h\|_{\infty}\leq 2A}\left|T_2(\alpha,\h)\right|^2,$$
we find by applying the Cauchy-Schwarz inequality again that
\begin{equation}\label{5.165.16}
\dint_{[0,1]^{2n}}T(\alpha,\beta_1,\beta_2,\boldsymbol{\gamma}_1,\boldsymbol{\gamma}_2)d\boldsymbol{\gamma}_1d\boldsymbol{\gamma}_2\leq U_1(\beta_1,\beta_2)^{1/4}\cdot U_2(\alpha)^{1/4}.
\end{equation}

We investigate $U_1(\beta_1,\beta_2)$. Recall Definition $\ref{defn4.2}.$ Furthermore, for $\d\in \Z^n$,  $\boldsymbol{\gamma}_1^{(1)}$, $\boldsymbol{\gamma}_1^{(2)}$, $\boldsymbol{\gamma}_2^{(1)}$, $\boldsymbol{\gamma}_2^{(2)}\in \R^n$, we define
\begin{equation*}
\begin{aligned}
    \Psi(\d,\boldsymbol{\gamma}_1^{(1)}, \boldsymbol{\gamma}_1^{(2)},\boldsymbol{\gamma}_2^{(1)},\boldsymbol{\gamma}_2^{(2)}) &:=\Psi(\d,\boldsymbol{\gamma}_1^{(1)}, \boldsymbol{\gamma}_1^{(2)},\boldsymbol{\gamma}_2^{(1)},\boldsymbol{\gamma}_2^{(2)};\beta_1,\beta_2)\\
    &=A^{-1}\beta_1 (\langle\d, v_{d}(\boldsymbol{\gamma}_1^{(1)})- v_{d}(\boldsymbol{\gamma}_1^{(2)})\rangle)-A^{-1}\beta_2( \langle \d,v_d(\boldsymbol{\gamma}_2^{(1)})- v_{d}(\boldsymbol{\gamma}_2^{(2)})\rangle).
\end{aligned}
\end{equation*}
Then, by squaring out and applying the triangle inequality, we deduce that
\begin{equation}\label{5.175.17}
    U_1(\beta_1,\beta_2)\ll A^{N-n}\dint_{[0,1]^{4n}}\biggl|\dsum_{\substack{\|\d\|_{\infty}\leq 2A\\\d\in \Z^n}}e(\Psi(\d,\boldsymbol{\gamma}_1^{(1)}, \boldsymbol{\gamma}_1^{(2)},\boldsymbol{\gamma}_2^{(1)},\boldsymbol{\gamma}_2^{(2)}))\biggr|d\boldsymbol{\gamma},
\end{equation}
where we wrote $d\boldsymbol{\gamma}=d\boldsymbol{\gamma}_1^{(1)}d\boldsymbol{\gamma}_1^{(2)}d\boldsymbol{\gamma}_2^{(1)}d\boldsymbol{\gamma}_2^{(2)},$ for simplicity.
By applying the Cauchy-Schwarz inequality and the triangle inequality, we see from $(\ref{5.175.17})$ that
\begin{equation*}
    U_1(\beta_1,\beta_2)\ll A^{N-n}\biggl(\dsum_{\substack{\|\d\|_{\infty}\leq 4A\\ \d\in \Z^n}}A^n\biggl|\dint_{[0,1]^{4n}}e(\Psi(\d,\boldsymbol{\gamma}_1^{(1)}, \boldsymbol{\gamma}_1^{(2)},\boldsymbol{\gamma}_2^{(1)},\boldsymbol{\gamma}_2^{(2)}))d\boldsymbol{\gamma}\biggr|\biggr)^{1/2}.
\end{equation*}

On writing $$I(\xi)=\dint_0^1e(\xi \gamma^d)d\gamma,$$
we find by applying the H\"older's inequality that
\begin{equation}\label{5.20202020}
\begin{aligned}
 U_1(\beta_1,\beta_2)&\ll A^{N-n/2}\biggl(\dsum_{|f|\leq 4A}|I(\beta _1A^{-1}f)|^2|I(A^{-1}\beta_2 f)|^2\biggr)^{n/2}\\
 &\ll A^{N-2n/(d+1)}\cdot \biggl(\dsum_{|f|\leq 4A}|I(\beta _1A^{-1}f)|^{d+1}\biggr)^{\frac{n}{d+1}}\cdot\biggl(\dsum_{|f|\leq 4A}|I(\beta _2A^{-1}f)|^{d+1}\biggr)^{\frac{n}{d+1}}.
\end{aligned}
\end{equation}

Meanwhile, since by [$\ref{ref15}$, Lemma 2.8] we have
$$I(\xi)\ll \text{min}\left(1,\frac{1}{|\xi|^{1/d}}\right),$$
it follows that whenever $w<|\beta|\leq\zeta^{-3}$, one has
\begin{equation}\label{5.212121}
    \begin{aligned}
        \dsum_{|f|\leq 4A}|I(\beta A^{-1}f)|^{d+1}&\ll \dsum_{|f|\leq A|\beta|^{-1}}1+\dsum_{A|\beta|^{-1}\leq |f|\leq 4A}\frac{1}{|A^{-1}\beta f|^{(d+1)/d}}\ll A|\beta|^{-1}.
    \end{aligned}
\end{equation}
Hence, whenever $w<|\beta_1|\leq \zeta^{-3}$ and $w<|\beta_2|\leq\zeta^{-3}$, we have
\begin{equation}\label{5.185.18}
    U_1(\beta_1,\beta_2)\ll A^N\cdot |\beta_1|^{-\frac{n}{d+1}}\cdot |\beta_2|^{-\frac{n}{d+1}}.
\end{equation}

Next, recall the definition of $U_2(\alpha).$ Since $I_{\h}$ in $U_2(\alpha)$ is a rectangular box depending on $\h,$ it follows by Lemma $\ref{lem3.2}$ with $l=1,B=0,A_1=A_2=2A$ and $\sigma=1/4$ that 
\begin{equation}\label{5.195.19}
\dint_0^1 U_2(\alpha)^{1/4}d\alpha \ll A^{3N/4-k}.
\end{equation}
On substituting $(\ref{5.185.18})$ into $(\ref{5.165.16})$ and that into $(\ref{5.155.15})$, we find that
\begin{equation}\label{5.205.20}
    \begin{aligned}
         &\dsum_{\substack{\|\a\|_{\infty}\leq A\\ P(\a)=0}}\left|A_2(\a)-\mathfrak{J}_{\a}(w)\right|^2\\
           &\ll (X^{n-d}A^{-1})^2\cdot A^{N/4}\dint_{\substack{w<|\beta_1|\leq \zeta^{-3}\\w<|\beta_2|\leq \zeta^{-3}}}|\beta_1|^{-n/(4(d+1))}|\beta_2|^{-n/(4(d+1))}d\beta_1 d\beta_2
 \dint_0^1 U_2(\alpha)^{1/4}d\alpha\\
 &\ll A^{N-k-2}\cdot X^{2n-2d}\cdot w^{2-n/(2(d+1))},
    \end{aligned}
\end{equation}
where we have used $(\ref{5.195.19})$.

Next, we turn to estimate the second term in $(\ref{5.145.145.14}).$ By the orthogonality, we see that 
\begin{equation*}
\begin{aligned}
    \dsum_{\substack{\|\a\|_{\infty}\leq A\\ P(\a)=0}}|A_1(\boldsymbol{a})-A_2(\boldsymbol{a})|^2=\dsum_{\|\a\|_{\infty}\leq A}\dint_0^1\left|A_1(\a)-A_2(\a)\right|^2e(\alpha P(\a))d\alpha.
\end{aligned}
\end{equation*}
On recalling the definition of $A_1(\a)$ and $A_2(\a)$, it follows by squaring out and swapping the order of summation and integration that 
\begin{equation}\label{revise}
\begin{aligned}
    &\dsum_{\substack{\|\a\|_{\infty}\leq A\\ P(\a)=0}}\left|A_1(\a)-A_2(\a)\right|^2  \\
    &\ll X^{2(n-d)}\cdot A^{-2}\int_0^1\dint_{\substack{|\beta_1|\leq\zeta^{-3}\\|\beta_2|\leq\zeta^{-3}}}(1-K(\beta_1))(1-K(\beta_2))\mathcal{T}(\alpha,\beta_1,\beta_2) d\beta_1 d\beta_2 d\alpha,
\end{aligned}
\end{equation}
where $$\mathcal{T}(\alpha,\beta_1,\beta_2)=\dint_{[0,1]^{2n}}T(\alpha,\beta_1,\beta_2,\boldsymbol{\gamma}_1,\boldsymbol{\gamma}_2)d\boldsymbol{\gamma}_1 d\boldsymbol{\gamma}_2.$$
Furthermore, by $(\ref{5.165.16})$ and $(\ref{5.20202020})$, one has 
\begin{equation}\label{5.2424}
    \mathcal{T}(\alpha,\beta_1,\beta_2)\ll U_1(\beta_1,\beta_2)^{1/4}\cdot U_2(\alpha)^{1/4}
\end{equation}
and
\begin{equation}\label{5.2525}
\begin{aligned}
   & U_1(\beta_1,\beta_2)\\
    &\ll A^{N-2n/(d+1)}\biggl(\dsum_{|f|\leq 4A}|I(\beta_1 A^{-1}f)|^{d+1}\biggr)^{n/(d+1)}\cdot\biggl(\dsum_{|f|\leq 4A}|I(\beta_2 A^{-1}f)|^{d+1}\biggr)^{n/(d+1)}.  
\end{aligned}
\end{equation}
Then, by using the inequality
$$I(\xi)\ll \text{min}\left(1,\frac{1}{|\xi|^{1/d}}\right)$$
again together with the first inequality in $(\ref{5.212121})$, we deduce that
$$\dsum_{|f|\leq 4A}|I(\beta A^{-1}f)|^{d+1}\ll A\cdot \text{min}(1,|\beta|^{-1}).$$
Hence, one has
\begin{equation}\label{5.2626}
    U_1(\beta_1,\beta_2)\ll A^N\cdot \text{min}\left(1,|\beta_1|^{-n/(d+1)}\right)\cdot \text{min}\left(1,|\beta_2|^{-n/(d+1)}\right).
\end{equation}
On substituting $(\ref{5.2626})$ into $(\ref{5.2424})$ and that into $(\ref{revise})$, we find that
\begin{equation}\label{5.2727}
\begin{aligned}
    &\dsum_{\substack{\|\a\|_{\infty}\leq A\\ P(\a)=0}}\left|A_1(\a)-A_2(\a)\right|^2  \\
    &\ll X^{2(n-d)}\cdot A^{N/4-2}\biggl(\int_{\R}|1-K(\beta)|\cdot \text{min}\left(1,|\beta|^{-n/(4(d+1))}\right)d\beta\biggr)^2\cdot \int_0^1U_2(\alpha)^{1/4}d\alpha\\
    &\ll X^{2n-2d}\cdot A^{N-k-2}\biggl(\int_{\R}|1-K(\beta)|\cdot \text{min}\left(1,|\beta|^{-n/(4(d+1))}\right)d\beta\biggr)^2,
\end{aligned}
\end{equation}
where we have used $(\ref{5.195.19}).$

Meanwhile, put $\mathfrak{C}=[-\zeta^{-1/2},\zeta^{-1/2}]$ and $\mathfrak{D}=\R\setminus \mathfrak{C}.$ Then from the Taylor expansion of $\mathfrak{w}_{\zeta}(\beta)$, we have
$$|1-K(\beta)|\ll \text{min}(1,|\beta\zeta|^2).$$
Whenever $n>4(d+1),$ we obtain
\begin{equation*}
\begin{aligned}
  &\dint_{\mathfrak{C}}|1-K(\beta)|\cdot \text{min}(1,|\beta|^{-n/(4(d+1))})d\beta\\
  &\ll \text{sup}_{\beta\in \mathfrak{C}}|1-K(\beta)|\cdot \int_{\mathfrak{C}}\text{min}(1,|\beta|^{-n/(4(d+1))})d\beta\ll \zeta.  
\end{aligned}
\end{equation*}
Whenever $n\geq 8(d+1)$, the contribution arising from $\mathfrak{D}$ is 
$$\dint_{\mathfrak{D}}|1-K(\beta)|\cdot \text{min}(1,|\beta|^{-n/(4(d+1))})d\beta\ll \dint_{\mathfrak{D}}|\beta|^{-2}d\beta\ll \zeta^{1/2}.$$
Therefore, on substituting these estimates into $(\ref{5.2727})$, we conclude that 
\begin{equation}\label{5.285.28}
    \dsum_{\substack{\|\a\|_{\infty}\leq A\\ P(\a)=0}}\left|A_1(\a)-A_2(\a)\right|^2\ll A^{N-k-2}\cdot X^{2n-2d}\cdot \zeta.
\end{equation}

Lastly, we investigate the third term in $(\ref{5.145.145.14})$.
By the definition of $A_1(\a)$, $\mathfrak{w}_{\zeta}(\beta)$ and  $\mathfrak{J}^*_{\a}$, we deduce that 
\begin{equation}\label{5.295.29}
\begin{aligned}
    &\dsum_{\substack{\|\a\|_{\infty}\leq A\\P(\a)=0}}\left|A_1(\a)-\mathfrak{J}_{\a}^*\right|^2\\
     &\leq X^{2n-2d}A^{-2}\dsum_{\substack{\|\a\|_{\infty}\leq A\\ P(\a)=0}}\biggl|\dint_{[0,1]^n}\dint_{|\beta|>\zeta^{-3}}\left(\frac{\text{sin}(\pi \beta\zeta)}{\pi \beta\zeta}\right)^2e(\beta A^{-1}f_{\a}(\boldsymbol{\gamma}))d\beta d\boldsymbol{\gamma}\biggr|^2\\
     &\leq X^{2n-2d}A^{-2}\dsum_{\substack{\|\a\|_{\infty}\leq A\\P(\a)=0}}\biggl(\frac{1}{\zeta}\int_{|\beta'|> \zeta^{-2}}\left|\frac{\text{sin}(\pi \beta')}{\pi \beta'}\right|^2d\beta'\biggr)^2\\
     &\ll X^{2n-2d}\cdot A^{-2} \cdot \zeta^2\cdot \dsum_{\substack{\|\a\|_{\infty}\leq A\\P(\a)=0}}1\ll  A^{N-k-2}\cdot X^{2n-2d}\cdot  \zeta^2,
\end{aligned}
\end{equation}
where we have used the triangle inequality, change of variable $\beta'=\beta\zeta$ and the bound $$\dsum_{\substack{\|\a\|_{\infty}\leq A\\P(\a)=0}}1\ll A^{N-k}.$$ Therefore, on substituting $(\ref{5.205.20})$, $(\ref{5.285.28})$ and $(\ref{5.295.29})$ into ($\ref{5.145.145.14}$), we complete the proof of Lemma $\ref{lemma5.2}.$
\end{proof}

\bigskip

\begin{proof}[Proof of Theorem $\ref{thm2.2}$]

Recall the definition $(\ref{4.1})$ and $(\ref{4.24.2})$ of $\mathcal{I}_{\a}(X,\mathfrak{B})$ and $\mathfrak{M}$. Then, we find that
\begin{equation*}
\begin{aligned}
\mathcal{I}_{\a}(X,\mathfrak{M})&=\dint_{\mathfrak{M}}\dsum_{\substack{1\leq \x\leq X\\ \x\in \Z^{n}}}e(\alpha f_{\a}(\x)) d\alpha  \\
&=\dsum_{1\leq q\leq w}\dsum_{\substack{1\leq b\leq q\\(q,b)=1}}\dint_{|\alpha-b/q|\leq \frac{w}{AX^d}}\dsum_{\substack{1\leq \x\leq X}}e(\alpha f_{\a}(\x)) d\alpha.
\end{aligned} 
\end{equation*}
Recall the definition ($\ref{defnJ}$) of $\mathfrak{J}_{\a}(w).$ By applying classical treatments in major arcs [$\ref{ref8}$, Lemma 5.1] and writing $\beta=\alpha-b/q,$ we readily find that 
\begin{equation}\label{6.1}
\begin{aligned}
        &\mathcal{I}_{\a}(X,\mathfrak{M})=\dsum_{1\leq q\leq w}S_{\a}(q)\mathfrak{J}_a(w)+O(A^{-1}X^{n-d-1}w^5),
        \end{aligned}
\end{equation}
where 
$$S_{\a}(q)=\dsum_{\substack{1\leq b\leq q\\ (q,b)=1}}q^{-n}\dsum_{1\leq \r\leq q}e\left(\frac{b}{q} f_{\a}(\r)\right).$$
Meanwhile, recall the definition $(\ref{def2.3})$ and ($\ref{6.26.26.2}$) of $W$ and $\mathfrak{S}_{\a}^*$, and note from the classical treatment that 
$$\mathfrak{S}_{\a}^*=\displaystyle\prod_{p\leq w}\biggl(\dsum_{\substack{0\leq h\leq \log_pw}}S_{\a}(p^h)\biggr).$$
If we define a set 
\begin{equation*}
    \mathcal{Q}=\{q\in (w,W]|\ \text{for all primes}\ p,\ p^h\|q\Rightarrow p^h\leq w\}
\end{equation*}
and define $$\mathcal{E}_{\a}=\dsum_{q\in \mathcal{Q}}S_{\a}(q),$$ we find from the multiplicativity of $S_{\a}(q)$ that
\begin{equation}\label{6.2}
    \dsum_{1\leq q\leq w}S_{\a}(q)=\mathfrak{S}_{\a}^*-\mathcal{E}_{\a}.
\end{equation}

Meanwhile, note that $$\mathcal{I}_{\a}(X)=\mathcal{I}_{\a}(X,\mathfrak{M})+\mathcal{I}_{\a}(X,\mathfrak{m}).$$ Then, on recalling the definition of $\mathfrak{J}_{\a}^*$, we deduce from $(\ref{6.1})$ and $(\ref{6.2})$ together with applications of the elementary inequality $(a+b)^2\leq 2a^2+2b^2$ that
\begin{equation}\label{6.3}
\begin{aligned}
      &\dsum_{\substack{\|\a\|_{\infty}\leq A\\P(\a)=0}}\left|\mathcal{I}_{\a}(X)-\mathfrak{S}_{\a}^*\mathfrak{J}_{\a}^*\right|^2\\
      &\ll \dsum_{\substack{\|\a\|_{\infty}\leq A\\P(\a)=0}}\left|\mathcal{I}_{\a}(X,\mathfrak{M})-\mathfrak{S}_{\a}^*\mathfrak{J}_{\a}^*\right|^2+\dsum_{\substack{\|\a\|_{\infty}\leq A\\P(\a)=0}}\left|\mathcal{I}_{\a}(X,\mathfrak{m})\right|^2\\
      &\ll \Sigma_1+\Sigma_2+\dsum_{\substack{\|\a\|_{\infty}\leq A\\P(\a)=0}}\left|\mathcal{I}_{\a}(X,\mathfrak{m})\right|^2+O\biggl(A^{-2}X^{2n-2d-2}w^{10}\dsum_{\substack{\|\a\|_{\infty}\leq A\\P(\a)=0}}1\biggr),
\end{aligned}
\end{equation}
where 
$$\Sigma_1=\dsum_{\substack{\|\a\|_{\infty}\leq A\\P(\a)=0}}\left|\mathcal{E}_{\a}\mathfrak{J}_{\a}^*\right|^2\ \text{and}\ \Sigma_2=\dsum_{\substack{\|\a\|_{\infty}\leq A\\P(\a)=0}}\left|(\mathfrak{S}_{\a}^*-\mathcal{E}_{\a})(\mathfrak{J}_{\a}^*-\mathfrak{J}_{\a}(w))\right|^2.$$

First, we estimate the third and fourth terms of the last expression in ($\ref{6.3}$). Since we have $$\#\{\a\in \Z^N|\ \|\a\|_{\infty}\leq A,\ P(\a)=0\}\ll A^{N-k},$$ it follows by Proposition $\ref{prop4.1}$ that 
\begin{equation}\label{6.161616}
\dsum_{\substack{\|\a\|_{\infty}\leq A\\P(\a)=0}}\left|\mathcal{I}_{\a}(X,\mathfrak{m})\right|^2+O\biggl(A^{-2}X^{2n-2d-2}w^{10}\dsum_{\substack{\|\a\|_{\infty}\leq A\\P(\a)=0}}1\biggr)\ll A^{N-k-2}X^{2n-2d}(\log A)^{-\delta},    
\end{equation}
with some $\delta>0.$ 

Next, we turn to estimate the first term of the last expression in ($\ref{6.3}$). From the trivial bound, we have 
\begin{equation}\label{6.1616}
\left|\mathfrak{J}_{\a}^*\right|^2\ll X^{2(n-d)}\zeta^{-2}A^{-2}=X^{2(n-d)}w^{10}A^{-2}.
\end{equation}
By Lemma $\ref{lem5.15.1}$ with $B=w,C=W$ and $\mathcal{C}=\mathcal{Q}$, we have
\begin{equation}\label{6.17}
    \begin{aligned}
        \dsum_{\substack{\|\a\|_{\infty}\leq A\\P(\a)=0}}\left|\mathcal{E}_{\a}\right|^2&\ll A^{N-k}\biggl(\dsum_{q\geq w}q^{1+\epsilon}(q^{-1}+q^{-4/d})^{n/16}\biggr)^2\\
        &\ll A^{N-k}\biggl(\dsum_{q\geq w}\left(q^{1-n/16+\epsilon}+q^{1-n/(4d)+\epsilon}\right)\biggr)^2.
    \end{aligned}
\end{equation}
Meanwhile, from the hypotheses $n_1\leq \lfloor(n-1)/2\rfloor/8$ and $2X^d\leq A\leq X^{n_1-d}$ in the statement of Theorem $\ref{thm2.2},$ we see that $n> 32d.$ Hence, it follows from $(\ref{6.17})$ together with the hypothesis $d\geq 4$ that
\begin{equation}\label{6.18}
\begin{aligned}
    \dsum_{\substack{\|\a\|_{\infty}\leq A\\P(\a)=0}}\left|\mathcal{E}_{\a}\right|^2&\ll  A^{N-k}\left(w^{2-2d}+w^{-6}\right)^2\ll A^{N-k}\cdot w^{-12}.
\end{aligned}
\end{equation}
Therefore, combining $(\ref{6.1616})$ and $(\ref{6.18})$, we conclude that
\begin{equation}\label{5.202020}
\Sigma_1=    \dsum_{\substack{\|\a\|_{\infty}\leq A\\P(\a)=0}}\left|\mathcal{E}_{\a}\mathfrak{J}_{\a}^*\right|^2\ll A^{N-k-2}X^{2n-2d}w^{-2}.
\end{equation}

Lastly, it remains to estimate the second term of the last expression in ($\ref{6.3}$). From the trivial bound, we have
$$\left|\mathfrak{S}_{\a}^*-\mathcal{E}_{\a}\right|^2= \biggl|\dsum_{1\leq q\leq w}S_{\a}(q)\biggr|^2\leq w^4.$$
Hence, we deduce by applying Lemma $\ref{lemma5.2}$ with $n> 32d$ and $d\geq 4$ that
\begin{equation}\label{5.3030}
    \Sigma_2=\dsum_{\substack{\|\a\|_{\infty}\leq A\\P(\a)=0}}\left|(\mathfrak{S}_{\a}^*-\mathcal{E}_{\a})(\mathfrak{J}_{\a}^*-\mathfrak{J}_{\a}(w))\right|^2\ll A^{N-k-2}\cdot X^{2n-2d}\cdot w^{-1}.
\end{equation}
 
 Then, on recalling the definition of $w$ and substituting $(\ref{6.161616})$, $(\ref{5.202020})$ and $(\ref{5.3030})$ into the last expression in ($\ref{6.3}$), one concludes that 
\begin{equation*}
    \dsum_{\substack{\|\a\|_{\infty}\leq A\\P(\a)=0}}\left|\mathcal{I}_{\a}(X)-\mathfrak{S}_{\a}^*\mathfrak{J}_{\a}^*\right|^2\ll A^{N-k-2}X^{2n-2d}(\log A)^{-\delta},
\end{equation*}
for some $\delta$ with $0<\delta<1.$ This completes the proof of Theorem $\ref{thm2.2}.$
\end{proof}

\bigskip

\section{Major arcs}\label{sec6}
In this section, we show that $\mathfrak{S}_{\a}^*$ and $\mathfrak{J}^*_{\a}$ are rarely small. In order to verify this, we develop a method combining the classical treatments of major arcs with those used in [$\ref{ref3},$ section 5].

\bigskip

\subsection{Singular series treatment I}\label{sec6.1}

Our purpose in sections $\ref{sec6.1}$ and $\ref{sec6.26.2}$ is to prove Proposition $\ref{prop6.1}$ below. 
In advance of the statement of this proposition, we recall the definition $(\ref{def2.3})$ of $W$, and recall the definition $\mathcal{A}^{\text{loc}}_{d,n}(A;P)$ in section 1 and that
$\mathfrak{S}_{\a}^*=\sigma(\a;W).$
\begin{prop}\label{prop6.1}
Let $A$ and $X$ be positive numbers with $X^3\leq A.$ Suppose that $n$ and $d$ are natural numbers with $n>d+1$ and $d\geq 2$, and that $X^3\leq A$. Suppose that $P\in \Z[\x]$ is a non-singular form in $N_{d,n}$ variables of degree $k\geq 2.$ Then, whenever $N\geq 1000n^28^k$, one has
\begin{equation*}
    A^{-N+k}\cdot\#\left\{\a\in \mathcal{A}^{\text{loc}}_{d,n}(A;P)\middle|\ \begin{aligned}
        \mathfrak{S}_{\a}^*\leq(\log A)^{-\eta}
    \end{aligned}\right\}\ll (\log A)^{-\eta/(20n)},
\end{equation*}
for any $\eta>0.$
\end{prop}

\bigskip

We provide here the structure of sections $\ref{sec6.1}$ and $\ref{sec6.26.2}$. For a prime number $p$ and a natural number $r,$ in section $\ref{6.1.2}$, we give asymptotic formulae for 
 $$\#\left\{\a\in [1,p^r]^N\middle|\ P(\a)\equiv0\ (\textrm{mod}\ p^r)\right\}\ \text{and}\ \dsum_{\substack{\a\in [1,p^r]^N\\P(\a)=0\ \textrm{mod}\ p^r}}\sigma(\a;p^r),$$
 respectively. In section $\ref{subsec5.1.3}$, we provide a bound for a variance of $\sigma(\a;p^r),$ that is
 $$ \dsum_{\substack{\a\in [1,p^r]^N\\P(\a)\equiv 0\ \text{mod}\ p^r}}(\sigma(\a;p^r)-1)^2.$$

Recall the definition of $v_{p^r}(\boldsymbol{v})$ and $\mathcal{R}_m(Q)$ leading to Lemma $\ref{lem3.4}$ and Lemma $\ref{lem3.5}.$ For $r,v,e\in \Z$ with $0\leq e\leq r-v$, we introduce a condition $\mathcal{C}_v^{(e)}(p^r)$ as follows. We say that $\a\in \Z^N$ satisfies the condition $\mathcal{C}_v^{(e)}(p^r)$ when one has
\begin{equation*}
    \begin{aligned}
       &(i)\ p^v\|\a\ \text{and}\ P(\a)\equiv 0\ \text{mod}\ p^r\\
       &(ii)\ \exists\ \x\in \mathcal{R}_n(p^{r-v})\ \text{such that}\ v_{p^{r-v}}(\nabla f_{p^{-v}\a}(\x))=e\ \textrm{and}\ 
f_{p^{-v}\a}(\x)\equiv 0\ \text{mod}\ p^{r-v}.
    \end{aligned}
\end{equation*}
For $r,v,e\in \Z$ with $0\leq e\leq r-v$, define a set
 \begin{equation}\label{6.59590}
 \begin{aligned}
A_v^{(e)}(p^r)=\left\{\a\in [1,p^r]^N\middle|\ \a \ \text{satisfies the condition}\ \mathcal{C}_v^{(e)}(p^r)\right\}.     
 \end{aligned}
\end{equation}
In section $\ref{6.1.4}$, we provide an upper bound for $\#A_v^{(e)}(p^r).$ In section $\ref{6.1.5}$, we give a lower bound for $\sigma (\a;p^r)$ with $\a\in A_v^{(e)}(p^r).$ Combining all these estimates together with the strategy used in  [$\ref{ref3},$ section 5], we shall prove Proposition $\ref{prop6.1}$ at the end of section $\ref{sec6.26.2}$.

\bigskip

\subsubsection{Proofs of  Lemmas $\ref{lem5.1}$ and $\ref{lem5.2}$}\label{6.1.2}
In order to describe the following lemma, recall that $P\in \Z[\x]$ is a non-singular form in $N$ variables of degree $k$, and write
$$N_1(p^r)=\#\left\{\a\in [1,p^r]^N\middle|\ P(\a)\equiv0\ (\textrm{mod}\ p^r)\right\}.$$ 

\begin{lem}\label{lem5.1}
Suppose that $p$ is a prime number, and $r$ is a natural number. Then, whenever $N>(k-1)2^{k+1}$, we have
\begin{equation*}
  N_1(p^r) =p^{r(N-1)}+O(p^{r(N-1)-N/(2^k(k-1))}).
\end{equation*}
\begin{proof}

This immediately follows from Lemma $\ref{lem3.93.9}$ with $W=p^r.$

\end{proof}
\end{lem}

\bigskip

Recall the definition $(\ref{def6.1})$ of $\sigma(\a;p^r)$. To describe the following lemma, recall again that $P\in \Z[\x]$ is a non-singular form in $N$ variables of degree $k.$ It is convenient to define
$$N_2(p^r)=\dsum_{\substack{\a\in [1,p^r]^N\\P(\a)\equiv0\ \textrm{mod}\ p^r}}\sigma(\a;p^r).$$

\begin{lem}\label{lem5.2}
Suppose that $p$ is a prime number, and $r$ is a natural number. Then, whenever $n>d$ and $N>(k-1)2^{k-1}(2+n/d+n-d)$, we have
\begin{equation*}
N_2(p^r)=p^{r(N-1)}+O(p^{rN-\lceil r/d\rceil n}+p^{rN-r-n+d}).    
\end{equation*}
\end{lem}
\begin{proof}
Inverting the order of summation we obtain

\begin{equation}\label{eqeq5.1}
    N_2(p^r)=p^{-r(n-1)}\dsum_{\g\in [1,p^r]^n}\#\left\{\a\in [1,p^r]^N\middle|\ \begin{aligned}
 \langle\a,\nu_{d,n}(\g)\rangle&\equiv 0\ \text{mod}\ p^r\\
  P(\a)&\equiv0\ \textrm{mod}\ p^r
    \end{aligned}\right\}.
\end{equation}
By orthogonality, one has
\begin{equation}\label{eqeq5.2}
\begin{aligned}
   &\#\left\{\a\in [1,p^r]^N\middle|\ \begin{aligned}
    \langle\a,\nu_{d,n}(\g)\rangle&\equiv 0\ \text{mod}\ p^r\\
 P(\a)&\equiv0\ \textrm{mod}\ p^r
    \end{aligned}\right\} =p^{-2r}\dsum_{1\leq l_1,l_2\leq p^r}\Xi(\g,l_1,l_2),
\end{aligned}
\end{equation}
where 
$$\Xi(\g,l_1,l_2)=\dsum_{1\leq\a\leq p^r}e\biggl(\frac{P(\a)l_1}{p^r}\biggr)e\biggl(\frac{\langle\a,\nu_{d,n}(\g)\rangle l_2}{p^r}\biggr).$$

 Write $(l_1,p^r)=p^{r-r_1}$ with $0\leq r_1\leq r.$ 
Then, on substituting ($\ref{eqeq5.2}$) into $(\ref{eqeq5.1})$ and by splitting summation over $\g$ in terms of values of $(\g,p^r)$, we see that
\begin{equation}\label{eqeqeq5.3}
    N_2(p^r)=p^{-r(n-1)}\dsum_{\substack{0\leq r_2\leq r}}\dsum_{\substack{1\leq \g\leq p^r\\(\g,p^r)=p^{r_2}}}(S(\g)+T(\g)),
\end{equation}
where $$S(\g)=p^{-2r}\dsum_{1\leq r_1\leq r}\dsum_{\substack{(l_1,p^r)=p^{r-r_1}\\1\leq l_2\leq p^r}}\Xi(\g,l_1,l_2)$$
and 
$$T(\g)=p^{-2r}\dsum_{1\leq l_2\leq p^r}\dsum_{\a\in [1,p^r]^N}e\biggl(\frac{\langle\a,\nu_{d,n}(\g)\rangle l_2}{p^r}\biggr).$$
Recall the definition of $\Xi(\g,l_1,l_2)$ and that $p^{r-r_1}=(l_1,p^r)$. For fixed $r_1,$ let us write $l_1=p^{r-r_1}\widetilde{l}_1$ with $(\widetilde{l}_1,p)=1.$ Then, on writing $\a=p^{r_1}\n+\m$ with $0\leq \n\leq p^{r-r_1}-1$ and $1\leq \m\leq p^{r_1}$, we see that
\begin{equation*}
    \Xi(\g,l_1,l_2)=\dsum_{0\leq \n\leq p^{r-r_1}-1}\dsum_{1\leq \m\leq p^{r_1}}e\biggl(\frac{P(\m)\widetilde{l}_1}{p^{r_1}}\biggr)e\biggl(\frac{\langle p^{r_1}\n+\m,\nu_{d,n}(\g)\rangle l_2}{p^r}\biggr).
\end{equation*}
On substituting this expression into $S(\g)$ and that into $(\ref{eqeqeq5.3})$, it follows by the triangle inequality that 
\begin{equation}\label{5.45.4}
    \begin{aligned}
    &p^{-r(n-1)}\dsum_{\substack{0\leq r_2\leq r}}\dsum_{\substack{1\leq \g\leq p^r\\(\g,p^r)=p^{r_2}}}S(\g)\\
    &\leq p^{-r(n+1)}\dsum_{\substack{1\leq r_1\leq r\\0\leq r_2\leq r}}\dsum_{\substack{1\leq \g\leq p^r\\(\g,p^r)=p^{r_2}}}\dsum_{\substack{1\leq \widetilde{l}_1\leq p^{r_1}\\(\widetilde{l}_1,p)=1}}\bigl|S_1(\widetilde{l}_1,\g)\bigr|\cdot\bigl|S_2(\g)\bigr|,
    \end{aligned}
\end{equation}
where $$S_1(\widetilde{l}_1,\g)=\sup_{1\leq l_2\leq p^r}\biggl|\dsum_{1\leq \m\leq p^{r_1}}e\biggl(\frac{P(\m)\widetilde{l}_1}{p^{r_1}}\biggr)e\biggl(\frac{\langle\m,\nu_{d,n}(\g)\rangle l_2}{p^r}\biggr)\biggr|$$
and 
$$S_2(\g)=\dsum_{1\leq l_2\leq p^r}\biggl|\dsum_{0\leq \n\leq p^{r-r_1}-1}e\biggl(\frac{\langle\n,\nu_{d,n}(\g)\rangle l_2}{p^{r-r_1}}\biggr)\biggr|.$$

We first analyze the sum $S_1(\widetilde{l}_1,\g).$ Since $P(\m)$ is a non-singular form and $(p^{r_1},\widetilde{l}_1)=1$, it follows by the Weyl type estimate for exponential sums over minor arcs  [$\ref{ref8}$, Lemma 5.4] that 
\begin{equation}\label{5.55.5}
    S_1(\widetilde{l}_1,\g)\ll p^{Nr_1-Nr_1/(2^{k-1}(k-1))+\epsilon}.
\end{equation}

Next, we analyze the sum $S_2(\g)$. We infer from orthogonality that the inner sum in $S_2(\g)$ is a non-negative integer. 
Thus, by changing the order of summations, we have
$$S_2(\g)=\dsum_{0\leq \n\leq p^{r-r_1}-1}\dsum_{1\leq l_2\leq p^r}e\biggl(\frac{\langle\n,\nu_{d,n}(\g)\rangle l_2}{p^{r-r_1}}\biggr).$$ 
Then, by orthogonality, one has
\begin{equation}\label{5.65.6}
\begin{aligned}
    S_2(\g)&=p^{r}\cdot \{0\leq \n\leq p^{r-r_1}-1|\ \langle\n,\nu_{d,n}(\g)\rangle\equiv 0\ \text{mod}\ p^{r-r_1}\}
\end{aligned}
\end{equation}
For fixed $\g\in [1,p^r]^n$ with $(\g,p^r)=p^{r_2},$ we observe that there exists $i$ such that $$((\nu_{d,n}(p^{-r_2}\g))_{i},p)=1.$$ With this observation in mind, we infer from ($\ref{5.65.6}$) that
\begin{equation}\label{6.86.8}
\begin{aligned}
     S_2(\g)&=p^{r}\cdot \{0\leq \n\leq p^{r-r_1}-1|\ \langle\n,\nu_{d,n}(p^{-r_2}\g)\rangle\equiv 0\ \text{mod}\ p^{\max\{0,r-r_1-dr_2\}}\}\\
     &= p^r\cdot p^{(r-r_1)N}\cdot p^{-\max\{0,r-r_1-dr_2\}}.
\end{aligned}
\end{equation}

Therefore, on substituting ($\ref{5.55.5}$) and $(\ref{6.86.8})$ into $(\ref{5.45.4})$, we find that
\begin{equation*}
\begin{aligned}
   & p^{-r(n-1)}\dsum_{\substack{0\leq r_2\leq r}}\dsum_{\substack{1\leq \g\leq p^r\\(\g,p^r)=p^{r_2}}}S(\g)\\
    &\ll p^{-r(n+1)}\dsum_{\substack{1\leq r_1\leq r\\0\leq r_2\leq r}}\dsum_{\substack{1\leq \g\leq p^r\\(\g,p^r)=p^{r_2}}}p^{r_1}\cdot p^{r+rN-r_1N/(2^{k-1}(k-1))+\epsilon}\cdot p^{-\max\{0,r-r_1-dr_2\}}\\
    &\ll p^{-r(n+1)}\dsum_{\substack{1\leq r_1\leq r\\0\leq r_2\leq r}}p^{(r-r_2)n}\cdot p^{r_1}\cdot p^{r+rN-r_1N/(2^{k-1}(k-1))+\epsilon}\cdot p^{-\max\{0,r-r_1-dr_2\}}.
\end{aligned}
\end{equation*}
By splitting the sum over $r_2$, we find that
\begin{equation}\label{6.96.96.9}
     p^{-r(n-1)}\dsum_{\substack{0\leq r_2\leq r}}\dsum_{\substack{1\leq \g\leq p^r\\(\g,p^r)=p^{r_2}}}S(\g)\ll X_1+X_2,
\end{equation}
where
$$X_1=p^{-r(n+1)}\dsum_{\substack{1\leq r_1\leq r\\0\leq r_2\leq (r-r_1)/d}}p^{(r-r_2)n}\cdot p^{r_1}\cdot p^{r+rN-r_1N/(2^{k-1}(k-1))+\epsilon}\cdot p^{-(r-r_1-dr_2)}$$
and
$$X_2=p^{-r(n+1)}\dsum_{\substack{1\leq r_1\leq r\\ (r-r_1)/d< r_2\leq r}}p^{(r-r_2)n}\cdot p^{r_1}\cdot p^{r+rN-r_1N/(2^{k-1}(k-1))+\epsilon}.$$
With our choice of $N$ in the statement of Lemma $\ref{lem5.2}$, a modicum of computation delivers that
$$X_1\ll p^{rN-r+2-N/(2^{k-1}(k-1)))+\epsilon}\ \text{and}\ X_2\ll p^{rN-(n/d)r+n/d-N/(2^{k-1}(k-1))+\epsilon}.$$

It remains to estimate $p^{-r(n-1)}\dsum_{\substack{0\leq r_2\leq r}}\dsum_{\substack{1\leq \g\leq p^r\\(\g,p^r)=p^{r_2}}}T(\g)$ in $(\ref{eqeqeq5.3}).$ By splitting the sum over $\a$ in $T(\g)$ in terms of values of $(\a,p^r)$, we see that 
\begin{equation}\label{6.1111}
\begin{aligned}
    &p^{-r(n-1)}\dsum_{\substack{0\leq r_2\leq r}}\dsum_{\substack{1\leq \g\leq p^r\\(\g,p^r)=p^{r_2}}}T(\g)   \\
    &=p^{-r(n+1)}\dsum_{0\leq r_2\leq r}\dsum_{(\g,p^r)=p^{r_2}}\dsum_{0\leq s\leq r}\dsum_{(p^r,\a)=p^{r-s}}\dsum_{1\leq l_2\leq p^r}e\biggl(\frac{\langle\a, \nu_{d,n}(\g)\rangle l_2}{p^r}\biggr).
\end{aligned}
\end{equation}
On noting that $(\nu_{d,n}(\g)p^{-dr_2},p)=1$ and $(\a p^{-r+s},p)=1,$ we find by applying orthogonality that the innermost sum is turned into
\begin{equation}\label{6.126.12}
\begin{aligned}
    &\dsum_{1\leq l_2\leq p^r}e\biggl(\frac{\langle\a p^{-r+s},\nu_{d,n}(\g)p^{-dr_2}\rangle}{p^r}\cdot p^{dr_2+r-s}\cdot l_2\biggr)\\
    &= \left\{ \begin{aligned}
       & p^r\ \ \text{when}\ \langle \a p^{-r+s},\nu_{d,n}(\g)p^{-dr_2}\rangle\equiv0\ \textrm{mod}\ p^{\max\{0,s-dr_2\}}\\
       & 0\ \ \ \textrm{otherwise} .
    \end{aligned} \right.
\end{aligned}
\end{equation}
By splitting sums over $s$ and $r_2$ according to whether $\max\{0,s-dr_2\}=0$ or $\max\{0,s-dr_2\}\neq 0$, we see from $(\ref{6.1111})$ and $(\ref{6.126.12})$ that
\begin{equation}\label{5.11}
    p^{-r(n-1)}\dsum_{\substack{0\leq r_2\leq r}}\dsum_{\substack{1\leq \g\leq p^r\\(\g,p^r)=p^{r_2}}}T(\g)=p^{-r(n+1)}(U_1+U_2),
\end{equation}
where
$$U_1=\dsum_{\substack{0\leq r_2,s\leq r\\s\leq dr_2}}\dsum_{(\g,p^r)=p^{r_2}}\dsum_{(p^r,\a)=p^{r-s}}p^r\ \text{and}\ U_2=\dsum_{\substack{0\leq r_2,s\leq r\\s>dr_2}}p^r\dsum_{(\g,p^r)=p^{r_2}}V,$$
in which 
$$V=\#\left\{\a\in [1,p^r]^N\middle|\ \begin{aligned}
\langle\a p^{-r+s},\nu_{d,n}(\g)p^{-dr_2}\rangle&\equiv0\ \textrm{mod}\ p^{s-dr_2}\\
(p^{r},\a)&= p^{r-s}\end{aligned}\right\}.$$
We first see that 
\begin{equation}\label{5.12}
    \begin{aligned}
    p^{-r(n+1)}U_1&\leq \dsum_{\substack{0\leq r_2,s\leq r\\ s\leq dr_2}}p^rp^{Ns}p^{(r-r_2)n}p^{-r(n+1)}\leq \dsum_{\substack{0\leq s\leq r\\s/d\leq r_2\leq r}}p^{Ns-r_2n}\ll p^{rN-\lceil r/d\rceil n}.
    \end{aligned}
\end{equation}
Next, we find that
\begin{equation}\label{5.13}
    \begin{aligned}
    p^{-r(n+1)}U_2=p^{-rn}\dsum_{\substack{0\leq r_2,s\leq r\\s>dr_2}}\dsum_{(\g,p^r)=p^{r_2}}V.
    \end{aligned}
\end{equation}
Meanwhile, for $\g$ with $(\g,p^r)=p^{r_2}$ and $\a$ with $(p^r,\a)=p^{r-s}$, one has $(\nu_{d,n}(\g)p^{-dr_2},p)=1$ and $(\a p^{-r+s},p)=1.$ Hence, we infer that
$$V=p^{sN}\cdot p^{-s+dr_2}-p^{(s-1)N}\cdot p^{-s+dr_2+1}.$$
Then, on noting that the number of $\g$ with $(\g,p^r)=p^{r_2}$ is $p^{(r-r_2)n}(1+O(p^{-n})),$ we deduce that
\begin{equation}\label{5.14}
\begin{aligned}
    &\dsum_{(\g,p^r)=p^{r_2}}V=p^{(r-r_2)n}(1+O(p^{-n}))(p^{s(N-1)}\cdot p^{dr_2}-p^{(s-1)(N-1)}\cdot p^{dr_2})\\
    &=(p^{s(N-1)}\cdot p^{rn-r_2n+dr_2}-p^{(s-1)(N-1)}\cdot p^{rn-r_2n+dr_2})(1+O(p^{-n}))\\
    &=p^{s(N-1)}\cdot p^{rn-r_2n+dr_2}(1+O(p^{-n})).
\end{aligned}
\end{equation}
Therefore, on substituting $(\ref{5.14})$ into $(\ref{5.13}),$ we conclude that
\begin{equation}\label{5.15}
\begin{aligned}
     p^{-r(n+1)}U_2&=p^{-rn}\dsum_{\substack{0\leq r_2,s\leq r\\s>dr_2}}p^{s(N-1)}\cdot p^{rn-r_2n+dr_2}(1+O(p^{-n})) \\
     &=p^{r(N-1)}+O(p^{r(N-1)-(n-d)}).
\end{aligned}
\end{equation}
Furthermore, by substituting $(\ref{5.12})$ and $(\ref{5.15})$ into $(\ref{5.11})$, we find that
\begin{equation}\label{5.16}
      p^{-r(n-1)}\dsum_{\substack{0\leq r_2\leq r}}\dsum_{\substack{1\leq \g\leq p^r\\(\g,p^r)=p^{r_2}}}T(\g)= p^{rN-r}+O(p^{rN-\lceil r/d\rceil n}+p^{rN-r-n+d}),
\end{equation}
and substituting $(\ref{6.96.96.9})$ and $(\ref{5.16})$ into $(\ref{eqeqeq5.3})$, it follows that whenever $$N> (k-1)2^{k-1}(2+n/d+n-d),$$ we have
\begin{equation*}
     N_2(p^r)=p^{rN-r}+O(p^{rN-\lceil r/d\rceil n}+p^{rN-r-n+d}).
\end{equation*}
\end{proof}

\bigskip

\subsubsection{Proof of Lemma $\ref{lem5.3}$}\label{subsec5.1.3}
Recall the definition $(\ref{def6.1})$ of $\sigma(\a;p^r)$ and that $P\in \Z[\x]$ is a non-singular form in $N$ variables of degree $k$. The following lemma provides the bound for a variance of $\sigma(\a;p^r)$.

\begin{lem}\label{lem5.3}
Suppose that $p$ is a prime number, and $r$ is a natural number. Then, whenever $n>d$ and $N>(k-1)2^{k-1}(n-d+3)$, one has
\begin{equation*}
    \dsum_{\substack{\a\in [1,p^r]^N\\P(\a)\equiv 0\ \text{mod}\ p^r}}(\sigma(\a;p^r)-1)^2\ll p^{rN-n\lceil r/d\rceil}+p^{rN-r+d-n}.
\end{equation*}
\end{lem}

We provide the proof of Lemma $\ref{lem5.3}$ at the end of this subsection $\ref{subsec5.1.3}$. We note here that this proof is technical and long, and thus we temporarily pause and provide a sketch of the proof of Lemma $\ref{lem5.3}.$ The proof is based on a method combining the classical treatment of major arcs with the strategy used in [$\ref{ref3},$ section 5]. 

First of all, by squaring out, one obtains
\begin{equation}\label{5.1717}
     \dsum_{\substack{\a\in [1,p^r]^N\\P(\a)\equiv 0\ \text{mod}\ p^r}}(\sigma(\a;p^r)-1)^2= \dsum_{\substack{\a\in [1,p^r]^N\\P(\a)\equiv 0\ \text{mod}\ p^r}}(\sigma(\a;p^r)^2-2\sigma(\a;p^r)+1).
\end{equation}
To bound the sums associated with the second and third summands in $(\ref{5.1717}),$ we make use of Lemma $\ref{lem5.1}$ and Lemma $\ref{lem5.2}.$ It remains to estimate the second moment of $\sigma(\a;p^r)$ as $\a$ runs over $1\leq \a\leq p^r$ with $P(\a)\equiv0\ \text{mod}\ p^r$. Define 
\begin{equation*}
    N_3(p^r):= \dsum_{\substack{\a\in [1,p^r]^N\\P(\a)=0\ \textrm{mod}\ p^r}}\sigma(\a;p^r)^2.
\end{equation*}
Then, inverting the order of summations we have 

\begin{equation}\label{eq5.1}
    N_3(p^r)=p^{-2r(n-1)}\dsum_{\g_1,\g_2\in [1,p^r]^n}\# \left\{\a\in [1,p^r]^N \middle|  \ \begin{aligned}
   \langle\a,\nu_{d,n}(\g_i)\rangle&\equiv0\ \text{mod}\ p^r\ (i=1,2)\\
   P(\a)&\equiv0\ \text{mod}\ p^r
    \end{aligned}\right\}.
\end{equation}
By orthogonality, the summand on the right hand side in ($\ref{eq5.1}$) is seen to be
\begin{equation}\label{eq5.2}
    p^{-3r}\dsum_{1\leq l_1,l_2,l_3\leq p^r}\Xi(\g_1,\g_2,l_1,l_2,l_3),
\end{equation}
where
\begin{equation}\label{6.2121}
    \Xi(\g_1,\g_2,l_1,l_2,l_3)=\dsum_{1\leq \a\leq p^r}e\biggl(\frac{P(\a)l_1}{p^r}\biggr)e\biggl(\frac{\langle\a,\nu_{d,n}(\g_1)\rangle l_2}{p^r}\biggr)e\biggl(\frac{\langle\a,\nu_{d,n}(\g_2)\rangle l_3}{p^r}\biggr).
\end{equation}

Then, on substituting $(\ref{eq5.2})$ into $(\ref{eq5.1})$ and by splitting summation over $\g_1$ and $\g_2$ in terms of values of $(\g_1,p^r)$ and $(\g_2,p^r)$, we see that
\begin{equation}\label{5.19}
    N_3(p^r)=p^{-2r(n-1)}\dsum_{0\leq r_2,r_3\leq r}\dsum_{\substack{1\leq \g_1\leq p^r\\(\g_1,p^r)=p^{r_2}}}\dsum_{\substack{1\leq \g_2\leq p^r\\(\g_2,p^r)=p^{r_3}}}(S(\g_1,\g_2)+T(\g_1,\g_2)),
\end{equation}
where 
\begin{equation}\label{6.2323}
S(\g_1,\g_2)=p^{-3r}\dsum_{1\leq r_1\leq r}\dsum_{\substack{(l_1,p^r)=p^{r-r_1}\\1\leq l_2,l_3\leq p^r}}\Xi(\g_1,\g_2,l_1,l_2,l_3)    
\end{equation}
and 
$$T(\g_1,\g_2)=p^{-3r}\dsum_{1\leq l_2,l_3\leq p^r}\Xi(\g_1,\g_2,p^r,l_2,l_3).$$
It is worth noting that the sum of the second summand $T(\g_1,\g_2)$ in $(\ref{5.19})$ gives the main term of $N_3(p^r),$ which is provided in Lemma $\ref{lem5.6}.$ Furthermore, the sum of the first summand $S(\g_1,\g_2)$ in $(\ref{5.19})$ yields the error term of $N_3(p^r),$ which is provided in Lemma $\ref{lem5.5}.$ Therefore, we combine all to prove Lemma $\ref{lem5.3}.$

\bigskip

For the proofs of Lemmas $\ref{lem5.6}$ and $\ref{lem5.5},$ we require two auxiliary lemmas.
In order to describe the first auxiliary lemma, we recall the definition of $N(d,m)=\binom{m+d-1}{d}$ with $m,d\in \N.$ Furthermore, given vectors $\z_1,\z_2\in \Z^{M},$ we use the notation  $\mathcal{G}(\z_1,\z_2)$ for the greatest common divisor of the $2\times 2$ minors of the $M\times 2$ matrix whose columns are the vectors $\z_1,\z_2.$

\begin{lem}\label{lem5.4}
Let $v$ be a non-negative integer. Let $c_1,c_2\in \Z$ with $(c_1,p)=(c_2,p)=1$, and let $t_1,t_2\in \N\cup\{0\}$ with $t_1\leq t_2.$ Suppose that $\h_1,\h_2\in \Z^{m}$ with
$$\h_1=c_1p^{t_1}\f_1,\ \h_2=c_2p^{t_2}\f_2,$$
where vectors $\f_1$ and $\f_2$ are primitive in $\Z^m$. Consider a quantity $\mathcal{N}(\h_1,\h_2)$ defined by
\begin{equation*}
\mathcal{N}(\h_1,\h_2):=\#\left\{\n\in [1,p^{v}]^{N(d,m)}\middle|
\begin{aligned}
\langle\n,\nu_{d,m}(\h_i)\rangle\equiv0\ \ \text{mod}\ p^{v}\ (i=1,2)
\end{aligned}\right\}.    
\end{equation*}
Then, we have 
\begin{equation*}
    \mathcal{N}(\h_1,\h_2)=\begin{aligned}
        &p^{v(N(d,m)-1)}\cdot p^{-\max\{0,v-dt_1\}}\cdot \mathcal{V},
    \end{aligned} 
\end{equation*}
where 
\begin{equation*}
   \mathcal{V}= \left\{ \begin{aligned}
      & p^{dt_2}{\rm{gcd}}(\mathcal{G}(\f_1,\f_2),p^{v-dt_2})\ \ \text{when}\ v-dt_2\geq0\\
      &  p^{v}\ \ \ \ \ \ \ \ \ \ \ \ \ \ \ \ \ \ \ \ \ \ \ \ \ \ \ \ \ \ \text{when}\ v-dt_2<0.
   \end{aligned}
\right.
\end{equation*}
\end{lem}
\begin{proof}
Notice that 
\begin{equation}\label{5.2020}
\mathcal{N}(\h_1,\h_2):=\#\left\{\n\in [1,p^{v}]^{N(d,m)}\middle|
\begin{aligned}
&\langle\n,\nu_{d,m}(c_1\f_1)\rangle\equiv0\ \ \ \ \ \ \ \ \ \text{mod}\ p^{\max\{0,v-dt_1\}}\\
&\langle\n,\nu_{d,m}(p^{t_2-t_1}c_2\f_2)\rangle\equiv0\ \ \text{mod}\ p^{\max\{0,v-dt_1\}}
\end{aligned}\right\}.    
\end{equation}

We first investigate the case $\f_1\neq \pm\f_2.$ Observe that $\f_1\neq \pm\f_2$ implies that $\nu_{d,m}(\f_1)$ and $\nu_{d,m}(\f_2)$ are linearly independent. We find by the Smith normal form theorem that there exists $f\in \Z$ and $T\in GL_{N(d,m)}(\Z)$ with $|\text{det}(T)|=1$ such that
\begin{equation*}
    \biggl(\begin{aligned}
        \nu_{d,m}(&c_1\f_1)\\
        \nu_{d,m}(c_2&p^{t_2-t_1}\f_2)
    \end{aligned}\biggr)T=\left(\begin{aligned}
        1\ \ &0\ \ 0\ \cdots\ 0\\
        f\ \ &S\ \ 0\ \cdots\ 0
    \end{aligned}\right),
\end{equation*}
where $S:=S(\f_1,\f_2)=\mathcal{G}\left(\nu_{d,m}(c_1\f_1),\nu_{d,m}(c_2p^{t_2-t_1}\f_2)\right).$ Thus, one has
\begin{equation*}
    N(\f_1,\f_2)=\#\left\{\n\in [1,p^{v}]^{N(d,m)}\middle|
\left(\begin{aligned}
       1\ \ &0\ \ 0\ \cdots\ 0\\
        f\ \ &S\ \ 0\ \cdots\ 0
    \end{aligned}\right)\n^{\top}\equiv\0\ \  \text{mod}\ p^{\max\{0,v-dt_1\}}\right\}.
\end{equation*}
By [$\ref{ref3}$, Lemma 3.11], we have
\begin{equation*}
    \begin{aligned}
        \mathcal{G}(\nu_{d,m}(c_1\f_1),\nu_{d,m}(c_2p^{(t_2-t_1)}\f_2))&=c_1^d\cdot c_2^d\cdot p^{(t_2-t_1)d}\cdot\mathcal{G}(\nu_{d,m}(\f_1),\nu_{d,m}(\f_2))\\
        &=c_1^d\cdot c_2^d\cdot p^{(t_2-t_1)d}\cdot\mathcal{G}(\f_1,\f_2).
    \end{aligned}
\end{equation*}
Hence, on recalling that $(c_1,p)=(c_2,p)=1$, we find that 
\begin{equation}\label{5.2121}
    \mathcal{N}(\h_1,\h_2)=(p^{v})^{N(d,m)-2}\cdot \mathcal{Z}_1\cdot \mathcal{Z}_2,
\end{equation}
where
$$\mathcal{Z}_1=\#\{1\leq n_1\leq p^{v}|\ n_1\equiv0\ \text{mod}\ p^{\max\{0,v-dt_1\}}\}$$
and 
$$\mathcal{Z}_2=\#\{0\leq n_2\leq p^{v}-1|\  p^{(t_2-t_1)d}\cdot\mathcal{G}(\f_1,\f_2)\cdot n_2\equiv0\  \text{mod}\ p^{\max\{0,v-dt_1\}}\}.$$

First, consider the case $v-dt_2\geq 0$. Since $v-dt_2\geq 0.$ Since $t_1\leq t_2,$ one has $v-dt_1\geq 0.$ Then, the condition $p^{(t_2-t_1)d}\cdot\mathcal{G}(\f_1,\f_2)\cdot n_2\equiv0\  \text{mod}\ p^{\max\{0,v-dt_1\}}$
implies that $$\mathcal{G}(\f_1,\f_2)\cdot n_2\equiv0\ \text{mod}\ p^{v-dt_2}.$$ Hence, we deduce that
\begin{equation*}
    \begin{aligned}
    \mathcal{Z}_2&= p^{v}(p^{v-dt_2}(\text{gcd}(\mathcal{G}(\f_1,\f_2),p^{v-dt_2}))^{-1})^{-1}\\
    &= p^{dt_2}\text{gcd}(\mathcal{G}(\f_1,\f_2),p^{v-dt_2}).
    \end{aligned}
\end{equation*}

Next, consider the case $v-dt_2<0$. Obviously, we have
$\mathcal{Z}_2= p^{v}.$
Therefore, on noting that $\mathcal{Z}_1=p^{v}\cdot p^{-\max\{0,v-dt_1\}}
$, it follows from $(\ref{5.2121})$ that 
\begin{equation}\label{5.2222}
    \mathcal{N}(\h_1,\h_2)= (p^{v})^{N(d,m)-1}\cdot p^{-\max\{0,v-dt_1\}}\cdot \mathcal{V},
\end{equation}
where 
\begin{equation*}
   \mathcal{V}= \left\{ \begin{aligned}
      & p^{dt_2}\textrm{gcd}(\mathcal{G}(\f_1,\f_2),p^{v-dt_2})\ \text{when}\ v-dt_2\geq0\\
      &  p^{v}\ \ \ \ \ \ \ \ \ \ \ \ \ \ \ \ \ \ \ \ \ \ \ \ \ \ \ \ \ \ \ \text{when}\ v-dt_2<0
   \end{aligned}
\right.
\end{equation*}

Next, we analyse $N(\f_1,\f_2)$ with $\f_1=\pm \f_2.$ On recalling that $\f_1$ is primitive and noting that $p\nmid\f_1$ implies $p\nmid \nu_{d,m}(\f_1),$ we see from ($\ref{5.2020}$) that \begin{equation}\label{5.2323}
    \begin{aligned}
        \mathcal{N}(\h_1,\h_2)&=\#\{\n\in [1,p^v]^{N(d,m)}|\ \langle\n,\nu_{d,m}(\f_1)\rangle\equiv0\ \text{mod}\ p^{\max\{0,v-dt_1\}}\}\\
        &= p^{vN(d,m)}\cdot p^{-\max\{0,v-dt_1\}}.
    \end{aligned}
\end{equation}
Then, on noting that $\mathcal{V}=p^v$ when $\f_1=\pm \f_2$, we complete the proof of Lemma $\ref{lem5.4}$, by $(\ref{5.2222})$ and $(\ref{5.2323})$.
\end{proof}

\bigskip

We make use of Lemma $\ref{lem5.4}$ in order to prove the second auxiliary lemma, which  will be used in the proofs of Lemmas $\ref{lem5.6}$ and $\ref{lem5.5}.$ In advance of the statement of Lemma $\ref{lem6.6}$, we recall the definition $(\ref{1.21.2})$ of $N.$ Furthermore, for $z\in \N\cup \{0\}$ and $\g_1,\g_2\in \Z^n$, we define 
 $$\mathcal{W}_z(\g_1,\g_2)=\#\left\{ \a\in [0,p^z-1]^N\middle|\ \langle\a,\nu_{d,n}(\g_i)\rangle\equiv 0\ \textrm{mod}\ p^z\ (i=1,2)\right\}.$$
\begin{lem}\label{lem6.6}
    Suppose that $r$ and $z$ are non-negative integers with $z\leq r$. Then, whenever $n>d$, one has
    \begin{equation}\label{6.2888}
        \dsum_{0\leq r_2,r_3\leq r}\dsum_{\substack{1\leq \g_1\leq p^r\\ (\g_1,p^r)=p^{r_2}\\\g_1\in \Z^n}}\dsum_{\substack{1\leq \g_2\leq p^r\\(\g_2,p^r)=p^{r_3}\\\g_2\in \Z^n}}\mathcal{W}_z(\g_1,\g_2)= p^{z(N-2)+2rn}\left(1+O\left(p^{d-n}+p^{z-n\lceil z/d\rceil}\right)\right).
    \end{equation}
\end{lem}
\begin{proof}
For $\g_1\in \Z^n$ and $\g_2\in \Z^n$ with $(\g_1,p^r)=p^{r_2}$ and $(\g_2,p^r)=p^{r_3}$, we see that there exist $c_1,c_2\in \Z$ with $(c_1,p)=(c_2,p)=1$ and primitive vectors $\f_1$ and $\f_2$ in $\Z^n$ such that
\begin{equation}\label{6.29292929}
    \g_1=c_1p^{r_2}\f_1,\ \g_2=c_2p^{r_3}\f_2.
\end{equation}
Then, in the case that $r_2\leq r_3$, by applying Lemma $\ref{lem5.4}$ with $t_1=r_2,t_2=r_3,m=n$ and $v=z$, one has
\begin{equation}\label{6.292929}
    \mathcal{W}_z(\g_1,\g_2)=\begin{aligned}
        &p^{z(N-1)}\cdot p^{-\max\{0,z-dr_2\}}\cdot \mathcal{V}_1,
    \end{aligned}
\end{equation}
where 
\begin{equation*}
   \mathcal{V}_1= \left\{ \begin{aligned}
      & p^{dr_3}\text{gcd}(\mathcal{G}(\f_1,\f_2),p^{z-dr_3})\ \text{when}\ z-dr_3\geq0\\
      &  p^{z}\ \ \ \ \ \ \ \ \ \ \ \ \ \ \ \ \ \ \ \ \ \ \ \ \ \ \ \ \ \ \ \ \text{when}\ z-dr_3<0.
   \end{aligned}
\right.
\end{equation*}

Define 
$$\Sigma_1=\dsum_{\substack{(\g_1,p^r)=1\\(\g_2,p^r)=1}}\mathcal{W}_z(\g_1,\g_2),$$
\begin{equation}\label{6.313131}
\Sigma_2=\dsum_{\substack{0\leq r_2,r_3\leq r\\r_2\leq r_3\\(r_2,r_3)\neq (0,0))}}\dsum_{\substack{(\g_1,p^r)=p^{r_2}\\(\g_2,p^r)=p^{r_3}}}\mathcal{W}_z(\g_1,\g_2)\end{equation}
and 
$$\Sigma_3=\dsum_{\substack{0\leq r_2,r_3\leq r\\ r_2>r_3\\ (r_2,r_3)\neq (0,0)}}\dsum_{\substack{(\g_1,p^r)=p^{r_2}\\(\g_2,p^r)=p^{r_3}}}\mathcal{W}_z(\g_1,\g_2).$$
On writing the left-hand side in $(\ref{6.2888})$ by $\mathfrak{T}$, we notice here that 
\begin{equation}\label{6.32323232}
    \mathfrak{T}=\Sigma_1+\Sigma_2+\Sigma_3.
\end{equation}

We first analyze $\Sigma_1.$ When $r_2=r_3=0$, it follows from $(\ref{6.292929})$ that 
$$\mathcal{W}_z(\g_1,\g_2)=p^{zN-2z}\text{gcd}(\mathcal{G}(\f_1,\f_2),p^{z}).$$
 Hence, we find that 
$$\Sigma_1=\dsum_{\substack{(\g_1,p)=1\\(\g_2,p)=1}}p^{zN-2z}\text{gcd}(\mathcal{G}(\f_1,\f_2),p^{z}).$$
By introducing 
$$\mathfrak{F}^{(e)}(p^z)=\left\{(\g_1,\g_2)\in [1,p^r]^n\times[1,p^r]^n\middle|\ \begin{aligned}
    &\text{gcd}(\mathcal{G}(\f_1,\f_2),p^{z})=p^e\\
    &(\g_1,p)=(\g_2,p)=1
\end{aligned}\right\},$$
one finds that 
$$\Sigma_1=p^{zN-2z}\biggl(\dsum_{\substack{(\g_1,p)=1\\ (\g_2,p)=1}}1+O\biggl(\dsum_{1\leq e\leq z}p^e\cdot \#\mathfrak{F}^{(e)}(p^z)\biggr)\biggr).$$
Note that $\sum_{\substack{(\g_1,p)=1\\ (\g_2,p)=1}}1=p^{2rn}(1+O(p^{-n})).$ Furthermore, for given $\g_2\in [1,p^r]^n$ with $(\g_2,p)=1$, one infers that the number of $\g_1\in [1,p^r]^n$ such that $(\g_1,\g_2)\in \mathfrak{F}^{(e)}(p^z)$ is at most $p^{rn-e(n-1)}.$ Thus, this gives $\#\mathfrak{F}^{(e)}(p^z)\ll p^{2rn-e(n-1)}.$ Therefore, we deduce that 
\begin{equation}\label{6.333}
    \Sigma_1=p^{z(N-2)+2rn}(1+O(p^{-n+2})).
\end{equation}

Next, we turn to estimate $\Sigma_2.$ When $dr_2\leq z<dr_3$, we find from $(\ref{6.292929})$ that
\begin{equation}\label{6.323232}
    \begin{aligned}
         \mathcal{W}_z(\g_1,\g_2)&=p^{z(N-1)}\cdot p^{dr_2}.
    \end{aligned}
\end{equation}
Furthermore, when $z<dr_2\leq dr_3$, one finds from $(\ref{6.292929})$ again that 
\begin{equation}\label{6.3333}
    \begin{aligned}
         \mathcal{W}_z(\g_1,\g_2)&= p^{zN}.
    \end{aligned}
\end{equation}
On noting that 
$$\#\{(\g_1,\g_2)\in [1,p^r]^{2n}|\ (\g_1,p^r)=p^{r_2},\ (\g_2,p^r)=p^{r_3}\}\ll (p^{r-r_2})^n\cdot (p^{r-r_3})^n,$$
we find from $(\ref{6.323232})$ and $(\ref{6.3333})$ that
\begin{equation*}
    \begin{aligned}
        &\dsum_{\substack{0\leq r_2,r_3\leq r\\r_2\leq r_3\\z<dr_3}}\dsum_{\substack{(\g_1,p^r)=p^{r_2}\\(\g_2,p^r)=p^{r_3}}}\mathcal{W}_z(\g_1,\g_2)\\
        &\ll \dsum_{\substack{0\leq r_2,r_3\leq r\\r_2\leq r_3\\dr_2\leq z<dr_3}}(p^{r-r_2})^n\cdot (p^{r-r_3})^n\cdot p^{z(N-1)+dr_2}+ \dsum_{\substack{0\leq r_2,r_3\leq r\\r_2\leq r_3\\z<dr_2\leq dr_3}}(p^{r-r_2})^n\cdot (p^{r-r_3})^n\cdot p^{zN}.
    \end{aligned}
\end{equation*}
A modicum of calculation leads to 
\begin{equation}\label{6.353535}
    \begin{aligned}
         &\dsum_{\substack{0\leq r_2,r_3\leq r\\r_2\leq r_3\\z<dr_3}}\dsum_{\substack{(\g_1,p^r)=p^{r_2}\\(\g_2,p^r)=p^{r_3}}}\mathcal{W}_z(\g_1,\g_2)\\
         &\ll \dsum_{\substack{0\leq r_3\leq r\\ z<dr_3}}(p^{r})^n\cdot (p^{r-r_3})^n\cdot p^{z(N-1)}+\dsum_{\substack{0\leq r_3\leq r\\ z<dr_3}}\left(p^{r-\lceil z/d\rceil }\right)^n\cdot (p^{r-r_3})^n\cdot p^{zN}\\
         &\ll \dsum_{\substack{0\leq r_3\leq r\\ \lceil z/d\rceil\leq r_3}}(p^{z(N-1)+2rn-nr_3}+p^{zN+2rn-n\lceil z/d\rceil-nr_3})\\
        & \ll p^{z(N-1)+2rn-n\lceil z/d\rceil},
    \end{aligned}
\end{equation}
where we have used the inequality $z-n\lceil z/d\rceil<0$ derived by the hypothesis $n>d$ in the statement of Lemma $\ref{lem6.6}$.

Consider the case that $dr_3\leq z.$ Then, it follows from $(\ref{6.292929})$ again that 
we see that
\begin{equation}\label{6.343434}
    \mathcal{W}_z(\g_1,\g_2)=p^{z(N-2)}\cdot p^{dr_2+dr_3}\cdot  \text{gcd}(\mathcal{G}(\f_1,\f_2),p^{z-dr_3}).
\end{equation}

Let us temporarily define
$$G=\{(\g_1,\g_2)\in [1,p^r]^{2n}|\ (\g_1,p^r)=p^{r_2},\ (\g_2,p^r)=p^{r_3}\}.$$
For given $(\g_1,\g_2)\in G$, recall the definition $(\ref{6.29292929})$ of $\f_1$ and $\f_2.$ Then, by introducing 
$$\mathfrak{F}^{(e)}=\left\{(\g_1,\g_2)\in G\middle|\ \text{gcd}(\mathcal{G}(\f_1,\f_2),p^{z-dr_3})=p^e\right\},$$
we deduce from $(\ref{6.343434})$ that 
$$\dsum_{\substack{(\g_1,p^r)=p^{r_2}\\ (\g_2,p^r)=p^{r_3}}}\mathcal{W}_z(\g_1,\g_2)=p^{z(N-2)}\cdot p^{dr_2+dr_3}\cdot \biggl(|G|+O\biggl(\dsum_{1\leq e\leq z-dr_3}p^e\cdot \#\mathfrak{F}^{(e)}\biggr)\biggr).$$
Note that $|G|=p^{(r-r_2)n+(r-r_3)n}(1+O(p^{-n})).$ Furthermore, for given $\g_2\in [1,p^r]^n$  with $(\g_2,p^r)=p^{r_3}$, the number of $\g_1\in [1,p^r]^n$ with $(\g_1,p^r)=p^{r_2}$ such that $(\g_1,\g_2)\in \mathfrak{F}^{(e)}$ is at most $(p^{r-r_2})^n\cdot p^{-e(n-1)}.$  Thus, this gives $\#\mathfrak{F}^{(e)}\ll p^{2rn-r_1n-r_2n-e(n-1)}.$ Therefore, we deduce that 
$$\dsum_{\substack{(\g_1,p^r)=p^{r_2}\\ (\g_2,p^r)=p^{r_3}}}\mathcal{W}_z(\g_1,\g_2)=p^{z(N-2)}\cdot p^{dr_2+dr_3}\cdot p^{(r-r_2)n+(r-r_3)n}(1+O(p^{-n+2})).$$
This equality yields that
\begin{equation}\label{6.363636}
    \begin{aligned}
        &\dsum_{\substack{0\leq r_2,r_3\leq r\\ r_2\leq r_3\\ dr_3\leq z\\(r_2,r_3)\neq (0,0)}}\dsum_{\substack{(\g_1,p^r)=p^{r_2}\\(\g_2,p^r)=p^{r_3}}}\mathcal{W}_z(\g_1,\g_2)\\
        &=\dsum_{\substack{0\leq r_2,r_3\leq r\\ r_2\leq r_3\\ dr_3\leq z\\(r_2,r_3)\neq (0,0)}}p^{z(N-2)}\cdot p^{dr_2+dr_3}\cdot p^{(r-r_2)n+(r-r_3)n}(1+O(p^{-n+2}))\\
        &\ll p^{z(N-2)+2rn+d-n}.
    \end{aligned}
\end{equation}

By adding the bounds in $(\ref{6.353535})$ and $(\ref{6.363636})$ and substituting that into $(\ref{6.313131})$, we find that
\begin{equation}\label{6.393939}
    \Sigma_2\ll p^{z(N-2)+2rn}\left(p^{z-n\lceil z/d\rceil}+p^{d-n}\right).
\end{equation}
Furthermore, the same argument leading from ($\ref{6.323232}$) and $(\ref{6.393939})$ delivers that 
\begin{equation}\label{6.404040}
\Sigma_3\ll p^{z(N-2)+2rn}\left(p^{z-n\lceil z/d\rceil}+p^{d-n}\right),
\end{equation}
by reversing the roles of $r_2$ and $r_3.$ 

Therefore, on substituting $(\ref{6.333})$, $\ref{6.393939}$ and $\ref{6.404040}$ into $(\ref{6.32323232})$, we conclude that 
$$\mathfrak{T}=p^{z(N-2)+2rn}(1+O(p^{d-n}+p^{z-n\lceil z/d\rceil})).$$
\end{proof}

\bigskip

In advance of the statement of the following lemma, we recall the formulation ($\ref{5.19}$) of $N_3(p^r)$.
\begin{lem}\label{lem5.6} 
Whenever $n>d$, we have
\begin{equation*}
  p^{-2r(n-1)}\dsum_{0\leq r_2,r_3\leq r}\dsum_{\substack{1\leq \g_1\leq p^r\\(\g_1,p^r)=p^{r_2}}}\dsum_{\substack{1\leq \g_2\leq p^r\\(\g_2,p^r)=p^{r_3}}}T(\g_1,\g_2)=p^{rN-r}(1+O(p^{d-n}+p^{r-n\lceil r/d\rceil})).    
\end{equation*}
\end{lem}

\begin{proof}
For simplicity, we write
\begin{equation}\label{5.42}
   \mathfrak{T} =p^{-2r(n-1)}\dsum_{0\leq r_2,r_3\leq r}\dsum_{\substack{1\leq \g_1\leq p^r\\(\g_1,p^r)=p^{r_2}}}\dsum_{\substack{1\leq \g_2\leq p^r\\(\g_2,p^r)=p^{r_3}}}T(\g_1,\g_2).
\end{equation}
 On recalling the definition of $T(\g_1,\g_2)$ leading to $(\ref{5.19})$, we see that
\begin{equation}\label{5.44}
   \mathfrak{T}= p^{-2rn-r}\dsum_{0\leq r_2,r_3\leq r}\dsum_{\substack{(\g_1,p^r)=p^{r_2}\\(\g_2,p^r)=p^{r_3}}}\dsum_{1\leq \a\leq p^r}\mathcal{T}_1(\a,\g_1,\g_2),
\end{equation}
where $$\mathcal{T}_1(\a,\g_1,\g_2)=\dsum_{1\leq l_2,l_3\leq p^r}e\biggl(\frac{\langle \a,\nu_{d,n}(\g_1)\rangle l_2}{p^r}\biggr)\biggl(\frac{\langle \a,\nu_{d,n}(\g_2)\rangle l_3}{p^r}\biggr).$$
It follows by orthogonality that for given $\g_1$ and $\g_2$ with $(\g_1,p^r)=p^{r_2}$ and $(\g_2,p^r)=p^{r_3}$, one has
\begin{equation}\label{6.28}
    \dsum_{1\leq \a\leq p^r}\mathcal{T}_1(\a,\g_1,\g_2)=p^{2r}\#\left\{1\leq \a\leq p^r\middle| \begin{aligned}
        \langle \a,\nu_{d,n}(\g_i)\rangle\equiv0\ \text{mod}\ p^r\ (i=1,2)
    \end{aligned}\right\}.
\end{equation}
Notice here that $ \sum_{1\leq \a\leq p^r}\mathcal{T}_1(\a,\g_1,\g_2)=p^{2r}\cdot \mathcal{W}_r(\g_1,\g_2)$, where $\mathcal{W}_r(\g_1,\g_2)$ is defined in the preamble to Lemma $\ref{lem6.6}$.

Therefore, by substituting $(\ref{6.28})$ into $(\ref{5.44})$ and by applying Lemma $\ref{lem6.6}$ with $z=r$, we see that
\begin{equation*}
    \begin{aligned}
         \mathfrak{T}&=p^{-2rn+r}\cdot p^{r(N-2)+2rn}(1+O(p^{d-n}+p^{r-n\lceil r/d\rceil}))\\
         &=p^{rN-r}(1+O(p^{d-n}+p^{r-n\lceil r/d\rceil})).
    \end{aligned}
\end{equation*}
This completes the proof of Lemma $\ref{lem5.6}.$

\end{proof}

In advance of the statement of the following lemma, we recall the formulation ($\ref{5.19}$) of $N_3(p^r)$.
\begin{lem}\label{lem5.5}
Whenever $n>d$ and $N>(k-1)2^{k-1}(n-d+3)$, one has
$$p^{-2r(n-1)}\dsum_{0\leq r_2,r_3\leq r}\dsum_{\substack{1\leq \g_1\leq p^r\\(\g_1,p^r)=p^{r_2}}}\dsum_{\substack{1\leq \g_2\leq p^r\\(\g_2,p^r)=p^{r_3}}}S(\g_1,\g_2)\ll p^{rN-r-(N/(2^{k-1}(k-1))-3)+\epsilon}.$$

\end{lem}

\begin{proof}
 On recalling that $\gcd(l_1,p^r)=p^{r-r_1}$, let us write $l_1=p^{r-r_1}\widetilde{l}_1$ with $\gcd(\widetilde{l}_1,p)=1.$ Then, on recalling the definition $(\ref{6.2121})$ of $\Xi(\g_1,\g_2,l_1,l_2,l_3)$ and by writing $\a=p^{r_1}\n+\m$ with $0\leq \n\leq p^{r-r_1}-1$ and $1\leq \m\leq p^{r_1}$, we see that 
\begin{equation*}
\begin{aligned}
   &\Xi(\g_1,\g_2,l_1,l_2,l_3)\\
&=\dsum_{0\leq \n\leq p^{r-r_1}-1}\dsum_{1\leq \m\leq p^{r_1}} e\biggl(\frac{P(\m)\widetilde{l}_1}{p^{r_1}}\biggr)e\biggl(\frac{\langle p^{r_1}\n+\m,\nu_{d,n}(\g_1)\rangle l_2+\langle p^{r_1}\n+\m,\nu_{d,n}(\g_2)\rangle l_3}{p^r}\biggr).
\end{aligned}
\end{equation*}

On recalling ($\ref{5.19}$) and $(\ref{6.2323})$ and by substituting this expression for $\Xi(\g_1,\g_2,l_1,l_2,l_3)$ into $S(\g_1,\g_2)$, it follows by the triangle inequality that 
\begin{equation}\label{5.20}
    \begin{aligned}
   & p^{-2r(n-1)}\dsum_{0\leq r_2,r_3\leq r}\dsum_{\substack{1\leq \g_1\leq p^r\\(\g_1,p^r)=p^{r_2}}}\dsum_{\substack{1\leq \g_2\leq p^r\\(\g_2,p^r)=p^{r_3}}}S(\g_1,\g_2)\\
    &\leq p^{-2rn-r}\dsum_{\substack{1\leq r_1\leq r\\0\leq r_2,r_3\leq r}}\dsum_{\substack{1\leq \g_1\leq p^r\\(\g_1,p^r)=p^{r_2}}}\dsum_{\substack{1\leq \g_2\leq p^r\\(\g_2,p^r)=p^{r_3}}}\dsum_{\substack{1\leq \widetilde{l}_1\leq p^{r_1}\\(\widetilde{l}_1,p)=1}}\bigl|S_1(\widetilde{l}_1,\g_1,\g_2)\bigr|\cdot\bigl|S_2(\g_1,\g_2)\bigr|,
    \end{aligned}
\end{equation}
where \begin{equation*}
\begin{aligned}
    &S_1(\widetilde{l}_1,\g_1,\g_2)\\
    &=\sup_{1\leq l_2,l_3\leq p^r}\biggl|\dsum_{1\leq \m\leq p^{r_1}}e\biggl(\frac{P(\m)\widetilde{l}_1}{p^{r_1}}\biggr)e\biggl(\frac{\langle \m,\nu_{d,n}(\g_1)\rangle l_2}{p^r}\biggr)e\biggl(\frac{\langle\m,\nu_{d,n}(\g_2)\rangle l_3}{p^r}\biggr)\biggr|
\end{aligned}
\end{equation*}
and 
\begin{equation*}
    S_2(\g_1,\g_2)=\dsum_{1\leq l_2,l_3\leq p^r}\biggl|\dsum_{0\leq \n\leq p^{r-r_1}-1}e\biggl(\frac{\langle\n,\nu_{d,n}(\g_1)\rangle l_2}{p^{r-r_1}}\biggr)e\biggl(\frac{\langle\n,\nu_{d,n}(\g_2)\rangle l_3}{p^{r-r_1}}\biggr)\biggr|.
\end{equation*}

First, by the same explanation leading to $(\ref{5.55.5})$, one has
\begin{equation}\label{5.21}
    S_1(\widetilde{l}_1,\g_1,\g_2)\ll p^{Nr_1-Nr_1/(2^{k-1}(k-1))+\epsilon}.
\end{equation}
Next, we analyse $S_2(\g_1,\g_2).$ We infer from orthogonality that the inner sum of $S_2(\g_1,\g_2)$ is a non-negative integer. Thus, by changing the order of summations, we have
\begin{equation*}
    S_2(\g_1,\g_2)=\dsum_{0\leq \n\leq p^{r-r_1}-1}\dsum_{1\leq l_2,l_3\leq p^r}e\biggl(\frac{\langle\n,\nu_{d,n}(\g_1)\rangle l_2}{p^{r-r_1}}\biggr)e\biggl(\frac{\langle\n,\nu_{d,n}(\g_2)\rangle l_3}{p^{r-r_1}}\biggr).
\end{equation*}
Hence, by orthogonality, we find that
\begin{equation}\label{6.4646}
    S_2(\g_1,\g_2)=p^{2r}\#\left\{0\leq \n\leq p^{r-r_1}-1\middle| \begin{aligned}
        \langle \n,\nu_{d,n}(\g_i)\rangle\equiv0\ \text{mod}\ p^{r-r_1}\ (i=1,2)
    \end{aligned}\right\}.
\end{equation}
Notice here that $  S_2(\g_1,\g_2)=p^{2r}\cdot \mathcal{W}_{r-r_1}(\g_1,\g_2)$, where $\mathcal{W}_{r-r_1}(\g_1,\g_2)$ is defined in the preamble to Lemma $\ref{lem6.6}.$

Therefore, we deduce by applying Lemma $\ref{lem6.6}$ with $z=r-r_1$ that 
\begin{equation}\label{6.47}
    \begin{aligned}
        &\dsum_{0\leq r_2,r_3\leq r}\dsum_{\substack{1\leq \g_1,\g_2\leq p^r\\ (\g_1,p^r)=p^{r_2}\\(\g_2,p^r)=p^{r_3}}}S_2(\g_1,\g_2)\\
        &=p^{2r}\cdot p^{(r-r_1)(N-2)+2rn}\left(1+O\left(p^{d-n}+p^{r-r_1-n\lceil (r-r_1)/d\rceil}\right)\right)\\
        &\ll p^{2r}\cdot p^{(r-r_1)(N-2)+2rn}=p^{(r-r_1)N+2rn+2r_1},
    \end{aligned}
\end{equation}
where we have used the inequality $r-r_1-n\lceil(r-r_1)/d\rceil<0$ derived from the hypothesis $n>d.$

By substituting ($\ref{5.21}$) and $(\ref{6.47})$ into $(\ref{5.20})$ and by recalling the hypothesis $N>(k-1)2^{k-1}(n-d+3)$, we conclude that 
\begin{equation*}
\begin{aligned}
  & p^{-2r(n-1)} \dsum_{0\leq r_2,r_3\leq r}\dsum_{\substack{1\leq \g_1,\g_2\leq p^r\\ (\g_1,p^r)=p^{r_2}\\(\g_2,p^r)=p^{r_3}}}S(\g_1,\g_2)\\
  &\leq p^{-2rn-r} \dsum_{1\leq r_1\leq r}\dsum_{\substack{1\leq \widetilde{l_1}\leq p^{r_1}\\ (\widetilde{l_1},p)=1}}p^{Nr_1-Nr_1/(2^{k-1}(k-1))+\epsilon}\cdot p^{(r-r_1)N+2rn+2r_1}\\
  &\ll p^{rN-r-(N/(2^{k-1}(k-1))-3)+\epsilon}.
\end{aligned}
\end{equation*}This completes the proof of Lemma $\ref{lem5.5}.$\end{proof}

\bigskip

We now provide the proof of Lemma $\ref{lem5.3}$ by making use of the method described at the beginning of section $\ref{subsec5.1.3}$, following the statement of Lemma $\ref{lem5.3}.$
\begin{proof}[Proof of Lemma $\ref{lem5.3}$]
Recall ($\ref{5.1717}$), that is 
\begin{equation*}
     \dsum_{\substack{1\leq \a\leq p^r\\P(\a)\equiv 0\ \text{mod}\ p^r}}(\sigma(\a;p^r)-1)^2= \dsum_{\substack{1\leq \a\leq p^r\\P(\a)\equiv 0\ \text{mod}\ p^r}}(\sigma(\a;p^r)^2-2\sigma(\a;p^r)+1).
\end{equation*}
 On substituting the bounds from Lemma $\ref{lem5.6}$ and Lemma $\ref{lem5.5}$ into $(\ref{5.19}),$ it follows that
whenever $N>(k-1)2^{k-1}(n-d+3)$, one has
\begin{equation}\label{5.57}
    \dsum_{\substack{1\leq \a\leq p^r\\P(\a)\equiv 0\ \text{mod}\ p^r}}\sigma(\a;p^r)^2=p^{rN-r}\bigl(1+O\bigl(p^{d-n}+p^{r-n\lceil r/d\rceil}\bigr)\bigr).
\end{equation}
Therefore, by adding $(\ref{5.57})$ and the expressions obtained in Lemma $\ref{lem5.1}$ and Lemma $\ref{lem5.2}$, we deduce that
\begin{equation*}
    \dsum_{\substack{1\leq \a\leq p^r\\P(\a)\equiv 0\ \text{mod}\ p^r}}(\sigma(\a;p^r)-1)^2= O(p^{rN-n\lceil r/d\rceil}+p^{rN-r+d-n}).
\end{equation*}
This completes the proof of Lemma $\ref{lem5.3}.$
\end{proof}

\bigskip

\subsection{Singular series treatment II}\label{sec6.26.2}
In this section, we provide an upper bound for $\#A_v^{(e)}(p^r)$, and give a lower bound for $\sigma (\a;p^r)$ with $\a\in A_v^{(e)}(p^r).$ Furthermore, combining all the estimates obtained in sections $\ref{sec6.1}$ and $\ref{sec6.26.2}$ together with the strategy used in  [$\ref{ref3},$ section 5], we shall prove Proposition $\ref{prop6.1}$ at the end of this section. 
\subsubsection{The proof of Lemma $\ref{lem6.8}$}\label{6.1.4}
 We recall the definition of the condition $\mathcal{C}_v^{(e)}(p^r)$ with $0\leq e\leq r-v$ in the preamble to the definition  $(\ref{6.59590})$ of $A_v^{(e)}(p^r)$, that is
\begin{equation*}
\begin{aligned}
A_v^{(e)}(p^r)=\left\{\a\in [1,p^r]^N\middle|\ \a \ \text{satisfies the condition}\ \mathcal{C}_v^{(e)}(p^r)\right\}.    
\end{aligned}
\end{equation*}
Recall that $P\in \Z[\x]$ is a non-singular form in $N$ variables of degree $k.$
\begin{lem}\label{lem6.8}
Let $d\geq3.$ Let $p$ be a prime number and $r\geq 1.$ For $e\in \{0,\ldots,r\}$, whenever $N>(k-1)2^{k-1}(n+3)$ we have
\begin{equation*}
    \# A_v^{(e)}(p^r)\ll p^{rN-r-e}.
\end{equation*}
\end{lem}

\bigskip

We shall prove Lemma $\ref{lem6.8}$ at the end of section $\ref{6.1.4}.$ In order to prove this lemma, we require an auxiliary lemma.  In advance of describing this auxiliary lemma, it is convenient to define
\begin{equation*}
    M_1(r,v,h;k)=\dsum_{\substack{\b\in [1,p^{r-v}]^{N}\\ P(\b)\equiv 0\ \text{mod}\ p^{r-kv}}}p^{-h}\cdot \#\left\{\x\in \mathcal{R}_n(p^h)\middle|\begin{aligned}
   (\nabla f_{\b}(\x),f_{\b}(\x))\equiv \boldsymbol{0}\ \text{mod}\ p^h
\end{aligned}\right\},
\end{equation*}
with $r,v,h \geq0,$ $r-v\geq h$ and $r-kv\geq0.$
\begin{lem}\label{lem6.96.9}
Let $k\in\N,$ and let $r,v,h$ be non-negative integers with  $r-v\geq h$ and $r- kv\geq 0$. Then, whenever $N>(k-1)2^{k-1}(n+3)$, one has
\begin{equation}\label{6.59}
M_1(r,v,h;k)\ll p^{-h}\cdot p^{kv-r}\cdot p^{(r-v)N}.
\end{equation}
\end{lem}
\begin{proof}
By orthogonality, one finds that
\begin{equation}\label{6.60}
\begin{aligned}
M_1(r,v,h;k)=p^{-(n+2)h-r}p^{kv}\dsum_{\x\in \mathcal{R}_n(p^h)}\dsum_{1\leq l_1\leq p^{r-kv}}\dsum_{\substack{1\leq \l\leq p^h\\\l\in \Z^{n+1}}}\Xi(\x,l_1,\l),
\end{aligned}
\end{equation}
where 
\begin{equation*}
    \Xi(\x,l_1,\l)=\dsum_{1\leq \b\leq p^{r-v}}e\left(\frac{P(\b)l_1}{p^{r-kv}}\right)e\left(\frac{\langle(\nabla f_{\b}(\x),f_{\b}(\x)),\l\rangle}{p^h}\right).
\end{equation*}
Furthermore, on splitting the sum over $l_1$ in $(\ref{6.60})$ in terms of values of $(l_1,p^{r-kv})$, we see that
\begin{equation}\label{6.61}
    M_1(r,v,h;k)=p^{-(n+2)h-r}p^{kv}\dsum_{\x\in \mathcal{R}_n(p^h)}(S(\x)+T(\x)),
\end{equation}
where 
$$S(\x)=\dsum_{1\leq r_1\leq r-kv}\dsum_{\substack{1\leq l_1\leq p^{r-kv}\\(l_1,p^{r-kv})=p^{r-kv-r_1}}}\dsum_{\substack{1\leq \l\leq p^h\\\l\in \Z^{n+1}}}\Xi(\x,l_1,\l)\ \text{and}\ T(\x)=\dsum_{\substack{1\leq \l\leq p^h\\\l\in \Z^{n+1}}}\Xi(\x,p^{r-kv},\l).$$

We first analyze $S(\x)$. For fixed $r_1$ with $p^{r-kv-r_1}=(l_1,p^{r-kv})$, let us write $l_1=p^{r-kv-r_1}\widetilde{l}_1.$ Then, on writing $\b=p^{r_1}\n+\m$ with $0\leq \n\leq p^{r-v-r_1}-1$ and $1\leq \m\leq p^{r_1}$, we see that
\begin{equation*}
\begin{aligned}
    \Xi(\x,l_1,\l)=\dsum_{0\leq \n\leq p^{r-v-r_1}-1}\dsum_{1\leq \m\leq p^{r_1}}e\left(\frac{P(\m)\widetilde{l}_1}{p^{r_1}}\right)e\left(\frac{\langle(\nabla f_{p^{r_1}\n+\m}(\x),f_{p^{r_1}\n+\m}(\x)),\l\rangle}{p^h}\right)    .
\end{aligned}
\end{equation*}
On substituting this expression into $S(\x)$ and from that into $(\ref{6.61})$, it follows by the triangle inequality that 
\begin{equation}\label{6.6262}
\begin{aligned}
    &p^{-(n+2)h-r}p^{kv}\dsum_{\x\in \mathcal{R}_n(p^h)}S(\x)\\
    &\leq p^{-(n+2)h-r}p^{kv}\dsum_{\substack{1\leq r_1\leq r-kv}}\dsum_{\x\in \mathcal{R}_n(p^h)}\dsum_{\substack{1\leq l_1\leq p^{r-kv}\\(l_1,p^{r-kv})=p^{r-kv-r_1}}}|S_1(\widetilde{l}_1,\x)|\cdot\left|S_2(\x)\right|,
\end{aligned}
\end{equation}
where 
$$S_1(\widetilde{l}_1,\x)=\sup_{1\leq \l\leq p^h}\biggl|\dsum_{1\leq \m\leq p^{r_1}}e\left(\frac{P(\m)\widetilde{l}_1}{p^{r_1}}\right)e\left(\frac{\langle(\nabla f_{\m}(\x),f_{\m}(\x)),\l\rangle}{p^h}\right)\biggr|$$
and
$$S_2(\x)=\dsum_{1\leq \l\leq p^h}\biggl|\dsum_{0\leq \n\leq p^{r-v-r_1}-1}e\left(\frac{\langle(\nabla f_{\n}(\x),f_{\n}(\x)),\l\rangle}{p^{h-r_1}}\right) \biggr|.$$

First, by the same explanation leading to $(\ref{5.55.5})$, we have
\begin{equation}\label{6.6363}
S_1(\widetilde{l}_1,\x)\ll p^{Nr_1-Nr_1/(2^{k-1}(k-1))+\epsilon}.
\end{equation}
Next, we turn to estimate the sum $S_2(\x).$ We first consider the case $h-r_1>0.$ Since $r-v\geq h,$ we infer from orthogonality that the inner sum in $S_2(\x)$ is a non-negative integer. Thus, by changing the order of summations, we see that
$$S_2(\x)=\dsum_{0\leq \n\leq p^{r-v-r_1}-1}\dsum_{1\leq \l\leq p^h}e\left(\frac{\langle(\nabla f_{\n}(\x),f_{\n}(\x)),\l\rangle}{p^{h-r_1}}\right).$$
Then, we find by orthogonality that
\begin{equation*}
\begin{aligned}
    S_2(\x)&=p^{h(n+1)}\cdot \#\left\{0\leq \n\leq p^{r-v-r_1}-1\middle|\ \begin{aligned}
         (\nabla f_{\n}(\x),f_{\n}(\x))\equiv \0\ \text{mod}\ p^{h-r_1}
    \end{aligned}\right\}  \\
    &=p^{h(n+1)+(r-v-h)N}\cdot \#\left\{0\leq \n\leq p^{h-r_1}-1\middle|\ \begin{aligned}
         (\nabla f_{\n}(\x),f_{\n}(\x))\equiv \0\ \text{mod}\ p^{h-r_1}
    \end{aligned}\right\} .
\end{aligned}
\end{equation*}
By splitting this expression for $S_2(\x)$ in terms of values of $(\n,p^{h-r_1})$, we see that
\begin{equation}\label{6.6565}
    S_2(\x)=p^{h(n+1)+(r-v-h)N}\cdot \dsum_{0\leq g\leq h-r_1}S^g(\x),
\end{equation}
where
$$S^g(\x)=\#\left\{\n\in \mathcal{R}_{N}(p^{h-r_1-g})\middle|\ \begin{aligned}
        ( \nabla f_{\n}(\x),f_{\n}(\x))\equiv \0\ \text{mod}\ p^{h-r_1-g}
    \end{aligned}\right\}.$$
    By applying Lemma $\ref{lem3.5}$ with $h-r_1-g$ in place of $r$ and $e$, one obtains 
    \begin{equation}\label{6.656565}
        S^g(\x)\ll p^{(h-r_1-g)(N-n)}.
    \end{equation}
Hence, on substituting ($\ref{6.656565}$) into $(\ref{6.6565})$, it follows that when $h-r_1>0$, one has
\begin{equation}\label{6.6767}
    S_2(\x)\ll p^{h(n+1)+(r-v-h)N}\cdot p^{(h-r_1)(N-n)}\ll p^{(r-v-r_1)N+r_1n+h}.
\end{equation}
Next, we consider the case $h-r_1\leq 0.$ Then, it follows by the trivial bound that
\begin{equation}\label{6.666}
S_2(\x)=p^{(n+1)h}\cdot p^{(r-v-r_1)N}.   
\end{equation}

Meanwhile, it follows from $(\ref{6.6262})$ that 
\begin{equation}\label{6.68}
\begin{aligned}
    &p^{-(n+2)h-r}\cdot p^{kv}\dsum_{\x\in \mathcal{R}_n(p^h)}S(\x) \ll X_1+X_2,
\end{aligned}
\end{equation}
where 
$$X_1= p^{-(n+2)h-r}\cdot p^{kv}\dsum_{\substack{1\leq r_1< h}}\dsum_{\x\in \mathcal{R}_n(p^h)}\dsum_{\substack{1\leq l_1\leq p^{r-kv}\\(l_1,p^{r-kv})=p^{r-kv-r_1}}}|S_1(\widetilde{l}_1,\x)|\cdot\left|S_2(\x)\right|$$
and
$$X_2= p^{-(n+2)h-r}\cdot p^{kv}\dsum_{\substack{h\leq r_1}}\dsum_{\x\in \mathcal{R}_n(p^h)}\dsum_{\substack{1\leq l_1\leq p^{r-kv}\\(l_1,p^{r-kv})=p^{r-kv-r_1}}}|S_1(\widetilde{l}_1,\x)|\cdot\left|S_2(\x)\right|$$

Hence, on substituting the bound $(\ref{6.6767})$ into $X_1$ and the quantity $(\ref{6.666})$ into $X_2$ together with the bound $(\ref{6.6363})$ for $S_1(\widetilde{l}_1,\x)$, it follows that whenever $N>(k-1)2^{k-1}(n+3)$, one has
\begin{equation*}
\begin{aligned}
X_1 &\ll p^{-(n+2)h-r}\cdot p^{kv}\cdot p^{hn} \cdot p^{(r-v)N+h}\dsum_{1\leq r_1<h}p^{r_1-Nr_1/(2^{k-1}(k-1))+r_1n+\epsilon}\\
    &\ll p^{-h-1}\cdot p^{kv-r}\cdot p^{(r-v)N},
\end{aligned}   
\end{equation*} 
and 
\begin{equation*}
\begin{aligned}
    X_2&\ll  p^{-(n+2)h-r}\cdot p^{kv}\cdot p^{hn} \cdot p^{(r-v)N+(n+1)h}\dsum_{h\leq r_1}p^{r_1-Nr_1/(2^{k-1}(k-1))+\epsilon}\\
    &\ll  p^{-h-1}\cdot p^{kv-r}\cdot p^{(r-v)N}.
\end{aligned}
\end{equation*}
Thus, it follows from $(\ref{6.68})$ together with these estimates for $X_1$ and $X_2$ that one has
\begin{equation}
    p^{-(n+2)h-r}\cdot p^{kv}\dsum_{\x\in \mathcal{R}_n(p^h)}S(\x) \ll p^{-h-1}\cdot p^{kv-r}\cdot p^{(r-v)N}.
\end{equation}


As the endgame, we turn to estimate the remaining quantity in ($\ref{6.61}$), that is
\begin{equation}\label{6.8282}
    p^{-(n+2)h-r}p^{kv}\dsum_{\x\in \mathcal{R}_n(p^h)}T(\x).
\end{equation}
It follows from orthogonality that one has
\begin{equation*}
\begin{aligned}
T(\x)&=p^{h(n+1)}\cdot\#\left\{1\leq \b\leq p^{r-v}\middle|\ \begin{aligned}
            (\nabla f_{\b}(\x),f_{\b}(\x))\equiv \boldsymbol{0}\ \text{mod}\ p^{h}
        \end{aligned}\right\}\\
        &=p^{h(n+1)+(r-v-h)N}\cdot\#\left\{1\leq \b\leq p^{h}\middle|\ \begin{aligned}
            (\nabla f_{\b}(\x),f_{\b}(\x))\equiv \boldsymbol{0}\ \text{mod}\ p^{h}
        \end{aligned}\right\}.
\end{aligned}
\end{equation*}
By splitting this expression for $T(\x)$ in terms of values of $(\b,p^{h})$, we see that
\begin{equation}\label{6.82}
    T(\x)=p^{h(n+1)+(r-v-h)N}\cdot \dsum_{0\leq g\leq h}T^g(\x),
\end{equation}
where
$$T^g(\x)=\#\left\{\b\in \mathcal{R}_{N}(p^{h-g})\middle|\ \begin{aligned}
           (\nabla f_{\b}(\x),f_{\b}(\x))\equiv \boldsymbol{0}\ \text{mod}\ p^{h-g}
        \end{aligned}\right\}.$$
By applying Lemma $\ref{lem3.5}$ with $h-g$ in place of $r$ and $e$, one obtains 
\begin{equation}\label{6.83}
    T^g(\x)\ll p^{(h-g)(N-n)}.
\end{equation}
Therefore, on substituting $(\ref{6.83})$ into $(\ref{6.82})$ and that into ($\ref{6.8282}$), one concludes that
\begin{equation}\label{6.86}
\begin{aligned}
     p^{-(n+2)h-r}p^{kv}\dsum_{\x\in \mathcal{R}_n(p^h)}T(\x)&\ll   p^{-(n+2)h-r}\cdot p^{kv}\cdot p^{hn}\cdot p^{h(n+1)+(r-v-h)N}\cdot  p^{h(N-n)} \\
     &\ll p^{-h}\cdot p^{kv-r}\cdot p^{(r-v)N}.
\end{aligned}
\end{equation}
Thus, on substituting $(\ref{6.68})$ and $(\ref{6.86})$ into $(\ref{6.61})$, we conclude that
\begin{equation*}
\begin{aligned}
M_1(r,v,h;k)&\ll p^{-h}\cdot p^{kv-r}\cdot p^{(r-v)N}.
\end{aligned}
\end{equation*}
This completes the proof of Lemma $\ref{lem6.96.9}.$
\end{proof} 

\bigskip

\begin{proof}[Proof of Lemma $\ref{lem6.8}$]
 We recall the definition of the condition $\mathcal{C}_v^{(e)}(p^r)$ with $0\leq e\leq r-v$ in the preamble to the definition  $(\ref{6.59590})$ of $A_v^{(e)}(p^r)$, that is
\begin{equation*}
\begin{aligned}
A_v^{(e)}(p^r)=\left\{\a\in [1,p^r]^N\middle|\ \a \ \text{satisfies the condition}\ \mathcal{C}_v^{(e)}(p^r)\right\}.    
\end{aligned}
\end{equation*}
For $\a\in [1,p^r]^N$ with $p^v|\a$, we define
\begin{equation}\label{6.888}
\mathcal{D}_{p^{-v}\a}(p^h)=\left\{\x\in \mathcal{R}_n(p^{h})\middle|\ \begin{aligned}
(\nabla f_{p^{-v}\a}(\x),f_{p^{-v}\a}(\x))\equiv \boldsymbol{0}\ \text{mod}\ p^{h}
\end{aligned}
\right\}.    
\end{equation}
Observe that given $\x\in \mathcal{D}_{p^{-v}\a}(p^{h})$, we see that for any $g\in \Z/p^{h}\Z$ with $(g,p)=1$ we have $g\x\in \mathcal{D}_{p^{-v}\a}(p^{h}).$

First, we consider the case $r-kv\leq 0$. By our assumptions concerning $e$ and $v$ in the definition of $A^{(e)}_v(p^r)$ we have $r-v\geq e$. Thus, one finds that
\begin{equation}\label{6.7878}
    \begin{aligned}
        \# A_v^{(e)}(p^r)&\leq \#\left\{\a\in [1,p^r]^N\middle|\ \begin{aligned}
    &\mathcal{D}_{p^{-v}\a}(p^{e})\neq \emptyset\\
    &P(\a)\equiv 0\ \text{mod}\ p^r\\
    &p^v\|\a
\end{aligned}\right\}.
    \end{aligned}
\end{equation}
Hence, we deduce from $(\ref{6.7878})$ and the observation following the definition $(\ref{6.888})$ of $\mathcal{D}_{p^{-v}\a}(p^h)$ that
\begin{equation}\label{6.899}
\begin{aligned}
     \# A_v^{(e)}(p^r)&\leq \dsum_{\substack{\a\in [1,p^r]^N\\ P(\a)\equiv 0\ \text{mod}\ p^r\\ p^v\|\a}}\frac{1}{p^{e}(1-p^{-1})}\cdot \#(\mathcal{D}_{p^{-v}\a}(p^{e})).
\end{aligned}
\end{equation}
Recall that $P$ is a homogeneous polynomial of degree $k$ and that $r-kv\leq 0$. Then, from the condition $p^v\|\a$, we may drop the condition $P(\a)\equiv 0\ \text{mod}\ p^r$ in ($\ref{6.899}$). Hence, we see from $(\ref{6.899})$ together with the substitution $\b=p^{-v}\a$ that
\begin{equation*}
     \# A_v^{(e)}(p^r)  \ll \dsum_{\substack{\b\in [1,p^{r-v}]^N}}p^{-e}\cdot \#\left\{\x\in \mathcal{R}_n(p^{e})\middle| \begin{aligned}
(\nabla f_{\b}(\x),f_{\b}(\x))\equiv \boldsymbol{0}\ \text{mod}\ p^{e}
\end{aligned}
\right\}.
\end{equation*}
By changing the order of summations, we have
\begin{equation}
     \# A_v^{(e)}(p^r)\ll \dsum_{\substack{\x\in \mathcal{R}_n(p^{e})}}p^{-e}\cdot \#\left\{\b\in [1,p^{r-v}]^N\middle| \begin{aligned}
(\nabla f_{\b}(\x),f_{\b}(\x))\equiv \boldsymbol{0}\ \text{mod}\ p^{e}
\end{aligned}
\right\}.
\end{equation}
Then, by the same argument leading from $(\ref{6.82})$ to $(\ref{6.86})$, it readily follows by Lemma $\ref{lem3.5}$ with $r-v$ in place of $r$ that
\begin{equation}\label{6.777}
     \# A_v^{(e)}(p^r)  \ll \dsum_{\substack{\x\in \mathcal{R}_n(p^{e})}}p^{-e}\cdot p^{(r-v)(N-n)}\cdot p^{(r-v-e)n}\ll p^{(r-v)N-e}\ll p^{rN-r-e},
\end{equation}
where we have used the inequalities $\#\mathcal{R}_n(p^e)\leq p^{en}$ and $vN\geq r$ stemming from $v\geq r/k.$


Next, we consider the case $r-kv>0$. We obtain again $(\ref{6.7878})$ and $(\ref{6.899})$. Hence, we find from $(\ref{6.899})$ together with a substitution $\b=p^{-v}\a$ that
\begin{equation*}
   \# A_v^{(e)}(p^r)  \ll \dsum_{\substack{\b\in [1,p^{r-v}]^N\\P(\b)\equiv 0\ \text{mod}\ p^{r-kv}}}p^{-e}\cdot \#\left\{\x\in \mathcal{R}_n(p^{e})\middle| \begin{aligned}
(\nabla f_{\b}(\x),f_{\b}(\x))\equiv \boldsymbol{0}\ \text{mod}\ p^{e}
\end{aligned}
\right\}.
\end{equation*}
Therefore, by applying Lemma $\ref{lem6.96.9}$ with $e$ in place of $h$, we have
\begin{equation}\label{6.8383}
     \# A_v^{(e)}(p^r) \ll p^{-e}\cdot p^{kv-r}\cdot p^{(r-v)N}\ll p^{rN-r-e}.
\end{equation}

Combining both the bounds $(\ref{6.777})$ and $(\ref{6.8383})$, we conclude that
\begin{equation*}
    \# A_v^{(e)}(p^r) \ll p^{rN-r-e}.
\end{equation*}
This completes the proof of Lemma $\ref{lem6.8}.$
\end{proof}

\bigskip

\subsubsection{The proof of Lemma $\ref{lem6.9}$}\label{6.1.5}

\begin{lem}\label{lem6.9}
Let $d\geq 3$. Let $p$ be a prime number and $r\geq 1.$ For $\a\in A_{v}^{(e)}(p^r)$ with $e\in \{0,\ldots, r-v\}$, we have
$$\sigma(\a;p^r)\geq p^{v-(e+1)(n-1)}.$$
\end{lem}
\begin{proof}

When $\a\in A_v^{(e)}(p^r),$ one has
\begin{equation}\label{6.88}
    \begin{aligned}
        \sigma(\a;p^r)&=p^{-r(n-1)}\cdot\#\{\g\in [1,p^r]^n|\ \langle\a,\nu_{d,n}(\g)\rangle\equiv 0\ \text{mod}\ p^r\}\\
        &=p^{-r(n-1)}\cdot p^{vn}\cdot\#\{\g\in [1,p^{r-v}]^n|\ \langle p^{-v}\a,\nu_{d,n}(\g)\rangle\equiv 0\ \text{mod}\ p^{r-v}\}\\
        &\geq \frac{p^v}{p^{(r-v)n-(r-v)}}\cdot \#\{\g\in \mathcal{R}_n(p^{r-v})|\ \langle p^{-v}\a,\nu_{d,n}(\g)\rangle\equiv 0\ \text{mod}\ p^{r-v}\}.
    \end{aligned}
\end{equation}
Since $\a\in A_{v}^{(e)}(p^r),$ there exists $\x\in \mathcal{R}_n(p^{r-v})$ such that
\begin{equation}\label{6.87}
v_{p^{r-v}}(\nabla f_{p^{-v}\a}(\x))=e\ \text{and}\ f_{p^{-v}\a}(\x)\equiv 0\ \text{mod}\ p^{r-v},   
\end{equation} with $p^{-v}\a\in \mathcal{R}_N(p^{r-v}).$ Thus, it follows by Lemma $\ref{lem3.4}$ with $p^{-v}\a$ and $r-v$ in place of $\a$ and $r$ that we obtain the lower bound
\begin{equation*}
    \frac{p^v}{p^{(r-v)n-(r-v)}}\cdot \#\{\g\in \mathcal{R}_n(p^{r-v})|\ \langle p^{-v}\a,\nu_{d,n}(\g)\rangle\equiv 0\ \text{mod}\ p^{r-v}\}\geq \frac{p^v}{p^{(e+1)(n-1)}}.
\end{equation*}
Hence, we conclude that
$$\sigma(\a;p^r)\geq p^{v-(e+1)(n-1)}.$$
\end{proof}

\bigskip

\begin{proof}[Proof of Proposition $\ref{prop6.1}$]
Note that a modicum of computation reveals that the condition on $N=N_{d,n}\geq 1000n^28^k$ satisfies conditions on $N$ in Lemma $\ref{lem5.1}$, Lemma $\ref{lem5.2}$, Lemma $\ref{lem5.3}$, Lemma $\ref{lem6.8}$, and Theorem $\ref{thm3.1}$, and thus these lemmas are available in this proof.

For given $\b\in \Z^N,$ we temporarily define $\mathbb{B}(\b)$ to be the set consisting of all prime powers $p^r$ having the property that there exists $\x\in \mathcal{R}_n(p^r/(p^r,\b))$ such that $f_{\b/(p^r,\b)}(\x)\equiv 0$ mod $p^r/(p^r,\b).$ Then, we define  a set $\mathbb{F}^{loc}_{d,n}(Q)$ to be the set consisting of $\b\in [1,Q]^N$ such that for all $p^r\|Q$, one has $p^r\in \mathbb{B}(\b).$

For simplicity, we write 
\begin{equation*}
    \mathcal{S}(A)= A^{-N+k}\cdot\#\left\{\a\in \mathcal{A}^{\text{loc}}_{d,n}(A;P)\middle|\   \sigma(\a;W)\leq(\log A)^{-\eta}\right\}.
\end{equation*}
Note that $\sigma(\a;W)=\sigma(\b;W)$ for $\a\equiv \b\ \text{mod}\ W.$ Then, one infers that
\begin{equation}\label{6.79}
    \mathcal{S}(A)=A^{-N+k}\cdot \dsum_{\substack{\b\in \mathbb{F}_{d,n}^{\text{loc}}(W)\\\sigma(\b;W)\leq(\log A)^{-\eta}}}\#\left\{\a\in \mathcal{A}_{d,n}^{\text{loc}}(A;P)\middle|\ \a\equiv\b\ \text{mod}\ W\right\}.
\end{equation}
Since the relations $\a\equiv \b\ \text{mod}\ W$ and $P(\a)=0$ together imply that $P(\b)\equiv 0\ \text{mod}\ W,$ we observe that the summands $\b$ in $(\ref{6.79})$ are additionally restricted to satisfy $P(\b)\equiv 0\ \text{mod}\ W$. With this observation in mind, we deduce by applying the Cauchy-Schwarz inequality that
\begin{equation}\label{6.816.81}
    \mathcal{S}(A)\leq A^{-N+k}\cdot \mathcal{S}_1^{1/2}\cdot \mathcal{S}_2^{1/2}, 
\end{equation}
where
\begin{equation*}
    \mathcal{S}_1=
    \#\left\{\b\in\mathbb{F}^{\text{loc}}_{d,n}(W)\middle|\ \begin{aligned}
        \sigma(\b;W)\leq(\log A)^{-\eta},\ P(\b)\equiv 0\ \text{mod}\ W
    \end{aligned}\right\}
\end{equation*}
and 
\begin{equation*}
    \mathcal{S}_2=\dsum_{\b\in [1,W]^N}\#\left\{\a\in \mathcal{A}_{d,n}^{\text{loc}}(A;P)\middle|\ 
        \a\equiv\b\ \text{mod}\ W
      \right\}^2.
\end{equation*}

Note that
\begin{equation*}
    \mathcal{S}_2\leq\{\a_1,\a_2\in [-A,A]^N|\ P(\a_1)=P(\a_2)=0,\ \a_1\equiv\a_2\ \text{mod}\ W\}.
\end{equation*}
Recall that $W\leq X^{2}$ and that we have the hypothesis $X^3\leq A$ in the statement of Proposition $\ref{prop6.1}$, and thus $W\leq A^{2/3}.$ Then, since $N\geq 18k(k-1)4^{k+3},$ on noting that the upper bound for $\mathcal{N}(A)$ in Theorem $\ref{thm3.1}$ is valid without the condition that $P(\x)=0$ has a nontrivial integer solution, we find by applying Theorem $\ref{thm3.1}$ that
\begin{equation}\label{6.826.826.82}
    \mathcal{S}_2\ll A^{2N-2k}W^{1-N}.
\end{equation}
We now analyse $\mathcal{S}_1$. Since $\sigma(\b;W)>0$, we deduce that for any $\kappa>0$ one has
\begin{equation*}
    \mathcal{S}_1\ll \dsum_{\substack{\b\in \mathbb{F}^{\text{loc}}_{d,n}(W)\\P(\b)\equiv 0\ \text{mod}\ W}}\left(\frac{1}{(\log A)^{\eta}\sigma(\b;W)}\right)^{\kappa}.
\end{equation*}
Thus, we find from ($\ref{6.26.26.2}$) that 
$$\mathcal{S}_1\ll \frac{1}{(\log A)^{\eta\cdot\kappa}} \dsum_{\substack{\b\in \mathbb{F}^{\text{loc}}_{d,n}(W)\\P(\b)\equiv 0\ \text{mod}\ W}}
\dprod_{p^r\|W}\frac{1}{\sigma(\b;p^r)^{\kappa}}.$$
Furthermore, an application of the Chinese remainder theorem delivers that
$$\mathcal{S}_1\ll \frac{1}{(\log A)^{\eta\cdot\kappa}} 
\dprod_{p^r\|W}\dsum_{\substack{\b\in \mathbb{F}^{\text{loc}}_{d,n}(p^r)\\P(\b)\equiv 0\ \text{mod}\ p^r}}\frac{1}{\sigma(\b;p^r)^{\kappa}}.$$
We investigate the sum over $\b\in \mathbb{F}^{\text{loc}}_{d,n}(p^r)$ with $P(\b)\equiv 0\ \text{mod}\ p^r$ by dividing this sum in terms of size of $\sigma(\b;p^r).$ Write
\begin{equation*}
    \mathcal{S}_{>}^{(\kappa)}(p^r)=\dsum_{\substack{\b\in \mathbb{F}^{\text{loc}}_{d,n}(p^r)\\P(\b)\equiv 0\ \text{mod}\ p^r\\\sigma(\b;p^r)>1/2}}\frac{1}{\sigma(\b;p^r)^{\kappa}}\ \text{and}\  \mathcal{S}_{\leq}^{(\kappa)}(p^r)=\dsum_{\substack{\b\in \mathbb{F}^{\text{loc}}_{d,n}(p^r)\\P(\b)\equiv 0\ \text{mod}\ p^r\\\sigma(\b;p^r)\leq1/2}}\frac{1}{\sigma(\b;p^r)^{\kappa}}, 
\end{equation*}
so that 
\begin{equation}\label{6.826.82}
    \mathcal{S}_1\ll \frac{1}{(\log A)^{\eta\cdot\kappa}} 
\dprod_{p^r\|W}\left(\mathcal{S}_{>}^{(\kappa)}(p^r)+ \mathcal{S}_{\leq}^{(\kappa)}(p^r)\right).
\end{equation}

We first analyse the sum $\mathcal{S}_{>}^{(\kappa)}(p^r).$ In order to do this, we use the estimate
$$\frac{1}{\sigma(\b;p^r)^{\kappa}}=1-\kappa(\sigma(\b;p^r)-1)+O((\sigma(\b;p^r)-1)^2),$$
where the implied constant depends at most on $\kappa.$ Then, we find that
\begin{equation}\label{6.88888}
    \begin{aligned}
   &\mathcal{S}_{>}^{(\kappa)}(p^r)\\
   &=\dsum_{\substack{\b\in \mathbb{F}^{\text{loc}}_{d,n}(p^r)\\P(\b)\equiv 0\ \text{mod}\ p^r\\\sigma(\b;p^r)>1/2}}1-\kappa \dsum_{\substack{\b\in \mathbb{F}^{\text{loc}}_{d,n}(p^r)\\P(\b)\equiv 0\ \text{mod}\ p^r\\\sigma(\b;p^r)>1/2}}(\sigma(\b;p^r)-1)+O\biggl(\dsum_{\substack{1\leq \b\leq p^r\\P(\b)\equiv 0\ \text{mod}\ p^r}}(\sigma(\b;p^r)-1)^2\biggr)\\
    &\leq \dsum_{\substack{1\leq \b\leq p^r\\P(\b)\equiv 0\ \text{mod}\ p^r}}1-\kappa \dsum_{\substack{\b\in \mathbb{F}^{\text{loc}}_{d,n}(p^r)\\P(\b)\equiv 0\ \text{mod}\ p^r}}(\sigma(\b;p^r)-1)+O\biggl(\dsum_{\substack{1\leq \b\leq p^r\\P(\b)\equiv 0\ \text{mod}\ p^r}}(\sigma(\b;p^r)-1)^2\biggr)\\
    &\leq \dsum_{\substack{1\leq \b\leq p^r\\P(\b)\equiv 0\ \text{mod}\ p^r}}1-\kappa \dsum_{\substack{1\leq \b\leq p^r\\P(\b)\equiv 0\ \text{mod}\ p^r}}(\sigma(\b;p^r)-1)+O\biggl(\dsum_{\substack{1\leq \b\leq p^r\\P(\b)\equiv 0\ \text{mod}\ p^r}}(\sigma(\b;p^r)-1)^2\biggr)+\kappa\cdot\mathcal{E},
    \end{aligned}
\end{equation}
where $$\mathcal{E}= \dsum_{\substack{1\leq \b\leq p^r\\P(\b)\equiv 0\ \text{mod}\ p^r\\\b\notin \mathbb{F}^{\text{loc}}_{d,n}(p^r)}}\sigma(\b;p^r).$$

By splitting the sum $\mathcal{E}$ in terms of values of $(\b,p^r),$ we see that
\begin{equation}\label{6.838383}
    \mathcal{E}\leq \dsum_{\substack{1\leq \b\leq p^r\\P(\b)\equiv 0\ \text{mod}\ p^r\\\b\notin \mathbb{F}^{\text{loc}}_{d,n}(p^r)\\(\b,p^r)=1}}\sigma(\b;p^r)+\dsum_{1\leq d_0\leq r}\dsum_{\substack{1\leq \b\leq p^r\\P(\b)\equiv 0\ \text{mod}\ p^r\\(\b,p^r)=p^{d_0}}}\sigma(\b;p^r).
\end{equation}

We analyze the second term in $(\ref{6.838383})$. Recall the definition ($\ref{def6.1}$) of $\sigma(\b;p^r)$. Also, when $(\b,p^r)=p^{d_0}$, we write $\widetilde{\b}=p^{-d_0}\b.$ Then, in the case $d_0\leq r/k,$ one has
\begin{equation*}
\begin{aligned}
    \sigma(\b;p^r)&=p^{-r(n-1)}\#\{1\leq \g\leq p^r|\ \langle \widetilde{\b},\nu_{d,n}(\g)\rangle\equiv 0\ \text{mod}\ p^{r-d_0}\}\\
    &\leq p^{-r(n-1)}\cdot p^{kd_0n} \#\{1\leq \g\leq p^{r-kd_0}|\ \langle \widetilde{\b},\nu_{d,n}(\g)\rangle\equiv 0\ \text{mod}\ p^{r-kd_0}\}\\
    &=p^{kd_0} \sigma(\widetilde{\b},p^{r-kd_0}).
\end{aligned}
\end{equation*}
Hence, we deduce that
\begin{equation*}
\begin{aligned}
    \dsum_{1\leq d_0\leq r/k}\dsum_{\substack{1\leq \b\leq p^r\\P(\b)\equiv 0\ \text{mod}\ p^r\\(\b,p^r)=p^{d_0}}}\sigma(\b;p^r)&\leq \dsum_{1\leq d_0\leq r/k}\dsum_{\substack{1\leq \widetilde{\b}\leq p^{r-d_0}\\P(\widetilde{\b})\equiv0\ \text{mod}\ p^{r-kd_0}}}p^{kd_0} \sigma(\widetilde{\b},p^{r-kd_0})\\
    &\leq \dsum_{1\leq d_0\leq r/k}p^{kd_0}\cdot p^{d_0(k-1)N}\dsum_{\substack{1\leq \widetilde{\b}\leq p^{r-kd_0}\\P(\widetilde{\b})\equiv0\ \text{mod}\ p^{r-kd_0}}} \sigma(\widetilde{\b},p^{r-kd_0}).
\end{aligned}
\end{equation*}
Then, by applying Lemma $\ref{lem5.2}$ with $r-kd_0$ in place of $r,$ we have
\begin{equation*}
\begin{aligned}
   \dsum_{1\leq d_0\leq r/k}\dsum_{\substack{1\leq \b\leq p^r\\P(\b)\equiv 0\ \text{mod}\ p^r\\(\b,p^r)=p^{d_0}}}\sigma(\b;p^r)\ll \dsum_{1\leq d_0\leq r/k}p^{kd_0+{d_0(k-1)N}}\cdot (p^{(r-kd_0)(N-1)}+O(E)),
\end{aligned}
\end{equation*}
where $$E=p^{(r-kd_0)N-\lceil (r-kd_0)/d\rceil n}+p^{(r-kd_0)N-(r-kd_0)-n+d}.$$
A modicum of computation delivers that
\begin{equation}\label{6.84}
\begin{aligned}
 &\dsum_{1\leq d_0\leq r/k}\dsum_{\substack{1\leq \b\leq p^r\\P(\b)\equiv 0\ \text{mod}\ p^r\\(\b,p^r)=p^{d_0}}}\sigma(\b;p^r)\\
   &\ll p^{r(N-1)+2k-N}+O(p^{rN-\lceil (r-k)/d\rceil n+k-N}+p^{rN-r-n+d-N+2k}).
\end{aligned}
\end{equation}
In the case $d_0> r/k,$ it follows by the trivial bound $\sigma(\b;p^r)\leq p^r$ that
\begin{equation}\label{6.85}
\begin{aligned}
    \dsum_{d_0> r/k}\dsum_{\substack{1\leq \b\leq p^r\\P(\b)\equiv 0\ \text{mod}\ p^r\\(\b,p^r)=p^{d_0}}}\sigma(\b;p^r)\leq  \dsum_{d_0> r/k}\dsum_{\substack{1\leq \b\leq p^r\\(\b,p^r)=p^{d_0}}}p^r\leq  \dsum_{d_0> r/k}p^{(r-d_0)N+r}\leq p^{(r-r/k)N+r}.
\end{aligned}
\end{equation}
Hence, it follows by $(\ref{6.84})$ and $(\ref{6.85})$ that 
\begin{equation}\label{6.}
\begin{aligned}
     \dsum_{1\leq d_0\leq r}\dsum_{\substack{1\leq \b\leq p^r\\P(\b)\equiv 0\ \text{mod}\ p^r\\(\b,p^r)=p^{d_0}}}\sigma(\b;p^r)\ll  p^{r(N-1)+2k-N}+p^{rN-\lceil (r-k)/d\rceil n+k-N}+p^{(r-r/k)N+r}.        
\end{aligned}
\end{equation}

 We turn to estimate the first term in $(\ref{6.838383}).$ Recall the definition of $\mathbb{F}^{\text{loc}}_{d,n}(p^r)$ at the beginning of this proof. Then, for $\b\notin \mathbb{F}^{\text{loc}}_{d,n}(p^r)$ and $(\b,p^r)=1$, one has
 \begin{equation}\label{6.8888}
 \begin{aligned}
       \sigma(\b;p^r)&=p^{-r(n-1)}\#\{1\leq \g\leq p^r|\ \langle \b,\nu_{d,n}(\g)\rangle\equiv 0\ \text{mod}\ p^{r}\}\\   
       &=p^{-r(n-1)}\#\{1\leq \g\leq p^r|\ p|\g,\ \langle \b,\nu_{d,n}(\g)\rangle\equiv 0\ \text{mod}\ p^{r}\}.
 \end{aligned}
 \end{equation}
 Hence, for $\b\notin \mathbb{F}^{\text{loc}}_{d,n}(p^r)$ and $(\b,p^r)=1$, it follows by the argument leading from $(\ref{eqeq5.1})$ to $(\ref{eqeqeq5.3})$ that 
\begin{equation}\label{6.876.87}
    \dsum_{\substack{1\leq \b\leq p^r\\P(\b)\equiv 0\ \text{mod}\ p^r\\\b\notin \mathbb{F}^{\text{loc}}_{d,n}(p^r)\\(\b,p^r)=1}}\sigma(\b;p^r)\leq p^{-r(n-1)}\dsum_{1\leq r_2\leq r}\dsum_{\substack{1\leq \g\leq p^r\\ (\g,p^r)=p^{r_2}}}(S(\g)+T(\g)),
\end{equation}
where $S(\g)$ and $T(\g)$ are defined in the sequel to $(\ref{eqeqeq5.3})$. By applying the same argument leading from $(\ref{5.45.4})$ to $(\ref{6.96.96.9})$ together with the hypotheses on $N,n,$ and $d$, we readily deduce that
\begin{equation}\label{6.886.88}
    p^{-r(n-1)}\dsum_{1\leq r_2\leq r}\dsum_{\substack{1\leq \g\leq p^r\\ (\g,p^r)=p^{r_2}}}S(\g)\ll p^{rN-r-2}.
\end{equation}
 Moreover, by applying the argument leading from $(\ref{5.11})$ to $(\ref{5.15})$ with the range $1\leq r_2\leq r$ in place of $0\leq r_2\leq r$, one infers that
 \begin{equation}\label{6.896.89}
 \begin{aligned}
     p^{-r(n-1)}\dsum_{1\leq r_2\leq r}\dsum_{\substack{1\leq \g\leq p^r\\ (\g,p^r)=p^{r_2}}}T(\g)=p^{-r(n+1)}(U_1+U_2) \ll p^{rN-\lceil r/d\rceil n}+p^{r(N-1)+d-n},
 \end{aligned}
 \end{equation}
 where $U_1$ and $U_2$ are defined in the sequel to $(\ref{5.11})$  with the range $1\leq r_2\leq r$ in place of $0\leq r_2\leq r$ and we have used inequalities 
 $$p^{-r(n+1)}U_1\ll p^{rN-\lceil r/d\rceil n}\ \text{and}\ p^{-r(n+1)}U_2\ll p^{r(N-1)+d-n}.$$
By substituting $(\ref{6.886.88})$ and $(\ref{6.896.89})$ into $(\ref{6.876.87})$ and by the hypothesis $n>d+1$ in the statement of Proposition $\ref{prop6.1}$, we find that 
\begin{equation}\label{6.92}
     \dsum_{\substack{1\leq \b\leq p^r\\P(\b)\equiv 0\ \text{mod}\ p^r\\\b\notin \mathbb{F}^{\text{loc}}_{d,n}(p^r)\\(\b,p^r)=1}}\sigma(\b;p^r)\ll p^{rN-r-2}.
\end{equation}
Hence, on substituting ($\ref{6.}$) and $(\ref{6.92})$ into ($\ref{6.838383}$), we obtain
\begin{equation}
    \mathcal{E}\ll p^{rN-r-2}.
\end{equation}
Then, by Lemma $\ref{lem5.1}$, Lemma $\ref{lem5.2}$ and Lemma $\ref{lem5.3}$, it follows from ($\ref{6.88888}$) that
\begin{equation}\label{6.946.94}
    \mathcal{S}_{>}^{(\kappa)}(p^r)\leq p^{rN-r}(1+O(p^{-2})).
\end{equation}

Next, we analyse the sum $\mathcal{S}_{\leq}^{(\kappa)}(p^r).$ Recall the definition $(\ref{6.59590})$ of $A_v^{(e)}(p^r)$ for given $e\in \{0,\ldots,r\}.$ On observing that
\begin{equation*}
    \{\b\in \mathbb{F}^{\text{loc}}_{d,n}(p^r)|\ P(\b)\equiv0\ \text{mod}\ p^r\}=\bigcup_{e=0}^{r}\bigcup_{v=0}^{r-e}A_v^{(e)}(p^r),
\end{equation*}
it follows that
\begin{equation}\label{6.95}
    \mathcal{S}_{\leq}^{(\kappa)}(p^r)\leq \dsum_{e=0}^r\dsum_{v=0}^{r-e}\mathcal{T}^{(\kappa)}(e,v;p^r),
\end{equation}
where, for $e\in \{0,\ldots,r\}$ and $v\in \{0,\ldots, r-e\}$, we have set
$$\mathcal{T}^{(\kappa)}(e,v;p^r)=\dsum_{\substack{\b\in A_v^{(e)}(p^r)\\ \sigma(\b;p^r)\leq 1/2}}\frac{1}{\sigma(\b;p^r)^{\kappa}}.$$

We first consider the case where $e\in \{0,1\}$. By applying Lemma $\ref{lem6.9}$, we find that
$$\mathcal{T}^{(\kappa)}(e,v;p^r)\leq p^{-\kappa v}\cdot p^{\kappa(e+1)(n-1)}\dsum_{\substack{1\leq \b\leq p^r\\P(\b)\equiv 0\ \text{mod}\ p^r\\ \sigma(\b;p^r)\leq 1/2 }}1.$$
From the elementary inequality $1\leq 4(\sigma(\b;p^r)-1)^2$ whenever $\sigma(\b;p^r)\leq 1/2,$ we obtain
\begin{equation*}
    \mathcal{T}^{(\kappa)}(e,v;p^r)\leq 4 p^{-\kappa v}\cdot p^{\kappa(e+1)(n-1)}\dsum_{\substack{1\leq \b\leq p^r\\P(\b)\equiv 0\ \text{mod}\ p^r}}(\sigma(\b;p^r)-1)^2.
\end{equation*}
Thus, it follows from Lemma $\ref{lem5.3}$ that
\begin{equation*}
    \mathcal{T}^{(\kappa)}(e,v;p^r)\ll  p^{rN-r-\kappa v+\kappa(e+1)(n-1)}(p^{r-n\lceil r/d\rceil}+p^{d-n}).
\end{equation*}
Then, we deduce that
\begin{equation}\label{6.96}
\begin{aligned}
    \dsum_{v=0}^{r} \mathcal{T}^{(\kappa)}(0,v;p^r)+\dsum_{v=0}^{r-1}\mathcal{T}^{(\kappa)}(1,v;p^r)\ll  p^{rN-r+2\kappa(n-1)}(p^{r-n\lceil r/d\rceil}+p^{d-n}).    
\end{aligned}
\end{equation}

We now consider the case where $e\in \{2,\ldots,r\}.$ On dropping the condition $\sigma(\b;p^r)\leq 1/2$ and applying Lemma $\ref{lem6.9}$, we find that 
\begin{equation*}
    \mathcal{T}^{(\kappa)}(e,v;p^r)\leq   p^{-\kappa v}\cdot p^{\kappa(e+1)(n-1)}\# A_v^{(e)}(p^r).
\end{equation*}
By taking $\kappa=1/(10n)$, we deduce from Lemma $\ref{lem6.8}$ that
\begin{equation}\label{6.97}
    \dsum_{e=2}^{r}\dsum_{v=0}^{r-e}\mathcal{T}^{(\kappa)}(e,v;p^r)\ll   p^{rN-r-2+3\kappa(n-1)}.
\end{equation}
It follows from $(\ref{6.95})$ together with bounds ($\ref{6.96}$) and ($\ref{6.97}$) that 
\begin{equation*}
     \mathcal{S}_{\leq}^{(\kappa)}(p^r)\ll p^{rN-r}\left(\frac{1}{p^{n\lceil r/d\rceil-r-2\kappa (n-1)}}+\frac{1}{p^{n-d-2\kappa (n-1)}}+\frac{1}{p^{2-3\kappa(n-1)}}\right).
\end{equation*}
By the assumption $n>d+1$ and the choice $\kappa=1/(10n)$, we have
\begin{equation}\label{6.98}
     \mathcal{S}_{\leq}^{(\kappa)}(p^r)\ll p^{rN-r-3/2}.
\end{equation}

On substituting ($\ref{6.946.94}$) and $(\ref{6.98})$ into ($\ref{6.826.82}$), one finds that
\begin{equation}\label{6.100}
    \mathcal{S}_1\ll \frac{W^{N-1}}{(\log A)^{\eta\cdot\kappa}}\dprod_{p^r\|W}\left(1+O(p^{-3/2})\right).
\end{equation}
By substituting $(\ref{6.826.826.82})$ and $(\ref{6.100})$ into ($\ref{6.816.81}$), we conclude that
\begin{equation*}
    \mathcal{S}(A)\ll  \frac{1}{(\log A)^{(\eta\cdot\kappa)/2}}.
\end{equation*}
Thus, since $\kappa=1/(10n)$, this completes the proof of Proposition $\ref{prop6.1}$.
\end{proof}

\bigskip

\subsection{Singular integrals treatment}\label{sec6.2}
Our goal in this subsection is to prove Proposition $\ref{prop6.11}$ below. In advance of the statement of this proposition, we recall the definition $\mathcal{A}^{\text{loc}}_{d,n}(A;P)$ in section 1, and recall the definition ($\ref{defnJ*}$) of  $\mathfrak{J}_{\a}^*.$


\begin{prop}\label{prop6.11}
 Let $A$ and $X$ be positive number with $X^3\leq A.$ Suppose that $n$ and $d$ are natural numbers with $d\geq 2$ and $n\geq 4$. Suppose that $P\in \Z[\x]$ is a non-singular form in $N_{d,n}$ variables of degree $k\geq 2.$ Then, whenever $N\geq 1000n^28^k$, one has
\begin{equation*}
     {A^{-N+k}}\cdot\#\left\{\a\in \mathcal{A}^{\text{loc}}_{d,n}(A;P)\middle|\ 
        \mathfrak{J}_{\a}^*\leq X^{n-d}A^{-1}(\log A)^{-\eta}\right\}\ll (\log A)^{-\eta/(2n)},
\end{equation*}
for some $\eta>0$, where the implicit constant may depend on $n$ and $d.$
\end{prop}

\bigskip

 We begin this section by observing the following. Recall the definition of $w$ and $N$. For $\a\in \Z^N$ with $\|\a\|_{\infty}\leq A,$ we define
\begin{equation*}
    \tau(\a;b)=b\cdot\text{mes}\left(\left\{\boldsymbol{\gamma}\in [0,1]^n\middle|\ |f_{\a}(\boldsymbol{\gamma})|\leq \frac{\|\nu_{d,n}(\boldsymbol{\gamma})\|\cdot \|\a\|}{b}\right\}\right).
\end{equation*}
We claim 
 that there exists $C>0$ such that
    $$C\cdot\tau(\a;2w^{5}N)\leq \mathfrak{J}^*_{\a}\cdot AX^{-n+d}.$$
   Define a function   
\begin{equation*}
    \chi_{\zeta}(\xi):=\left\{\begin{aligned}
        &1,\ \text{when}\ \xi\in [-(1/2)\zeta,(1/2)\zeta]\\
        &0,\ \text{otherwise}.
    \end{aligned}\right.
\end{equation*}
  On recalling the definition of $\widehat{\mathfrak{w}}_{\zeta}(\xi),$   one sees that for $\xi\in \R$
$$\widehat{\mathfrak{w}}_{\zeta}(\xi)\geq \frac{1}{2}\cdot\chi_{\zeta}(\xi).$$ Then,
 we find from $(\ref{defnJ*})$ together with the above inequality that
\begin{equation*}
    \mathfrak{J}^*_{\a}\cdot AX^{-n+d}\geq \frac{1}{2}\cdot\dint_{[0,1]^n}\zeta^{-1}\chi_{\zeta}(A^{-1}f_{\a}(\boldsymbol{\gamma}))d\boldsymbol{\gamma}.
\end{equation*}
Furthermore, for given $\a\in\Z^N$ with $\|\a\|_{\infty}\leq A,$ define a function
\begin{equation*} 
    \boldsymbol{\chi}(\boldsymbol{\gamma}):=\left\{\begin{aligned}
        &1,\ \text{when}\ |f_{\a}(\boldsymbol{\gamma})|\leq \frac{\|\nu_{d,n}(\boldsymbol{\gamma})\|\cdot \|\a\|}{2w^5N}\ \text{and}\ \boldsymbol{\gamma}\in [0,1]^n\\
        &0,\ \text{otherwise}.
    \end{aligned}\right.
\end{equation*}
Observe that $\|\nu_{d,n}(\boldsymbol{\gamma})\|\leq N^{1/2}$ and $\|\a\|\leq AN^{1/2}$ for $\boldsymbol{\gamma}\in [0,1]^n$ and $\|\a\|_{\infty}\leq A$, one has
\begin{equation*}
   A^{-1} \cdot\frac{\|\nu_{d,n}(\boldsymbol{\gamma})\|\cdot \|\a\|}{2w^5N}\leq \frac{1}{2w^5}=(1/2)\zeta.
\end{equation*}
Hence, for a given $\boldsymbol{\gamma}\in [0,1]^n$ and $\a\in\Z^N$ with $\|\a\|_{\infty}\leq A,$ one sees that the inequality $$|f_{\a}(\boldsymbol{\gamma})|\leq \frac{\|\nu_{d,n}(\boldsymbol{\gamma})\|\cdot \|\a\|}{2w^5N}$$ implies that  $A^{-1}|f_{\a}(\boldsymbol{\gamma})|\leq (1/2)\zeta.$
This means that whenver $\boldsymbol{\chi}(\boldsymbol{\gamma})=1$, we have $\chi_{\zeta}(A^{-1}f_{\a}(\boldsymbol{\gamma}))=1.$ Thus,  one has
\begin{equation*}
    \chi_{\zeta}(A^{-1}f_{\a}(\boldsymbol{\gamma}))\geq \boldsymbol{\chi}(\boldsymbol{\gamma}),
\end{equation*}
for a given $\boldsymbol{\gamma}\in [0,1]^n$ and $\a\in\Z^N$ with $\|\a\|_{\infty}.$
Therefore, we discern that
\begin{equation*}
\begin{aligned}
     \mathfrak{J}^*_{\a}\cdot AX^{-n+d}&\geq \frac{1}{2}\cdot\dint_{[0,1]^n}\zeta^{-1}\chi_{\zeta}(A^{-1}f_{\a}(\boldsymbol{\gamma}))d\boldsymbol{\gamma}\\ 
     &\geq \frac{1}{2}\cdot\dint_{[0,1]^n}\zeta^{-1}\boldsymbol{\chi}(\boldsymbol{\gamma})d\boldsymbol{\gamma}\\
     &=\frac{1}{2}\cdot\frac{1}{2N}\dint_{[0,1]^n}(2w^5N)\cdot \boldsymbol{\chi}(\boldsymbol{\gamma})d\boldsymbol{\gamma}\\
    &=\frac{1}{4N}\cdot \tau(\a;2w^{5}N).
     \end{aligned}
\end{equation*}
Therefore, we have confirmed the claim with $C=1/(4N).$

In order to prove Proposition $\ref{prop6.11}$, therefore, we infer that it suffices to show that
\begin{equation}\label{6.989898}
     {A^{-N+k}}\cdot\#\left\{\a\in \mathcal{A}^{\text{loc}}_{d,n}(A;P)\middle|\ 
       C\cdot\tau(\a;2w^{5}N)\leq(\log A)^{-\eta}\right\}\ll (\log A)^{-\eta/(2n)},
\end{equation}
where the implicit constant may depend on $n$ and $d$. We will prove the inequality $(\ref{6.989898})$ at the end of this section, via the strategy used in [$\ref{ref3}$, section 5].

\bigskip

\subsubsection{The proof of Lemma $\ref{lem6.13}$}
In this section, we prove Lemma $\ref{lem6.13}$. In order to describe this lemma, we introduce some definitions. 

\bigskip

We define the set 
\begin{equation}\label{6.99}
    \mathbb{I}_{d,n}^{\text{loc}}=\{\a\in \R^N\setminus\{\0\}|\ \text{there exists}\ \x\in \mathbb{S}^{n-1}\ \text{such that}\ f_{\a}(\x)=0\},
\end{equation}
where we wrote 
$\mathbb{S}^{n-1}=\{\x\in \R^{n}|\ \|\x\|=1\}.$
On recalling the definition of $B_m(u)$ in the preamble to Lemma $\ref{lem3.6}$, we define 
$$\mathcal{N}(\a)=\{\y\in \R^N: \y-\a\in B_N(1)\}.$$
Furthermore, on recalling that $P$ is a non-singular form in $N_{d,n}$ variables of degree $k\geq 2,$ we set
$$\mathcal{U}_{d,n}(A)=\{\a\in \Z^N\cap [-A,A]^N \cap \mathbb{I}_{d,n}^{\text{loc}}|\ \mathcal{N}(\a)\not\subset \mathbb{I}_{d,n}^{\text{loc}},\ P(\a)=0\}.$$

 The following lemma shows that the cardinality of $\mathcal{U}_{d,n}(A)$ is small enough so that we can begin the proof of Proposition $\ref{prop6.11}$ by excluding the set $\mathcal{U}_{d,n}(A)$ from the set in the left-hand side of the inequality $(\ref{6.989898})$.
\begin{lem}\label{lem6.13}
    Let $d\geq 2$ and $n\geq 4.$ Whenever $N_{d,n}>2^{k+1}(k-1)(k+n)$, we have
    $$\#\mathcal{U}_{d,n}(A)\ll A^{N-k-1/2}.$$
\end{lem}
\begin{proof}
If $\mathcal{U}_{d,n}(A)$ is the empty set, there is nothing to prove.
Given $\a\in \mathcal{U}_{d,n}(A)$, let $\b\in \mathcal{N}(\a)\setminus \mathbb{I}_{d,n}^{\text{loc}}$ and define
$$M_{\a}=\text{max}\left\{t\in (0,1]|\ \a+t(\b-\a)\in \mathbb{I}_{d,n}^{\text{loc}}\right\}.$$
 This maximum value exists, since $ \mathbb{I}_{d,n}^{\text{loc}}$ is a closed set.
We also set $\c=\a+M_{\a}(\b-\a)$ and we will show that for any $\x\in \mathbb{S}^{n-1}$ satisfying $f_{\c}(\x)=0$ one has $\nabla f_{\c}(\x)=0$. Indeed, for $\rho\in (0,1/A^2)$ and $\y\in B_{n}(\rho)$ we have
\begin{equation*}
    \begin{aligned}
        f_{\c+\rho^2(\b-\a)}(\x+\y)&=f_{\c+\rho^2(\b-\a)}(\x)+\langle \nabla f_{\c+\rho^2(\b-\a)}(\x),\y\rangle+O(A\rho^2)\\
        &=f_{\c}(\x)+\langle \nabla f_{\c}(\x),\y\rangle+O(\rho^{3/2})\\
        &=\langle \nabla f_{\c}(\x),\y\rangle+O(\rho^{3/2}).
    \end{aligned}
\end{equation*}
Assume that $\nabla f_{\c}(\x)\neq \boldsymbol{0}$ and let $\y_0\in \mathbb{S}^{n-1}$ satisfying $\langle \nabla f_{\c}(\x),\y_0\rangle\neq 0.$ For $|u|\leq \rho$ we thus have $$f_{\c+\rho^2(\b-\a)}(\x+u\y_0)=u\cdot \langle \nabla f_{\c}(\x),\y_0\rangle +O(\rho^{3/2}).$$
We see that if $\rho$ is chosen sufficiently small then the intermediate value theorem shows that there exists $u_0\in \R$ such that $f_{\c+\rho^2(\b-\a)}(\x+u_0\y_0)=0$, which contradicts the maximality of $M_{\a}.$ Hence, we find that
    \begin{equation*}
    \begin{aligned}
\mathcal{U}_{d,n}(A)\subseteq \{\a\in \Z^N\cap [-A,A]^N\cap \mathbb{I}_{d,n}^{\text{loc}}|\ P(\a)=0,\ \exists \c\in \mathcal{N}(\a)\ \exists \x\in \mathbb{S}^{n-1}, \nabla f_{\c}(\x)=0\}.        
    \end{aligned}        
    \end{equation*}

    Next, we will show that given $\a\in \Z^N\cap [-A,A]^N,$ if $\c\in \mathcal{N}(\a)$ and $\x\in \mathbb{S}^{n-1}$ satisfy $\nabla f_{\c}(\x)=\boldsymbol{0}$ then for any $\y\in \R^{n}$ such that $\|\y-\x\|\leq A^{-1}$ we have $\|\nabla f_{\a}(\y)\|\ll 1.$ Indeed, the triangle inequality gives 
    \begin{equation*}
        \begin{aligned}
            \|\nabla f_{\a}(\y)\|&\leq \|\nabla f_{\a-\c}(\y)\|+\|\nabla f_{\c}(\y)-\nabla f_{\c}(\x)\|+\|\nabla f_{\c}(\x)\|\\
            &\ll \|\a-\c\|\cdot \|\y\|^{d-1}+\|\c\|\cdot \|\y-\x\|\cdot \text{max}\{\|\x\|,\|\y\|\}^{d-2}.
        \end{aligned}
    \end{equation*}
    We thus get $\|\nabla f_{\a}(\y)\|\ll 1$ as we claimed. Since we have in addition
    $$\text{mes}\bigl(\bigl\{\y\in \R^n\big|\ \|\y-\x\|\leq A^{-1}\bigr\}\bigr)\gg A^{-n},$$ on recalling the definition of $\mathcal{H}_n(A)$ in section $\ref{sec3.1}$, it follows that 
\begin{equation}\label{6.9999}
\begin{aligned}
 \#\mathcal{U}_{d,n}(A)&\ll A^n\dsum_{\substack{\a\in \Z^N\cap [-A,A]^N\\P(\a)=0}}\text{mes}\left(\left\{\y\in \R^n\middle|\ \begin{aligned}
     &1-1/A\leq \|\y\|\leq 1+1/A\\
     &\|\nabla f_{\a}(\y)\|\ll 1
 \end{aligned}\right\}\right)\\
 &\ll A^n\dint_{\mathcal{H}_n(A)}\#\{\a\in \Z^N\cap [-A,A]^N|\ \|\nabla f_{\a}(\y)\|\ll 1,\ P(\a)=0\}d\y
\end{aligned}
\end{equation}

We consider a function $\varphi\in C^{\infty}(\R)$ such that supp($\varphi$)$\subseteq [-1,1],\ \varphi>0$ and $\frac{d^m\varphi(y)}{dy^m}\ll 1,$ where the implicit constant may depend on $m.$ 
Furthermore, for fixed $\y\in \R^n$, if we define a function $\mathbf{1}_{\y}(\a)$ to be
$$\mathbf{1}_{\y}(\a):=\begin{cases} 1, &\text{when $\|\nabla f_{\a}(\y)\|\ll 1$}\\
0, &\text{otherwise}
\end{cases},$$
we observe that there exists $D>0$ such that
$\mathbf{1}_{\y}(\a)\ll \prod_{i=1}^n \varphi\bigl(\frac{(\nabla f_{\a}(\y))_i}{D}\bigr).$
Then, it follows from $(\ref{6.9999})$ that
\begin{equation*}
\begin{aligned}
 \#\mathcal{U}_{d,n}(A)\ll A^n\int_{\mathcal{H}_n(A)}\dsum_{\substack{\|\a\|_{\infty}\leq A\\P(\a)=0}}\mathbf{1}_{\y}(\a)d\y\ll A^n \int_{\mathcal{H}_{n}(A)}\dsum_{\substack{\|\a\|_{\infty}\leq A\\ P(\a)=0}}\prod_{i=1}^n \varphi\biggl(\frac{(\nabla f_{\a}(\y))_i}{D}\biggr)d\y.
\end{aligned}
\end{equation*}

On writing $\varphi(\xi)=\int_{\R}\widehat{\varphi}(\beta)e(\beta \xi)d\beta,$
it follows by orthogonality that
\begin{equation}\label{6.100100}
\begin{aligned}
    &\#\mathcal{U}_{d,n}(A)\\
    &\ll A^n\dint_{\mathcal{H}_n(A)}\dint_{\R^n}\int_0^1\dsum_{ \|\a\|_{\infty}\leq A}e(\alpha P(\a))\prod_{i=1}^n\widehat{\varphi}(\beta_i)e\biggl(\left\langle\boldsymbol{\beta},\frac{\nabla f_{\a}(\y)}{D}\right\rangle\biggr)d\alpha d\boldsymbol{\beta}d\y,   
\end{aligned}
\end{equation}
where $\boldsymbol{\beta}=(\beta_1,\beta_2,\ldots,\beta_n)$ and $d\boldsymbol{\beta}=d\beta_1 d\beta_2\cdots d\beta_n.$
We define
\begin{equation}\label{6.102}
\mathfrak{M}_{\delta}=\bigcup_{\substack{0\leq a\leq q\leq A^{\delta}\\(q,a)=1}}\mathfrak{M}_{\delta}(q,a),    
\end{equation}
where $\mathfrak{M}_{\delta}(q,a)=\{\a\in [0,1)|\ |\alpha-a/q|\leq q^{-1}A^{\delta-k}\}$, and define $\mathfrak{m}_{\delta}=[0,1)\setminus \mathfrak{M}_{\delta}$. Then, we deduce from $(\ref{6.100100})$ together with the triangle inequality that
\begin{equation}\label{6.101}
    \#\mathcal{U}_{d,n}(A)\ll A^n(\Sigma_1+\Sigma_2),
\end{equation}
where 
\begin{equation}
\begin{aligned}
 \Sigma_1&=\dint_{\mathcal{H}_n(A)}\dint_{\mathfrak{M}_{1/4}}\dsum_{ \|\a\|_{\infty}\leq A}\biggl|\int_{\R^n}\prod_{i=1}^n\widehat{\varphi}(\beta_i)e\biggl(\biggl\langle\boldsymbol{\beta},\frac{\nabla f_{\a}(\y)}{D}\biggr\rangle\biggr)d\boldsymbol{\beta}\biggr|d\alpha d\y  \\
 &=\dint_{\mathcal{H}_n(A)}\dint_{\mathfrak{M}_{1/4}}\dsum_{ \|\a\|_{\infty}\leq A}\prod_{i=1}^n\varphi\biggl(\frac{(\nabla f_{\a}(\y))_i}{D}\biggr)d\alpha d\y
\end{aligned}
\end{equation}
and
$$\Sigma_2=\int_{\mathcal{H}_n(A)}\int_{\R^n}\prod_{i=1}^n\widehat{\varphi}(\beta_i)\biggl|\int_{\mathfrak{m}_{1/4}}\dsum_{-A\leq \a\leq A}e(\alpha P(\a))e\biggl(\biggl\langle\boldsymbol{\beta},\frac{\nabla f_{\a}(\y)}{D}\biggr\rangle\biggr)d\alpha\biggr|d\boldsymbol{\beta}d\y.$$

First, we analyze $\Sigma_2.$ Since $P(\a)$ is a non-singular form, it follows by the Weyl type estimate for exponential sums over minor arcs $[\ref{ref20}, \text{Lemma}\ 3.6]$ that 
$$\sup_{\alpha\in \mathfrak{m}_{1/4}}\dsum_{-A\leq \a\leq A}e(\alpha P(\a))e\biggl(\biggl\langle\boldsymbol{\beta},\frac{\nabla f_{\a}(\y)}{D}\biggr\rangle\biggr)\ll A^{N-N/(2^{k+1}(k-1))+\epsilon}.$$
Since $N>2^{k+1}(k-1)(k+n)$, $\text{mes}(\mathcal{H}_n(A))\ll A^{-1}$ and 
$\int_{\R^n}\prod_{i=1}^n \widehat{\varphi}(\beta_i)d\boldsymbol{\beta}\ll 1,$ we find that
\begin{equation}\label{6.103}
\Sigma_2\ll A^{N-k-n-1}.
\end{equation}

Next, we turn to estimate $\Sigma_1.$ Note that there exists a positive number $D_1=O(1)$ such that
$$ \prod_{i=1}^n \varphi\biggl(\frac{(\nabla f_{\a}(\y))_i}{D}\biggr)\ll \mathbf{1}_{\y}\biggl(\frac{\a}{D_1}\biggr).$$
Hence, we have
\begin{equation}
    \begin{aligned}
        \Sigma_1&\ll \dint_{\mathcal{H}_n(A)}\dint_{\mathfrak{M}_{1/4}}\dsum_{ \|\a\|_{\infty}\leq A} \mathbf{1}_{\y}\biggl(\frac{\a}{D_1}\biggr) d\alpha d\y\\
        &\ll \dint_{\mathfrak{M}_{1/4}}\dint_{\mathcal{H}_n(A)}\#\{\a\in \Z^N\cap [-A,A]^N|\ \|\nabla f_{\a}(\y)\|\ll 1\}d\y d\alpha,
    \end{aligned}
\end{equation}
where we have used the fact that $D_1=O(1).$
Therefore, since $\text{mes}(\mathfrak{M}_{1/4})\leq A^{-k+1/2}$, we conclude by applying Lemma $\ref{lem3.6}$ that 
\begin{equation}\label{6.104}
    \Sigma_1\ll A^{N-k-n+1/2}.
\end{equation}

By substituting $(\ref{6.103})$ and $(\ref{6.104})$ into $(\ref{6.101})$, we complete the proof of Lemma $\ref{lem6.13}$.
\end{proof}

\bigskip

Recall that $P\in \Z[\x]$ is a non-singular form in $N_{d,n}$ variables of degree $k\geq 2.$ We define the set
\begin{equation*}
C_P(A)=\left\{\a\in B_N(N)\middle|\ \|\a-\b A^{-1}\|\leq A^{-1}\ \text{with}\ \b\in \Z^N\cap [-A,A]^N\ \text{and}\ P(\b)=0\right\}.
\end{equation*}
For $\lambda>0,$ we introduce the set 
\begin{equation}\label{6.108}
    B^{(\lambda)}_N=\left\{\a\in C_P(A)\middle|\ \begin{aligned}
\exists\x\in \mathbb{S}^{n-1}\ \text{such that}\  \left\{\begin{aligned}
        &(1)\ f_{\a}(\x)=0\\
        &(2)\ \lambda\|\a\|<\|\nabla f_{\a}(\x)\|\leq 2\lambda\|\a\|
    \end{aligned}\right.\end{aligned}\right\}.
\end{equation}
Furthermore, if we set
\begin{equation}\label{6.109}
    M_{d,n}=\max\left\{\|\nabla f_{\a}(\x)\||\ (\a,\x)\in \mathbb{S}^{N-1}\times \mathbb{S}^{n-1}\right\},
\end{equation}
then we observe that
\begin{equation}\label{6.110}
\mathbb{I}_{d,n}^{\text{loc}}\cap C_P(A)=\bigcup_{l=1}^{\infty}B_N^{(M_{d,n}/2^l)}.
\end{equation}

In the proof of Proposition $\ref{prop6.11},$ the following lemma plays a similar role with that of Lemma $\ref{lem6.8}$ in the proof of Proposition $\ref{prop6.1}.$
\begin{lem}\label{lem6.14}
    Let $d\geq 2$ and $n\geq4.$ For $\lambda\in (0,M_{d,n})$, whenever $$N_{d,n}> 6n(n+1)k(k-1)2^{k-1},$$ we have
$$\text{mes}\left(B_{N}^{(\lambda)}\right)\ll \lambda^{}A^{-k}.$$
\end{lem}
\begin{rmk}
   By modifying the argument in the proof of Lemma 5.14, one could prove that $\text{mes}\left(B_{N}^{(\lambda)}\right)\ll \lambda^{c}A^{-k}$ for any $c<2,$ provided that $N_{d,n}$ is sufficiently large in terms of $n$ and $k.$ However, we did not put our effort into optimizing the result, as the conclusion of Lemma 5.14 is sufficient to prove Proposition 5.12. 
\end{rmk}
\begin{proof}[Proof of Lemma 5.14]
    For $\a\in B_N(N)$ we define
    $$\mathcal{D}_{\a}(\lambda)=\left\{\x\in \mathbb{S}^{n-1}\middle|\ \begin{aligned}
        |f_{\a}(\x)|\leq \lambda^2,\ \|\nabla f_{\a}(\x)\|\leq 2\lambda N
    \end{aligned}\right\},$$
    and we observe that
    $$\text{mes}\left(B_N^{(\lambda)}\right)\leq \text{mes}(\{\a\in C_P(A)|\ \mathcal{D}_{\a}(\lambda)\neq \emptyset\}).$$
    Given $\x\in \mathcal{D}_{\a}(\lambda/2)$, it follows from the estimates
    $$f_{\a}(\y)=f_{\a}(\x)+\langle\nabla f_{\a}(\x),\y-\x\rangle+O(\|\y-\x\|^2),$$
    and $$\|\nabla f_{\a}(\y)\|=\|\nabla f_{\a}(\x)\|+O(\|\y-\x\|),$$
    that there exists an absolute constant $K>0$ such that if $\y\in \mathbb{S}^{n-1}$ and $\|\y-\x\|\leq K\lambda$ then $\y\in \mathcal{D}_{\a}(\lambda)$. Since we have 
    $\text{mes}(\{\y\in \mathbb{S}^{n-1}|\ \|\x-\y\|\leq K\lambda\})\gg \lambda^{n-1},$
    we see that 
    $$\text{mes}\left(B_N^{(\lambda)}\right)\ll \int_{C_P(A)}\frac{\text{mes}(\mathcal{D}_{\a}(\lambda))}{\lambda^{n-1}}d\a.$$
    Therefore, we find that
    \begin{equation}\label{6.111}
        \text{mes}\left(B_N^{(\lambda)}\right)\ll \lambda^{1-n}\int_{\mathbb{S}^{n-1}}\text{mes}\left(\left\{\a\in C_P(A)\middle|\ \begin{aligned}
            |f_{\a}(\x)|\leq \lambda^2,\ \|\nabla f_{\a}(\x)\|\leq 2\lambda N
        \end{aligned}\right\}\right)d\x.
    \end{equation}
    For simplicity, we write $$\mathfrak{K}(\x,\lambda)=\text{mes}\left(\left\{\a\in C_P(A)\middle|\ \begin{aligned}
            |f_{\a}(\x)|\leq \lambda^2,\ \|\nabla f_{\a}(\x)\|\leq 2\lambda N
        \end{aligned}\right\}\right).$$ We investigate $\mathfrak{K}(\x,\lambda)$ by splitting into three cases $$(a)\ \lambda\leq A^{-k}\ (b)\ A^{-k}< \lambda\leq A^{-1/(2n)}\ (c)\ A^{-1/(2n)}< \lambda.$$

    Case (a): Assume that $\lambda\leq A^{-k}.$ Since $C_P(A)\subseteq B_N(N)$ and $N=O(1),$ one infers from Lemma $\ref{lem3.7}$ with $2\lambda N$ in place of $2\lambda$ that 
    \begin{equation}\label{6.112}
        \begin{aligned}
          \text{mes}\left(B_N^{(\lambda)}\right)
        &\ll \lambda^{1-n}\int_{\mathbb{S}^{n-1}}\mathfrak{K}(\x,\lambda)d\x\ll \lambda^2\leq\lambda A^{-k}.
        \end{aligned}
    \end{equation}

    Case (b): Assume that $A^{-k}<\lambda\leq A^{-1/(2n)}.$ Let $\varphi\in C^{\infty}$ be such that $\text{supp}(\varphi)\subseteq [-5,5]$, $\varphi>0$ and $\frac{d^m\varphi(y)}{dy^m}\ll 1,$ where the implicit constant may depend on $m.$ Then, we find by the definition of $C_P(A)$ that
    \begin{equation*}
        \mathfrak{K}(\x,\lambda)\ll \int_{B_N(N)}\dsum_{\substack{\|\b\|_{\infty}\leq A\\ P(\b)=0}}\biggl(\dprod_{i=1}^N\varphi(Aa_i-b_i)\biggr)\varphi\biggl(\frac{f_{\a}(\x)}{\lambda^2}\biggr)\biggl(\dprod_{i=1}^n\varphi\left(\frac{\partial_{x_i}f_{\a}(\x)}{2\lambda N}\right)\biggr)d\a.
    \end{equation*}
    For simplicity, here and throughout this proof, we write
    $$w(\c,\x)=\varphi\biggl(\frac{f_{\c}(\x)}{\lambda^2}\biggr)\biggl(\dprod_{i=1}^n\varphi\left(\frac{\partial_{x_i}f_{\c}(\x)}{2\lambda N}\right)\biggr),$$
    with $\c\in \R^N.$
    By orthogonality, one deduces that
    \begin{equation*}
        \mathfrak{K}(\x,\lambda)\ll \int_{B_N(N)}w(\a,\x)\cdot\Sigma(\a) d\a,
    \end{equation*}
    where 
$$\Sigma(\a)=\int_0^1\dsum_{\substack{\|\b\|_{\infty}\leq A}}e(\alpha P(\b))\prod_{i=1}^N\varphi(Aa_i-b_i)d\alpha.$$
    Recalling the definition $(\ref{6.102})$ of $\mathfrak{M}_{\delta}$ and $\mathfrak{m}_{\delta}$, one finds that
    \begin{equation}\label{6.113}
       \mathfrak{K}(\x,\lambda)\ll \int_{B_N(N)} w(\a,\x)\cdot\left(\Sigma_1(\a)+\Sigma_2(\a)\right)d\a,
    \end{equation}
    where
$$\Sigma_1(\a)=\int_{\mathfrak{M}_{1/(6n)}}\dsum_{\substack{\|\b\|_{\infty}\leq A}}e(\alpha P(\b))\prod_{i=1}^N\varphi(Aa_i-b_i)d\alpha$$
and
$$\Sigma_2(\a)=\int_{\mathfrak{m}_{1/(6n)}}\dsum_{\substack{\|\b\|_{\infty}\leq A}}e(\alpha P(\b))\prod_{i=1}^N\varphi(Aa_i-b_i)d\alpha.$$

We first analyze the quantity
$$\int_{B_N(N)}w(\a,\x)\cdot\Sigma_1(\a)d\a$$
Since $\text{supp}(\varphi)\subseteq [-5,5]$, we observe that for fixed $\a$
$$\dsum_{\substack{\|\b\|_{\infty}\leq A}}e(\alpha P(\b))\prod_{i=1}^N\varphi(Aa_i-b_i)\ll 1.$$
Hence, on noting that $\text{mes}(\mathfrak{M}_{1/(6n)})\leq A^{1/(3n)-k},$ we obtain that
$$\Sigma_1\ll \text{mes}(\mathfrak{M}_{1/(6n)})\leq A^{1/(3n)-k}.$$
Then, on noting that 
\begin{equation*}
\begin{aligned}
    \int_{B_N(N)} w(\a,\x)d\a\ll \text{mes}\left(\left\{\a\in B_N(N)\middle|\ \begin{aligned}
        |f_{\a}(\x)|\leq 5\lambda^2,\ \|\nabla f_{\a}(\x)\|\leq 10n^{1/2}\lambda N
    \end{aligned}\right\}\right),
\end{aligned}
    \end{equation*}
it follows by applying Lemma $\ref{lem3.7}$ with $5\lambda^2$ and $10n^{1/2}\lambda N$ in place of $\lambda^2$ and $2\lambda$ that
\begin{equation}\label{6.114}
    \begin{aligned}
       & \int_{B_N(N)}w(\a,\x)\cdot\Sigma_1(\a)d\a\\
       & \ll A^{1/(3n)-k}\cdot \text{mes}\left(\left\{\a\in B_N(N)\middle|\ \begin{aligned}
        |f_{\a}(\x)|\leq 5\lambda^2,\ \|\nabla f_{\a}(\x)\|\leq 10n^{1/2}\lambda N
    \end{aligned}\right\}\right)\\
       &\ll A^{1/(3n)-k}\lambda^{n+1}\ll \lambda^n A^{-k},
    \end{aligned}
\end{equation}
where we have used the upper bound $\lambda\leq A^{-1/(2n)}.$ Here, one easily infers that Lemma $\ref{lem3.7}$ is applicable in $(\ref{6.114})$, since $10n^{1/2}N$ and $5$ are $O(1).$

We turn to estimate
\begin{equation}\label{6.115}
    \int_{B_N(N)} w(\a,\x)\cdot\Sigma_2(\a)d\a.
\end{equation}
By using the identity
$\varphi(y)=\int_{-\infty}^{\infty}\widehat{\varphi}(\xi)e(\xi y)d\xi,$
we find that $(\ref{6.115})$ is seen to be
\begin{equation}\label{6.116}
 \dint_{\R^{n+1}}\widehat{\varphi}(\xi_0)\prod_{i=1}^n\widehat{\varphi}(\xi_i)\int_{\mathfrak{m}_{1/(6n)}}\Xi_0(\alpha,\xi_0,\boldsymbol{\xi})d\alpha d\xi_0 d\boldsymbol{\xi},   
\end{equation}
where
$$\Xi_0(\alpha,\xi_0,\boldsymbol{\xi})=\dsum_{\|\b\|_{\infty}\leq A}e(\alpha P(\b))\int_{B_N(N)}\prod_{i=1}^N\varphi(Aa_i-b_i)e\biggl(\frac{f_{\a}(\x)\xi_0}{\lambda^2}\biggr)e\biggl(\frac{\nabla f_{\a}(\x)\cdot \boldsymbol{\xi}}{2\lambda N}\biggr)d\a.$$
By the change of variable $\a=\a'+\b A^{-1}$, we see that $\Xi_0(\alpha,\xi_0,\boldsymbol{\xi})$ is
\begin{equation}\label{6.117}
    \begin{aligned}
    & \dsum_{\|\b\|_{\infty}\leq A}e(\alpha P(\b)) \int_{B_N(N)-\b A^{-1}}\prod_{i=1}^N\varphi(Aa_i')e\left(\frac{f_{\a'+\b A^{-1}}(\x)\xi_0}{\lambda^2}\right)e\left(\frac{\nabla f_{\a'+\b A^{-1}}(\x)\cdot \boldsymbol{\xi}}{2\lambda N}\right)d\a'\\
    & =\dsum_{\|\b\|_{\infty}\leq A}e(\alpha P(\b)) \int_{B_N(N/2)}\prod_{i=1}^N\varphi(Aa_i')e\left(\frac{f_{\a'+\b A^{-1}}(\x)\xi_0}{\lambda^2}\right)e\left(\frac{\nabla f_{\a'+\b A^{-1}}(\x)\cdot \boldsymbol{\xi}}{2\lambda N}\right)d\a'\\
    &=S_1(\alpha,\xi_0,\boldsymbol{\xi})\cdot S_2(\xi_0,\boldsymbol{\xi}),
    \end{aligned}
\end{equation}
where $$S_1(\alpha,\xi_0,\boldsymbol{\xi})=\dsum_{\|\b\|_{\infty}\leq A}e(\alpha P(\b))e\left(\frac{f_{\b A^{-1}}(\x)\xi_0}{\lambda^2}\right)e\left(\frac{\nabla f_{\b A^{-1}}(\x)\cdot \boldsymbol{\xi}}{2\lambda N}\right)$$
and
$$S_2(\xi_0,\boldsymbol{\xi})=\int_{B_N(N/2)}\prod_{i=1}^N\varphi(Aa_i')e\left(\frac{f_{\a'}(\x)\xi_0}{\lambda^2}\right)e\left(\frac{\nabla f_{\a'}(\x)\cdot \boldsymbol{\xi}}{2\lambda N}\right)d\a', $$
 and for the first equality in $(\ref{6.117})$ we used the fact that
$$\text{supp}\biggl(\prod_{i=1}^N\varphi(Aa_i')\biggr)\subseteq [-5/A,5/A]^N$$ and  $$\text{supp}\biggl(\prod_{i=1}^N\varphi(Aa_i')\biggr)\subseteq B_N(N/2)\subseteq B_N(N)-\b A^{-1}$$ for all $\b$ with $\|\b \|_{\infty}\leq A.$
Meanwhile, since $P(\b)$ is a non-singular polynomial of degree $k$, by the Weyl type estimate over minor arcs  [$\ref{ref20}$, Lemma 3.6], we have
$$\int_{\mathfrak{m}_{1/(6n)}}S_1(\alpha,\xi_0,\boldsymbol{\xi})d\alpha\ll A^{N-N/(6n(k-1)2^{k-1})+\epsilon}.$$
Furthermore, since $S_2(\xi_0,\boldsymbol{\xi})\leq A^{-N}$ and $\int_{\R^{n+1}}\widehat{\varphi}(\xi_0)\prod_{i=1}^n\widehat{\varphi}(\xi_i)d\xi_0 d\boldsymbol{\xi}\ll 1$, one obtains
$$\dint_{\R^{n+1}}\widehat{\varphi}(\xi_0)\prod_{i=1}^n\widehat{\varphi}(\xi_i)S_2(\xi_0,\boldsymbol{\xi})d\xi_0 d\boldsymbol{\xi}\ll A^{-N}.$$
Hence, we find from expressions $(\ref{6.116})$ and $(\ref{6.117})$ that
\begin{equation*}
\begin{aligned}
&\int_{B_N(N)} w(\a,\x)\cdot\Sigma_2(\a)d\a  \\
& =\dint_{\R^{n+1}}\widehat{\varphi}(\xi_0)\prod_{i=1}^n\widehat{\varphi}(\xi_i)\int_{\mathfrak{m}_{1/(6n)}}S_1(\alpha,\xi_0,\boldsymbol{\xi})\cdot S_2(\xi_0,\boldsymbol{\xi})d\alpha d\xi_0 d\boldsymbol{\xi}\ll A^{-N/(6n(k-1)2^{k-1})+\epsilon}.
\end{aligned}  
\end{equation*}
Noting a choice of $N$ in the statement of Lemma $\ref{lem6.14}$ with $N>6n(n+1)k(k-1)2^{k-1}$ and the condition $A^{-k}<\lambda\leq A^{-1/(2n)},$ we find that
\begin{equation}\label{6.118}
    \int_{B_N(N)} w(\a,\x)\cdot\Sigma_2(\a)d\a \ll \lambda^nA^{-k}.
\end{equation}
On substituting $(\ref{6.114})$ and $(\ref{6.118})$ into ($\ref{6.113}$) and from that into $(\ref{6.111})$, we see that whenever $A^{-k}<\lambda\leq A^{-1/(2n)}$ one has \begin{equation}\label{6.119}
   \text{mes}\left(B_N^{(\lambda)}\right)\ll \lambda A^{-k}. 
\end{equation}

Case (c): Assume that $A^{-1/(2n)}<\lambda.$ Recall the definition of $\mathfrak{K}(\x,\lambda),$ that is
$$\mathfrak{K}(\x,\lambda)=\text{mes}\left(\left\{\a\in C_P(A)\middle|\ \begin{aligned}
            |f_{\a}(\x)|\leq \lambda^2,\ \|\nabla f_{\a}(\x)\|\leq 2\lambda N
        \end{aligned}\right\}\right).$$
Note that for given $\x\in \S^{n-1}$ and $\a,\b\in \Z^N$ with $\|\a-\b A^{-1}\|\leq A^{-1}$ for sufficiently large $A>0$, it follows by applying the Cauchy-Schwarz inequality that we have 
$$|f_{\a-\b A^{-1}}(\x)|\leq \|\a-\b A^{-1}\|\|\nu_{d,n}(\x)\|\leq \lambda^2,$$
where we used $A^{-1/(2n)}< \lambda$ for the second inequality.
Thus, whenever $|f_{\a}(\x)|\leq \lambda^2$, we deduce by the triangle inequality that
$$|f_{\b A^{-1}}(\x)|\leq |f_{\a}(\x)|+|f_{\a-\b A^{-1}}(\x)|\leq 2\lambda^2.$$
Similarly, for given $\x\in \S^{n-1}$ and $\a,\b\in \Z^N$ with $\|\a-\b A^{-1}\|\leq A^{-1}$ for sufficiently large $A>0$, whenever $\|\nabla f_{\a}(\x)\|\leq 2\lambda N$, we have 
$$\|\nabla f_{\b A^{-1}}(\x)\|\leq \|\nabla f_{\a}(\x)\|+\|\nabla f_{\a-\b A^{-1}}(\x)\|\leq 4\lambda N.$$
Hence, on recalling the definitions of $C_P(A)$ and $w(\c,\x)$, we see that 
\begin{equation*}
\begin{aligned}
    &\mathfrak{K}(\x,\lambda)\\
    &\ll \dsum_{\substack{\|\b\|_{\infty}\leq A\\ P(\b)=0}}\text{mes}\left(\left\{\a\in B_N(N)\middle|\ \begin{aligned}
      &(i)\  |f_{\a}(\x)|\leq \lambda^2,\|\nabla f_{\a}(\x)\|\leq 2\lambda N\\
      &(ii)\ \|\a-\b A^{-1}\|\leq A^{-1} 
    \end{aligned}\right\}\right)\\
    &\ll\dsum_{\substack{\|\b\|_{\infty}\leq A\\ P(\b)=0}}\text{mes}\left(\left\{\a\in B_N(N)\middle|\ \begin{aligned}
      &(i)\  |f_{\b A^{-1}}(\x)|\leq 2\lambda^2,\|\nabla f_{\b A^{-1}}(\x)\|\leq 4\lambda N\\
      &(ii)\ \|\a-\b A^{-1}\|\leq A^{-1} 
    \end{aligned}\right\}\right)\\
    &\ll  A^{-N}\dsum_{\substack{\|\b\|_{\infty}\leq A\\ P(\b)=0}}w(\b A^{-1},\x).
\end{aligned}
\end{equation*}

By orthogonality and major and minor arcs dissection again, one sees that
\begin{equation}\label{6.120}
\begin{aligned}
    \mathfrak{K}(\x,\lambda)\ll A^{-N}\int_0^1 \dsum_{\substack{\|\b\|_{\infty}\leq A}}e(\alpha P(\b))w(\b A^{-1},\x)d\alpha =A^{-N}(\Xi_1+\Xi_2),
\end{aligned}
\end{equation}
where 
$$\Xi_1=\int_{\mathfrak{M}_{1/(6n)}}\dsum_{\|\b\|_{\infty}\leq A}e(\alpha P(\b))w(\b A^{-1},\x)d\alpha$$
and 
$$\Xi_2=\int_{\mathfrak{m}_{1/(6n)}}\dsum_{\|\b\|_{\infty}\leq A}e(\alpha P(\b))w(\b A^{-1},\x)d\alpha.$$
By using the Fourier transform of $\varphi$, we find that
\begin{equation}\label{6.121}
    \Xi_2=\int_{\R^{n+1}}\widehat{\varphi}(\xi_0)\left(\prod_{i=1}^n\widehat{\varphi}(\xi_i)\right)T(\xi_0,\boldsymbol{\xi})d\xi_0 d\boldsymbol{\xi},
\end{equation}
where
$$T(\xi_0,\boldsymbol{\xi})=\int_{\mathfrak{m}_{1/(6n)}}\dsum_{\|\b\|\leq A}e(\alpha P(\b))e\left(\frac{f_{\b}(\x)\xi_0}{A\lambda^2}\right)e\left(\frac{\nabla f_{\b}(\x)\cdot \boldsymbol{\xi}}{2A\lambda N}\right)d\alpha.$$
Hence, since $\int_{\R^{n+1}}\widehat{\varphi}(\xi_0)\prod_{i=1}^n\widehat{\varphi}(\xi_i)d\xi_0 d\boldsymbol{\xi}\ll 1$ and $A^{-1/(2n)}<\lambda$, it follows by the Weyl type estimate over minor arcs for the exponential sum of the integrand in
$T(\xi_0,\boldsymbol{\xi})$ [$\ref{ref20}$, Lemma 3.6] and our choice of $N$ that
\begin{equation}\label{6.122}
    \Xi_2\ll A^{N-k}\lambda^n.
\end{equation}

We turn to estimate $\Xi_1.$
On recalling the definition of $\mathfrak{M}_{1/(6n)},$ we find that
\begin{equation}\label{6.123}
    \Xi_1=\dsum_{1\leq q\leq A^{1/(6n)}}\dsum_{\substack{1\leq a\leq q\\(q,a)=1}}\dsum_{\substack{1\leq \z\leq q\\\z\in \Z^N}}e\left(\frac{a}{q}P(\z)\right)\int_{|\beta|\leq \frac{A^{1/(6n)}}{qA^k}}\dsum_{\substack{\y\in B_{\z}\\ \y\in \Z^N}}h_{(\beta)}(\y)d\beta,
\end{equation}
where $h_{(\beta)}(\y)=w\left(\frac{q\y+\z}{A},\x\right)e(\beta P(q\y+\z))$, and $B_{\z}$ is an $N$-dimensional box such that $\y\in B_{\z}$ implies $q\y+\z\in [-A,A]^N.$ If $\x\in [0,1]^N$, then $$h_{(\beta)}(\boldsymbol{\gamma}+\x)=h_{(\beta)}(\boldsymbol{\gamma})+O\biggl(\max_{\boldsymbol{u}\in [0,1]^N}|\nabla h_{(\beta)}(\boldsymbol{\gamma}+\boldsymbol{u})|\biggr).$$
Hence, on noting that $\text{mes}(\text{supp}(B_{\z}(\cdot)))\ll (A/q)^N,$ we see that 
\begin{equation}\label{6.124}
 \biggl|\int_{B_{\z}}h_{(\beta)}(\boldsymbol{\gamma})d\boldsymbol{\gamma}-\dsum_{\substack{\y\in B_{\z}\\ \y\in \Z^N}}h_{(\beta)}(\y)\biggr|\ll (A/q)^N\max_{\y\in B_{\z}}|\nabla h_{(\beta)}(\y)|+(A/q)^{N-1},   
\end{equation}
 where the second term on the right-hand side accounts for the initial and final interval with length at most $O(1)$ for each coordinate of $\boldsymbol{\gamma}$. Meanwhile, since $\frac{d\varphi(y)}{dy}\ll 1$, we note from the chain rule that
 $$\sup_{\substack{\x\in \mathbb{S}^{n-1}\\\b\in \R^N}}\|\nabla w(\b,\x)\|\ll \lambda^{-2}+\lambda^{-1}\ll \lambda^{-2},$$
 where we wrote $\nabla=(\partial_{b_1},\ldots,\partial_{b_N}).$
 Thus, we deduce from the product rule and the chain rule that the right-hand side in $(\ref{6.124})$ is bounded above by 
 \begin{equation}
 \begin{aligned}
 &(A/q)^N\cdot (q/A\lambda^{-2}+q|\beta|A^{k-1})+(A/q)^{N-1}   \\
 &\ll (A/q)^{N-1}\cdot \lambda^{-2}+(A/q)^N\cdot q|\beta|A^{k-1},
 \end{aligned} 
 \end{equation}
 where we used $\lambda\leq M_{d,n}=O(1).$
 Hence, we obtain from $(\ref{6.124})$ and the definition of $h_{(\beta)}(\cdot)$ that
 \begin{equation}\label{6.126}
     \dsum_{\substack{\y\in B_{\z}\\ \y\in \Z^N}}h_{(\beta)}(\y)=(A/q)^N\int_{[-1,1]^N}w(\boldsymbol{\gamma},\x)e(\beta A^k P(\boldsymbol{\gamma}))d\boldsymbol{\gamma}+O\bigl(\ (A/q)^{N-1}(\lambda^{-2}+|\beta|A^{k})\bigr).
 \end{equation}
 On substituting $(\ref{6.126})$ into $(\ref{6.123})$, it follows that $\Xi_1$ is seen to be
 \begin{equation}\label{6.127}
  \dsum_{1\leq q\leq A^{1/(6n)}}\dsum_{\substack{1\leq a\leq q\\(q,a)=1)}}\dsum_{\substack{1\leq \z\leq q\\\z\in \Z^N}}(A/q)^Ne\biggl(\frac{a}{q}P(\z)\biggr)\int_{|\beta|\leq \frac{A^{1/(6n)}}{qA^k}}I(A^k\beta)d\beta+O(E),   
 \end{equation}
 where
 $$I(\beta)=\int_{[-1,1]^N}w(\boldsymbol{\gamma},\x)e(\beta  P(\boldsymbol{\gamma}))d\boldsymbol{\gamma}$$
 and
 \begin{equation*}
 \begin{aligned}
 E&=\dsum_{1\leq q\leq A^{1/(6n)}}\dsum_{\substack{1\leq a\leq q\\(q,a)=1}}\dsum_{\substack{1\leq \z\leq q\\\z\in \Z^N}}e\biggl(\frac{a}{q}P(\z)\biggr)\dint_{|\beta|\leq \frac{A^{1/(6n)}}{qA^k}} (A/q)^{N-1}(\lambda^{-2}+|\beta|A^{k})d\beta \\
 &\leq \dsum_{1\leq q\leq A^{1/(6n)}}q^{N+1}\cdot\frac{A^{1/(6n)}}{qA^k}\cdot(A/q)^{N-1}\left(\lambda^{-2}+\frac{A^{1/(6n)}}{q}\right)\ll \lambda^{-2}A^{N-k-1+1/(2n)}.
 \end{aligned}   
 \end{equation*}
 Meanwhile, on recalling that 
$$w(\boldsymbol{\gamma},\x)=\varphi\left(\frac{f_{\boldsymbol{\gamma} }(\x)}{\lambda^2}\right)\cdot \prod_{i=1}^n\varphi\left(\frac{\partial_{x_i}f_{\boldsymbol{\gamma}}(\x)}{2\lambda N}\right),$$
 one finds by the triangle inequality and applying Lemma $\ref{lem3.7}$ that
$$I(\beta)\leq\int_{[-1,1]^N}|w(\boldsymbol{\gamma},\x)|d\boldsymbol{\gamma}\leq \dint_{[-1,1]^N}\varphi\left(\frac{f_{\boldsymbol{\gamma} }(\x)}{\lambda^2}\right)\cdot \prod_{i=1}^n\varphi\left(\frac{\partial_{x_i}f_{\boldsymbol{\gamma}}(\x)}{2\lambda N}\right)d\boldsymbol{\gamma}\ll\lambda^{n+1},$$
where, for the last inequality, we used the same treatment leading to the second expression in $(\ref{6.114}).$
 We next apply the same treatment as applied in the proofs of [$\ref{ref8}$, Lemmas 5.2 and 5.4] together with $(\ref{6.124})$ and $(\ref{6.126})$ using $q=1,a=0$ and with the Weyl type estimate over minor arcs for the exponential sum of the integrand in $T(\xi_0,\boldsymbol{\xi})$ leading from $(\ref{6.121})$ to $(\ref{6.122})$. Thus, we infer from our choice of $N$ that
 $$\int_{|\beta|\leq A^{1/(6n)}}|I(\beta)|^{1/(n+1)}d\beta\ll \int_{|\beta|\leq A^{1/(6n)}} \text{min}\left(1,|\beta|^{-\frac{N}{(n+1)(k-1)2^{k-1}}}\right)d\beta\ll 1$$
 and
$$\dsum_{1\leq q\leq A^{1/(6n)}}\dsum_{\substack{1\leq a\leq q\\(q,a)=1}}\biggl|\dsum_{\substack{1\leq \z\leq q\\\z\in \Z^N}}q^{-N}e\left(\frac{a}{q}P(\z)\right)\biggr|\ll 1.$$
Therefore, by making use of these bounds, we deduce from $(\ref{6.127})$ together with the triangle inequality that
\begin{equation}\label{6.128}
    \begin{aligned}
        \Xi_1&\ll A^{N-k}\dsum_{1\leq q\leq A^{1/(6n)}}\dsum_{\substack{1\leq a\leq q\\(q,a)=1}}\biggl|\dsum_{\substack{1\leq \z\leq q\\\z\in \Z^N}}q^{-N}e\left(\frac{a}{q}P(\z)\right)\biggr|\dint_{|\beta|\leq \frac{A^{1/(6n)}}{q}}|I(\beta)|d\beta+E\\
        &\ll A^{N-k}\dint_{|\beta|\leq A^{1/(6n)}}|I(\beta)|^{n/(n+1)}|I(\beta)|^{1/(n+1)}d\beta+E\\
        &\ll \lambda^{n}A^{N-k}\int_{|\beta|\leq A^{1/(6n)}}|I(\beta)|^{1/(n+1)}d\beta+E\ll \lambda^n A^{N-k},
    \end{aligned}
\end{equation}
where we have used 
$$E\ll \lambda^{-2}A^{N-k-1+1/(2n)}\ll \lambda^n A^{N-k},$$
obtained from $A^{-1/(2n)}< \lambda$ and $n\geq 4.$

Therefore, on substituting $(\ref{6.122})$ and $(\ref{6.128})$ into $(\ref{6.120})$ and from that into $(\ref{6.111})$, whenever $A^{-1/(2n)}<\lambda$, we conclude that
\begin{equation}\label{6.129}
  \text{mes}\left(B_N^{(\lambda)}\right)\ll \lambda A^{-k}.  
\end{equation}
Hence, combining ($\ref{6.112}$), $(\ref{6.119})$ and $(\ref{6.129}),$ we complete the proof of Lemma $\ref{lem6.14}.$
\end{proof}

\bigskip

The following lemma provides a lower bound for the quantity $\tau(\a;\gamma)$ for $\a\in B_N^{(\lambda)}.$ Furthermore, in the proof of Proposition $\ref{prop6.11}$, this lemma plays similar a role with that of Lemma $\ref{lem6.9}$ in the proof of Proposition $\ref{prop6.1}.$
\begin{lem}\label{lem6.1414}
    Let $d\geq 2$ and $n\geq 4$. Let also $\gamma>0.$ For $\lambda\in (0,M_{d,n})$ and $\a\in B_N^{(\lambda)}$, we have
    $$\tau(\a;\gamma)\gg \lambda ^{n}\cdot \min\left\{\gamma, \frac{1}{\lambda^2}\right\}.$$
\end{lem}
\begin{proof}
From $[\ref{ref3}, \text{Lemma 5.10}]$, one readily infers that whenever $$\a\in \mathcal{L}(\lambda)=\left\{\a\in B_N(N)\middle|\  \exists\x\in \mathbb{S}^{n-1}\ \begin{aligned}
   &f_{\a}(\x)=0\\
   & \lambda\|\a\|<\|\nabla f_{\a}(\x)\|\leq 2\lambda \|\a\|
\end{aligned}\right\},$$
with $\gamma>0$ and $\lambda\in (0,M_{d,n})$, we have $\tau(\a;\gamma)\gg \lambda^{n}\cdot \text{min}\left\{\gamma, \frac{1}{\lambda^2}\right\}.$ Since we obviously have $B_N^{(\lambda)}\subseteq \mathcal{L}(\lambda)$, we complete the proof of Lemma $\ref{lem6.1414}.$
\end{proof}

\begin{proof}[Proof of Proposition $\ref{prop6.11}$]
    As we discussed after the statement of Proposition $\ref{prop6.11},$ it suffices to confirm the inequality $(\ref{6.989898})$. For simplicity, we write
    $$\mathfrak{Y}(A)= {A^{-N+k}}\cdot\#\left\{\a\in \mathcal{A}^{\text{loc}}_{d,n}(A;P)\middle|\ \begin{aligned}
      C\cdot \tau(\a;2w^{5}N)\leq(\log A)^{-\eta}
    \end{aligned}\right\}.$$
    Recall the definition $(\ref{6.99})$ of the set $\mathbb{I}_{d,n}^{\text{loc}}$. Then, we see that
$$\mathfrak{Y}(A)\ll {A^{-N+k}}\cdot\#\left\{\a\in \Z^N\cap [-A,A]^N\cap \mathbb{I}_{d,n}^{\text{loc}}\middle|\ \begin{aligned}
      C\cdot \tau(\a;2w^{5}N)&\leq(\log A)^{-\eta}\\
      P(\a&)=0
    \end{aligned}\right\}.$$
    Then, it follows from Lemma $\ref{lem6.13}$ that
    \begin{equation}\label{6.130}
       \mathfrak{Y}(A)\ll {A^{-N+k}}\cdot\#\left\{\a\in \Z^N\cap [-A,A]^N\middle|\ \begin{aligned}
      \mathcal{N}(\a)\subset \mathbb{I}_{d,n}^{\text{loc}}&,\ P(\a)=0\\
      C\cdot \tau(\a;2w^{5}N)&\leq(\log A)^{-\eta}
    \end{aligned}\right\}+A^{-1/2}. 
    \end{equation}

We now show that if $\a\in \R^N$ satisfies $\|\a\|\geq 8w^{5}N$ then for any $\y\in \mathcal{N}(\a)$, we have
\begin{equation}\label{6.131}
    \tau(\y;4w^{5}N)\leq 2\cdot \tau(\a;2w^{5}N).
\end{equation}
Let $\u\in B_n(1)$ be such that 
$$|\langle\nu_{d,n}(\u),\y\rangle|\leq \frac{\|\nu_{d,n}(\u)\|\cdot\|\y\|}{4w^{5}N}.$$
Since $\y\in \mathcal{N}(\a)$, it follows from the Cauchy-Schwarz inequality that $|\langle\nu_{d,n}(\u),\y-\a\rangle|\leq \|\nu_{d,n}(\u)\|.$ Hence, we find that
$$|\langle \nu_{d,n}(\u),\a\rangle|\leq \frac{\|\nu_{d,n}(\u)\|\cdot\|\y\|}{4w^{5}N}+\|\nu_{d,n}(\u)\|.$$
Since we have assumed that $\|\a\|\geq 8w^5N,$ we have in particular $\|\y\|\leq \frac{3}{2}\|\a\|.$ We deduce that 
$$|\langle \nu_{d,n}(\u),\a\rangle|\leq \frac{3\|\nu_{d,n}(\u)\|\cdot\|\a\|}{8w^{5}N}+\|\nu_{d,n}(\u)\|.$$ Then, our assumption $\|\a\|\geq 8w^{5}N$ gives
$$|\langle \nu_{d,n}(\u),\a\rangle|\leq \frac{\|\nu_{d,n}(\u)\|\cdot\|\a\|}{2w^{5}N},$$
which establishes ($\ref{6.131}$), from the definition of $\tau(\a;b).$

On recalling the bound $(\ref{6.130})$, we note that the inequality $(\ref{6.131})$ gives 
\begin{equation}\label{6.132}
    \begin{aligned}
           &\mathfrak{Y}(A)\\
           &\ll {A^{-N+k}}\cdot\#\left\{\a\in \Z^N\cap ([-A,A]^N\setminus B_N(8w^{5}N))\middle|\ \begin{aligned}
      \mathcal{N}(\a)\subset \mathbb{I}_{d,n}^{\text{loc}}&,\ P(\a)=0\\
      C\cdot \tau(\a;2w^{5}N)&\leq(\log A)^{-\eta}
    \end{aligned}\right\}+A^{-1/2}\\
    &\ll {A^{-N+k}}\dsum_{\substack{\a\in \Z^N\cap [-A,A]^N\\ P(\a)=0}}\text{mes}\left(\left\{\y\in \mathcal{N}(\a)\cap \mathbb{I}_{d,n}^{\text{loc}}\middle|\ C\cdot\tau(\y;4w^{5}N)\leq 2(\log A)^{-\eta}\right\}\right)+A^{-1/2}.
    \end{aligned}
\end{equation}
Inverting the summation over $\a$ and the integration over $\y$ and recalling the definition of $C_P(A),$ the first term in the last display is bounded above by
\begin{equation}\label{6.133}
    \begin{aligned}
      &{A^{-N+k}}\cdot\text{mes}\left(\left\{\y\in B_N(NA)\cap \mathbb{I}_{d,n}^{\text{loc}}\middle|\ \begin{aligned}
       &C\cdot\tau(\y;4w^{5}N)\leq 2(\log A)^{-\eta} \\
       &\exists \b\in [-A,A]^N\cap \Z^N\ \text{s.t}\ \left\{\begin{aligned}
        &\|\y-\b\|\leq 1  \\
&P(\b)=0
\end{aligned}\right.\end{aligned}\right\}\right)\\
&\ll A^k\cdot\text{mes}\left(\left\{\y\in C_P(A)\cap  \mathbb{I}_{d,n}^{\text{loc}}\middle|\ C\cdot\tau(\y;4w^{5}N)\leq 2(\log A)^{-\eta} \right\}\right).
    \end{aligned}
\end{equation}

We now use the same trick as in the proof of Proposition $\ref{prop6.1}$. On substituting $(\ref{6.133})$ into ($\ref{6.132}$), we deduce that for $\kappa>0$
\begin{equation}\label{6.134}
    \mathfrak{Y}(A)\ll \frac{A^k}{(\log A)^{\eta\cdot \kappa}}\int_{C_P(A)\cap \mathbb{I}_{d,n}^{\text{loc}}}\frac{1}{\tau(\y;4w^{5}N)^{\kappa}}d\y+A^{-1/2}.
\end{equation}

Recall the definition ($\ref{6.108}$) and $(\ref{6.109})$ of $B_N^{(\lambda)}$ for $\lambda>0$ and the quantity $M_{d,n}$. We deduce from the equality $(\ref{6.110})$ that
$$\int_{C_P(A)\cap \mathbb{I}_{d,n}^{\text{loc}}}\frac{1}{\tau(\y;4w^{5}N)^{\kappa}}d\y\leq \sum_{l=1}^{\infty}\int_{B_N^{(M_{d,n}/2^l)}}\frac{1}{\tau(\y;4w^{5}N)^{\kappa}}d\y. $$
Let us set $\kappa=1/(2n)$. It follows by applying Lemma $\ref{lem6.14}$ and Lemma $\ref{lem6.1414}$ that 
\begin{equation*}
\begin{aligned}
    \int_{C_P(A)\cap \mathbb{I}_{d,n}^{\text{loc}}}\frac{1}{\tau(\y;4w^{5}N)^{\kappa}}d\y&\ll \dsum_{l=1}^{\infty}\left(\frac{2^{l\cdot n}}{w^{5}}+2^{l\cdot(n-2)}\right)^{\kappa}\cdot \text{mes}\left(B_N^{(M_{d,n}/2^l)}\right)\\
    &\ll A^{-k}\cdot\left(\frac{1}{w^{5/(2n)}}\dsum_{l=1}^{\infty}2^{-l/2}+\dsum_{l=1}^{\infty}2^{-l/2-l/n}\right).   
\end{aligned}
\end{equation*}
On substituting this into $(\ref{6.134})$, we conclude that
$$ \mathfrak{Y}(A)\ll \frac{1}{(\log A)^{\eta/(2n)}}+A^{-1/2}.$$
This completes the proof of Proposition $\ref{prop6.11}.$
\end{proof}

\section{Proof of Theorem 1.2}\label{sec7}
\begin{proof}
We observe that
\begin{equation*}
\begin{aligned}
    & {A^{-N+k}}\cdot\#\left\{\a\in \mathcal{A}^{\text{loc}}_{d,n}(A;P)\middle|\ \begin{aligned}
        \mathfrak{S}^*_{\a}\mathfrak{J}_{\a}^*\leq X^{n-d}A^{-1}(\log A)^{-\eta}
    \end{aligned}\right\}\\
    &\leq    {A^{-N+k}}\cdot\#\left\{\a\in \mathcal{A}^{\text{loc}}_{d,n}(A;P)\middle|\ \begin{aligned}
        \text{min}(\mathfrak{S}^*_{\a}, AX^{d-n}\mathfrak{J}_{\a}^*)\leq (\log A)^{-\eta/2}
    \end{aligned}\right\}.
\end{aligned}
\end{equation*}
Then, on noting that the cardinality of the set in the last expression is bounded above by
\begin{equation*}
\begin{aligned}
 \#\left\{\a\in \mathcal{A}^{\text{loc}}_{d,n}(A;P)\middle|\ \begin{aligned}
        \mathfrak{S}^*_{\a}\leq (\log A)^{-\eta/2}
    \end{aligned}\right\}    + \#\left\{\a\in \mathcal{A}^{\text{loc}}_{d,n}(A;P)\middle|\ \begin{aligned}
         AX^{d-n}\mathfrak{J}_{\a}^*\leq (\log A)^{-\eta/2}
    \end{aligned}\right\}       ,
\end{aligned}
\end{equation*}
it follows by applying Proposition $\ref{prop6.1}$ and Proposition $\ref{prop6.11}$ that
\begin{equation*}
    {A^{-N+k}}\cdot\#\left\{\a\in \mathcal{A}^{\text{loc}}_{d,n}(A;P)\middle|\ \begin{aligned}
        \mathfrak{S}^*_{\a}\mathfrak{J}_{\a}^*\leq X^{n-d}A^{-1}(\log A)^{-\eta}
    \end{aligned}\right\}\ll (\log A)^{-\eta/(40n)}.
\end{equation*}
\end{proof}

\end{document}